\documentclass[10pt]{article}

\usepackage{lmodern}          %
\usepackage[T1]{fontenc}
\usepackage[utf8]{inputenc}
\usepackage[letterpaper,margin=0.9in]{geometry}

\usepackage{amsmath}
\usepackage{amssymb}
\usepackage{amsthm}

\usepackage{tikz}
\usetikzlibrary{arrows.meta,calc}

\usepackage{setspace}
\usepackage{microtype}
\usepackage[round,authoryear]{natbib}

\usepackage[colorlinks=true,allcolors=blue]{hyperref}

\theoremstyle{plain}
\newtheorem{theorem}{Theorem}[section]
\newtheorem{proposition}[theorem]{Proposition}
\newtheorem{lemma}[theorem]{Lemma}
\newtheorem{corollary}[theorem]{Corollary}

\theoremstyle{definition}
\newtheorem{definition}[theorem]{Definition}
\newtheorem{assumption}[theorem]{Assumption}

\theoremstyle{remark}
\newtheorem{remark}[theorem]{Remark}

\begin{document}

\title{Forward recursive aggregator systems and the forward
Epstein--Zin recursive paradigm}

\author{%
  Zakaria Bensaid\,$^{*}$ \and Anis Matoussi\,$^{\dagger}$ \and
  Thaleia Zariphopoulou\,$^{\ddagger}$%
}

\date{}

\maketitle

\begingroup
\renewcommand{\thefootnote}{}
\footnotetext{The research of the $^{*}$ and $^{\dagger}$ authors was supported by the
Agence Nationale de la Recherche through the project DREAMeS
(ANR-21-CE46-0002) and by the France 2030 programme through the MIRTE
project (ANR-23-EXMA-0011). Their research was also supported by the Chair
``Risques \'{E}mergents en Assurance'' (RE2A), the Chair ``Impact de la
Transition Climatique en Assurance'' (ITCA), and GRESC ``Gestion des risques
ESG pour le cr\'{e}dit'', under the aegis of the Fondation du Risque, in
partnership with the Risk and Insurance Institute of Le Mans,
MMA-Cov\'{e}a, Groupama, and Natixis, respectively. The funding bodies had no
role in the design of the study, the analysis, the writing of the manuscript,
or the decision to submit the article.}
\endgroup

\begin{center}
\small
$^{*,\dagger}$Le Mans Universit\'{e}, LMM, IRA, F-72000 Le Mans, France\\[0.3em]
$^{\ddagger}$Departments of Mathematics and IROM, The University of Texas at Austin,
Austin, Texas 78712, U.S.A.\\
and Oxford--Man Institute, University of Oxford, Oxford, OX2 6ED, U.K.\\[0.5em]
$^{*}$\texttt{zakaria.bensaid@univ-lemans.fr} \quad
$^{\dagger}$\texttt{anis.matoussi@univ-lemans.fr} \quad
$^{\ddagger}$\texttt{zariphop@math.utexas.edu}
\end{center}

\begin{abstract}
\noindent We introduce forward recursive aggregator systems in
It\^{o}-diffusion markets, bridging the theories of recursive utilities and of forward
performance criteria. The resulting
criteria are constructed forward in time and accommodate arbitrary or rolling horizons. We derive the associated HJB SPDE and sufficient
conditions for consistency and admissibility. We also develop a convex
duality theory based on state price densities and identify the density
generated by the optimal wealth process. For homothetic aggregators, we characterize a
class of forward Epstein--Zin recursive utility processes, derive the optimal
investment and consumption policies, and obtain explicit solutions in the
Black and Scholes and Heston models.
\end{abstract}

\medskip


\medskip


\medskip


\section{Introduction}\label{sec:1}

The paper contributes to the literature of recursive utilities and to
optimal portfolio and consumption models under such stochastic preferences.
We introduce and develop a novel class of recursive utility processes, the \textit{forward}
\textit{recursive utilities} (or forward recursive performance criteria),
which incorporate progressively measurable aggregators,
arbitrary or rolling horizons, and progressively measurable adaptation of market
model dynamics. Thus, our work bridges the areas of recursive preferences
and forward performance processes, and proposes an extended class of
utility processes which contain the existing ones as special cases. In
parallel, the characterization, construction and specification of forward
recursive utilities give rise to new stochastic optimization
problems which are interesting in their own right.

Recursive utilities occupy a central place in asset pricing and in optimal
consumption and portfolio choice. In deterministic settings, they first
appeared in the seminal works of \citet{Koopmans-1960} and
\citet{Lucas-Stokey-1984} in discrete-time models, while their
continuous-time analogue was developed by \citet{Epstein-1987}.
In stochastic environments, recursive utilities of separable kind were
proposed by \citet{Epstein-1983} for discrete time and by
\citet{Uzawa-1968} in continuous time. The general class of non-separable
recursive utilities were subsequently developed by \citet{Epstein-Zin-1989} in discrete-time and \citet{Duffie-Epstein} in continuous models. Since then, recursive utilities have
been developed and incorporated in many areas, including optimal investment and
consumption choice, asset-price equilibrium, optimal growth, insurance, and
pensions. Providing a comprehensive list of references lies
outside the scope of this paper.

To motivate the forward recursive construction, we first briefly review
some key elements and structural
characteristics of their existing classical counterparts. Since we will be working in
It\^{o}-diffusion markets, we only recall results for such processes (we
refer the reader to \citep{Duffie-Epstein} for further details). To this end,
for a given horizon $\left[ 0,T\right] ,$ $T\leq \infty ,$ and a probability
space $\left( \Omega ,\mathcal{F},\mathbb{P}\right) $ endowed with a
Brownian filtration $(\mathcal{F}_{t})_{t\in \left[ 0,T%
\right] },$ a recursive utility is an $\mathcal{F}_{t}-$adapted process,
denoted by $V_{t,T},$ $t\in \left[ 0,T\right] ,$ that measures the
performance of admissible consumption processes $c_{t}$. It is constructed
via an \textit{aggregator pair} $\left( f,m\right) ,$ with $m$ resembling a
certainty equivalent rule and $f$ acting upon the upcoming consumption
stream and its future utility value. Intertemporal consistency requirements
yield the relation 
\begin{equation}
\left. \frac{d}{ds}\,m\! \left( \mathrm{Law}(V_{t+s,T}\,|\, \mathcal{F}%
_{t})\right) \right \vert _{s=0}=-f\left( c_{t},V_{t,T}\right) ,\text{ \ }%
t\in \left[ 0,T\right] .  \label{m-f}
\end{equation}%
When $m$ is smooth at certainty, \citet[p.~361]{Duffie-Epstein}
conjecture that $V_{t,T}$ has a stochastic differential representation of
the form (taking, for exposition purposes, the Brownian motion $W$
to be one-dimensional), 
\begin{equation*}
dV_{t,T}=\mu _{t}dt+\sigma _{t}dW_{t},\text{ \ }t\in \left[ 0,T\right] ,%
\text{ \  \  \ }V_{T,T}=0,
\end{equation*}%
for suitable drift and volatility processes $\mu $ and $\sigma $ satisfying
the compatibility condition 
\begin{equation}
\mu _{t}=-f\left( c_{t},V_{t,T}\right) -\frac{1}{2}A^{\mathrm{DE}}\left( V_{t,T}\right)
\sigma _{t}^{2},\text{ \ }t\in \left[ 0,T\right] ,  \label{miu-intro}
\end{equation}%
with $A^{\mathrm{DE}}(\cdot )$, being non-positive and known as the \textit{%
variance multiplier} of Duffie and Epstein. If such a process $V_{t,T}$ exists, it can be represented as%
\begin{equation}
V_{t,T}=\mathbb{E}\left[ \left. \int_{t}^{T}\left( f\left(
c_{s},V_{s,T}\right) +\frac{1}{2}A^{\mathrm{DE}}\left( V_{s,T}\right) \frac{d}{ds}\left[ V%
\right] _{s,T}\right) ds\right \vert \mathcal{F}_{t}\right] ,\text{ }t\in %
\left[ 0,T\right] .  \label{V-t-T-intro}
\end{equation}

The above may be simplified, working with the so-called \textit{normalized }%
version. Specifically, let $\left( f,m\right) $ and $\left( \bar{f},\bar{m}%
\right) $ be two aggregator pairs such that, for a given smooth function $%
\varphi ,$%
\begin{equation*}
f\left( c,u\right) =\frac{\bar{f}\left[ c,\varphi \left( u\right) \right] }{%
\varphi ^{\prime }\left( u\right) }\text{ \  \  \ and \  \  \ }m\left( \mathrm{%
Law}(\xi )\right) =\varphi ^{-1}\! \left( \bar{m}\left[ \mathrm{Law}(\varphi
(\xi ))\right] \right) ,
\end{equation*}%
where $\xi $ is a random variable. Then, the normalized (but ordinally
equivalent) version of $V_{t,T},$ denoted by $\bar{V}_{t,T},$ is given by
\begin{equation}
\bar{V}_{t,T}=\mathbb{E}\left[ \left. \int_{t}^{T}\bar{f}\left( c_{s},\bar{V}%
_{s,T}\right) ds\right \vert \mathcal{F}_{t}\right] ,\text{\  \ }t\in \left[
0,T\right] .  \label{V-final-inro}
\end{equation}

Formulations \eqref{V-t-T-intro} and \eqref{V-final-inro}\ (original or
normalized) have generated a wide range of interesting problems about the
existence, uniqueness, smoothness, time-consistency and further
characterization (concavity, monotonicity and others) of recursive utility
processes. Furthermore, it turned out that there exists a deep connection
between recursive utilities and backward stochastic differential equations
(BSDE), which prompted the considerable development of many elegant works in
this area of stochastic analysis.

The normalized formulation has important tractability and representation
advantages. For the remainder of this introductory review, we work with the
normalized pair $\left( \bar{f},\bar{m}\right) $ but, to ease the presentation, omit the bars, writing
$(f,m)$. In section~\ref{sec:2}, we first introduce the general forward construction and then the corresponding normalized analogue.

Recursive utilities were subsequently incorporated in stochastic optimization
problems of controlled diffusion processes,
\begin{equation*}
dX_{s}=b(s,X_{s},c_{s})ds+\sigma (s,X_{s},c_{s})dW_{s},\text{ \  \ }s\in %
\left[ t,T\right] ,\text{ \  \ }X_{t}=x>0,
\end{equation*}%
with the related value function process defined as
\begin{equation*}
U_{t,T}=\operatorname*{ess\,sup}_{c}\mathbb{E}\left[ \left. \int_{t}^{T}f\left(
c_{s},V^{c}_{s,T}\right) ds+\xi _{T}\right \vert \mathcal{F}_{t}\right] ,%
\text{ \  \ }t\in \left[ 0,T\right] ,
\end{equation*}%
for a pre-specified aggregator pair $(f,m)$ and a terminal condition $\xi
_{T}.$ Here, $V^{c}_{s,T}$ denotes the recursive utility associated with
$c$, and the essential supremum is over admissible continuations from
the current state $X_t$.

In Markovian environments, with deterministic coefficients and terminal
condition $\xi_T=\xi(X_T)$, we have $U_{t,T}=u(t,X_t)$ almost surely, with $u$ solving
the Hamilton-Jacobi-Bellman (HJB) equation
\begin{equation}
u_{t}(t,x)+\sup_{c}\left( \frac{1}{2}\sigma
^{2}(t,x,c)u_{xx}(t,x)+b(t,x,c)u_{x}(t,x)+f(c,u(t,x))\right) =0,
\text{ \ }t\in \left[ 0,T\right) ,\text{ }x>0,
\text{ \ }u(T,x)=0.
\label{HJB1-intro}
\end{equation}%
The derivatives $u_t$ and $u_x$ denote the first-order partial derivatives
with respect to time and the state variable, respectively, and $u_{xx}$
denotes the second-order partial derivative with respect to the state variable. Under suitable assumptions, the
equation holds in the viscosity sense, and classically when $u$ is
sufficiently smooth.
It is assumed for simplicity that $\xi _{T}=0,$ but the case of non-zero
terminal condition can be incorporated through $u(T,x)=\xi(x)$.
The PDE approach to stochastic differential utility was developed by
\citet{Duffie-Lions}. Similar optimization criteria were subsequently
developed in non-Markovian market environments and the related problems
were studied using BSDE and FBSDE.

More broadly, stochastic recursive utilities have been extensively
incorporated and widely studied in models that allow for both intertemporal
investment and consumption, with an underlying controlled process $X$
satisfying the well known SDE
\begin{equation}
dX_{s}^{\pi ,c}=b_{s}\pi _{s}ds-c_{s}ds+\sigma _{s}\pi _{s}dW_{s},\text{ \ }%
0 \leq t \leq s \leq T,\text{ \ }X_{t}=x>0,\text{\ }  \label{wealth-intro}
\end{equation}%
with control processes $\left( \pi ,c\right) $ satisfying the self-financing
budget equation and certain integrability properties. Here, $\pi_s$ denotes
the amount invested in the risky asset. The coefficients $\left(
b,\sigma \right) $ represent the return drift and volatility of the traded
stock (under the simplified assumption of zero interest rate). The value
function process is then defined as
\begin{equation}
U_{t,T}=\operatorname*{ess\,sup}_{\left( \pi ,c\right) }\mathbb{E}\left[ \left.
\int_{t}^{T}f\left( c_{s},V^{\pi,c}_{s,T}\right) ds+\xi _{T}\right \vert \mathcal{F}%
_{t}\right] ,\text{ \  \ }t\in \left[ 0,T\right] ,
\label{valuefunction-intro}
\end{equation}%
where $V^{\pi,c}_{s,T}$ denotes the recursive utility associated with
the admissible control pair $(\pi,c)$.

In the Markovian case, a HJB equation similar to \eqref{HJB1-intro}\ arises. For
example, if the stock price process is log-normal, with constant coefficients
$(b,\sigma )$ and terminal condition $\xi_T=\xi(X_T)$, we have
$U_{t,T}=u(t,X_t)$ almost surely, where the related HJB\ takes the form
\begin{equation}
\begin{aligned}
&u_{t}(t,x)+\sup_{(\pi,c)}\left( \frac{1}{2}\sigma ^{2}\pi ^{2}u_{xx}(t,x)
+b\pi u_{x}(t,x)-cu_{x}(t,x)+f(c,u(t,x))\right) =0,
\text{ \ }t\in \left[ 0,T\right) ,\text{ }x>0,\\
&u(T,x)=\xi \left( x\right) .
\end{aligned}
\label{HJB2-intro}
\end{equation}%

In general, equations of type \eqref{HJB1-intro} and \eqref{HJB2-intro} are
non-tractable, are fully degenerate, may not have smooth solutions and are
typically analyzed in the weak (viscosity) sense. Closed-form solutions are
obtained only for various homothetic classes, which we revisit in detail
later on.

While recursive utilities play a central role in asset pricing and their
study has and, still, continues to give rise to very interesting
mathematical problems at the intersection of several areas, there are some
modeling elements that may impede their universal applicability. We discuss
these modeling issues next and propose an extended framework to
accommodate some of the related limitations. We choose to do this by working
directly in the context of optimal portfolio and consumption problems in It%
\^{o}-diffusion markets, in \eqref{wealth-intro}\ and \eqref{valuefunction-intro}\ above, and leave the original foundational
construction, as in \eqref{m-f}, \eqref{miu-intro} and \eqref{valuefunction-intro}, for future work.

Inspecting the recursive stochastic optimization model \eqref{valuefunction-intro}, we see that there are three modeling inputs, chosen
at initiation time $t=0$. The first is the aggregator pair $\left(
f,m\right) ,$ the foundational building block of the recursive utilities.
Indeed, this pair is chosen once and for all at initial time, without any
flexibility to change it afterwards. As a result, any evolution of
preferences depends heavily on these initially specified elements. However,
it is often the case that preferences do change over time, and frequently,
after, for example, big market movements, realized losses and relative
performance outcomes.

The second modeling element is the pre-determined investment horizon $\left[
0,T\right] ,$ which is strongly embedded throughout in the recursive
construction. On the other hand, it is also often the case that investment
horizons may not be a priori known and may be adapted to incoming
information. This is particularly crucial if random endowments are
incorporated and their arrival times are not a priori known. The third
modeling element is the choice of the market model, denoted by $\mathcal{M}_{%
\left[ 0,T\right] },$ since the choice of the processes $\left( b,\sigma \right) $ in \eqref{wealth-intro} is made at initial time as well. However, such initial
choice may turn out to be erroneous as more information is acquired in
upcoming times. We add that one may try to remedy the limitation of
pre-specified horizon by working in an infinite horizon setting. This
choice, however, makes the pre-choice of the market model even more
restrictive.

In summary, one may think of the modeling triplet $\mathcal{T=}\left( \left(
f,m\right) ,\left[ 0,T\right] ,\mathcal{M}_{\left[ 0,T\right] }\right) $ as
a static choice at $t=0,$ which results in various restrictions on the
flexibility of the model to incorporate upcoming preferences, horizon and
market model changes. Such issues motivated the last author and Musiela to
introduce in the early 2000s the concept of \textit{forward performance
criteria }(or forward utilities). These are processes defined for all times
and constructed via supermartingale and martingale requirements along
admissible and optimal control processes, respectively. In many ways, they resemble the
classical value function process which inherently obeys the dynamic
programming principle (DPP). However, DPP\ yields optimality backwards in
time and, thus, the pre-specification of triplet $\mathcal{T}$ is
unavoidable. In contrast, forward utilities are defined forward in time,
which allows the elements of $\mathcal{T}$ to be updated as information arrives.

We continue with some key structural and theoretical properties of forward
utility processes, stating them informally to ease the presentation. In It%
\^{o}-diffusion environments, a forward utility process, $U(t,x),$ is
defined such that along an arbitrary admissible policy, say $\alpha ,$ generating
process $X^{\alpha },$ $U(t,X_{t}^{\alpha })$ is a (local) supermartingale
while, along an optimal one, $\alpha ^{\star },$ $U(t,X_{t}^{\alpha ^{\star
}})$ becomes a (local) martingale. Assuming the decomposition
\begin{equation*}
dU(t,x)=b^{U}(t,x)dt+\delta (t,x)^{\top }dW_{t},\text{ }t\geq 0,\text{ }x>0,
\end{equation*}%
for suitable drift and volatility processes $b^{U}(t,x)$ and $\delta (t,x),$
these two properties impose a certain structure on $b^{U}\left( t,x\right) $
for a given forward volatility process $\delta \left( t,x\right) $. We
stress that, contrary to the classical case where the volatility of the
value function process is part of the solution, in the forward setting it is
a modeling input. In other words, embedded in the definition of the forward
utility process is the choice of $\delta (t,x)$. The conditions resulting on 
$b^{U}(t,x)$ essentially lead to an ill-posed SPDE which plays the role of
the classical HJB\ equation \citep[see, e.g.,][]{Musiela-Z-2008}. To date,
general existence, uniqueness and regularity results for such SPDE\ are
lacking.

There are, however, three classes of forward utilities that have been
extensively studied. The first class is the one of time-monotone forward criteria ($\delta
(t,x)\equiv 0$), studied in \citet{musiela-Z-monotone} \citep[see, also, preprint][]{Rogers-Tehranchi}. The second class is in Markovian, stochastic
factor models in which $U(t,x)$ has a finite-dimensional representation, $%
U(t,x)=u(t,x,Y_{1,t},...,Y_{N,t}),$ for some "factor" processes $Y_{i,t},$ $%
i=1,...,N,$ with function $u$ solving an ill-posed, but finite-dimensional,
HJB\ equation. Within the homothetic family of forward utilities, further
reduction leads to tractable problems that have been analyzed using, among
others, the Martin boundary, ergodic BSDE, systems of FBSDE and others.
Thirdly, predictable forward utilities have been analyzed with functional
and non-local equations. Providing a complete bibliography of forward
utilities and related applications, among others in mean field games,
entropic risk measures, optimized certainty equivalent, indifference valuation, risk sharing, pension
fund management, equilibrium models and many others, lies outside the scope
of this paper. We refer the reader to the review article
\citep{Musiela-ReviewArticle} and to the recent volume collections
\citep{Liang-Volume1} and \citep{Liang-Volume2}.

For the definition of \textit{forward recursive utility processes}, we focus for now
on It\^{o}-diffusion market environments with controlled processes $X^{\pi
,c}$ as in \eqref{wealth-intro}. We combine elements from the definition of
forward utilities and structural characteristics of the classical recursive
utilities, and seek a \textit{pair }of random fields, denoted by $U\left(
t,x\right) $ and $F(t,c,u,z),$ together with a volatility process $\delta
(t,x),$ such that the process 
\begin{equation*}
U(t,X_{t}^{\pi ,c})+\int_{0}^{t}F(s,c_{s},U(s,X_{s}^{\pi ,c}),Z_{s}^{\pi
,c})ds,\text{ }t\geq 0,
\end{equation*}
is a (local) supermartingale along an admissible policy $\left( \pi
,c\right) $, where process $Z^{\pi ,c}$ is defined as 
\begin{equation*}
Z_{t}^{\pi ,c}=\delta \left( t,X_{t}^{\pi ,c}\right) +U_{x}(t,X_{t}^{\pi
,c}) X_{t}^{\pi ,c}\sigma _{t}^{\top }\pi _{t},\text{ }t\geq 0,
\end{equation*}%
and along an optimal policy $\left( \pi ^{\star },c^{\star }\right) ,$ the
process 
\begin{equation*}
U(t,X_{t}^{\pi ^{\star },c^{\star }})+\int_{0}^{t}F(s,c_{s}^{\star
},U(s,X_{s}^{\pi ^{\star },c^{\star }}),Z_{s}^{\pi ^{\star },c^{\star }})ds,%
\text{ }t\geq 0,
\end{equation*}%
is a (local) martingale. We will refer to the pair of processes 
$(U,F)$ (or $(U,F;\delta \left( t,x\right) ),$ when specific reference to $%
\delta \left( t,x\right) $ is needed) as a consistent \textit{forward
aggregator system, }with $U(t,x)$ called its \textit{forward recursive
utility process} and $F(t,c,u,z)$ the related \textit{forward
aggregator}. The definition of $Z^{\pi ,c}$ displays explicitly how the
model input $\delta \left( t,x\right) $ enters the forward recursive setting.

Working, in alignment with the normalized classical recursive framework,
with forward aggregators of the form 
\begin{equation*}
F(t,c,u,z)=f\left( t,c,u\right) +\frac{1}{2}A(u)|z|^{2},
\end{equation*}%
for a suitable pair $\left( f,A\right) $ with $A\leq 0$ continuous and $f$
independent of $z$, leads to the (normalized) forward recursive SPDE
\begin{equation}
dU(t,x)=\left( \frac{\left \vert U_{x}(t,x)\theta _{t}+\sigma_{t}^{+}\sigma_{t}\delta
_{x}(t,x)\right \vert ^{2}}{2U_{xx}(t,x)}-\widetilde{F}(t,U_{x}(t,x),U(t,x))%
\right) dt+\delta \left( t,x\right) ^{\top }dW_{t},\quad t\geq0,\ x>0,
\label{SPDE-intro}
\end{equation}%
where $\widetilde{F}$ is the Fenchel transform of $F$ with respect to the consumption and process $\theta $ is the market price of risk.

Here $\sigma _{t}^{+}\sigma _{t}$ denotes the orthogonal projection of $%
\mathbb{R}^{d}$ onto the traded subspace $\mathrm{Im}\, \sigma _{t}^{\top }.$
It enters only through the recursive volatility, which may not lie in this
subspace. The market price of risk is left unchanged, $\theta _{t}=\sigma
_{t}^{+}\sigma _{t}\,\theta _{t},$ $t\geq 0.$

If SPDE~\eqref{SPDE-intro} admits a sufficiently regular solution, the
optimal feedback maps are identified in Theorem~\ref{thm:2-10} and derived in its
proof.

\bigskip

\noindent We also provide a detailed study in the associated dual domain. We associate with the market its family of state
price densities and show that the convex conjugate random field
$\widetilde{U}$ satisfies the Legendre counterpart of the primal Bellman
SPDE. This conjugate equation yields a submartingale inequality for every
state price density together with a dual control for which the dual
criterion becomes a local martingale. Contrary to the time-additive forward utilities studied in \citet{EKHM18} the marginal utility
along the optimal wealth is not itself a state price density, in the recursive setting. A state price density is obtained after
multiplication by a specific discount factor $\kappa ^{\star }$,
namely $Y^{\star }=\kappa ^{\star }U_{x}(\cdot ,X^{\star })$. When $F$ is
convex in its continuation-utility argument, a variational transform in the
spirit of \citet{MX17} yields a global Fenchel bound, which bound we show is attained at the primal optimum and at this same density. 

Finally, we address the passage from candidate optimal feedback policies to
admissible ones. \citet{EKM13} connect the utility SPDE
with two solvable SDEs, and \citet{EKHM18}
develop the corresponding investment and consumption framework. Following
this route, we show that explicit bounds on the model inputs render the
deflated dual SDE globally well-posed, and that the optimal wealth is then
constructed from the inverse marginal field, the two solution families being linked by
the identity $U_{x}(t,X_{t}^{\star }(x))=D_{t}^{\star }(U_{x}(0,x))$.
These results are presented in Theorem~\ref{thm:3-14}, where consistency,
admissibility and the primal--dual relations are also established
under suitable conditions on the model inputs.

Having defined forward recursive utilities and aggregators and examined the
primal and dual problems, we turn to representative cases. Given the popularity of
Epstein--Zin recursive utilities in the classical framework, we focus on
their forward analogue. In particular, we examine forward recursive
aggregator systems of the separable and multiplicative form 
\begin{equation*}
U(t,x)=\Phi _{t}u(x)\quad \text{and}\quad F(t,c,v)=\Psi _{t}f(c,v),\text{ \ }%
t\geq 0,\text{ }x>0,\text{ }c>0,
\end{equation*}%
where
\begin{equation*}
u(x)=\dfrac{x^{1-\gamma }}{1-\gamma },\text{ \ }x>0,
\end{equation*}%
and%
\begin{equation*}
f(c,v)=\frac{1}{1-\frac{1}{\eta }}\! \left( c^{\,1-\frac{1}{\eta }}\left(
(1-\gamma )v\right) ^{1-\frac{1}{\lambda }}-(1-\gamma )v\right) ,\text{ \ }%
c>0,\text{ }(1-\gamma )v>0,
\end{equation*}%
parametrized by constants $\gamma ,\eta $ and $\lambda =(1-\gamma
)/(1-1/\eta )$, and suitable processes $\Phi $ and $\Psi .$ We provide
various characterization results. In particular, we specify $\Phi $ as the
solution of a scalar Bernoulli-type SDE driven by $\Psi $ by looking at an
auxiliary linear SDE solved by the process $\Xi _{t}=\Phi _{t}^{\eta /\lambda }$%
. We also derive the optimal feedback control processes and the optimal
state price density in closed form. We further consider the Heston
stochastic volatility model and analyze it in detail. 

\bigskip

\noindent \textbf{Organization of the paper.} In section~\ref{sec:2}, we introduce the
It\^{o}-diffusion market and the new notion forward recursive aggregator systems. We
derive the unnormalized and normalized versions and the related forward
recursive SPDEs, and state the analogous verification theorems. In
section~\ref{sec:3}, we develop the duality theory for forward recursive utilities.
We study the state price densities and the wealth-conjugate SPDE. Furthermore, we provide a theorem that derives, from explicit bounds on the model inputs, the existence
and uniqueness of the dual and primal SDEs (optimal wealth process and state price density), the admissibility of the
feedback control pair, and the corresponding primal and dual consistency
properties. In section~\ref{sec:4}, we develop the forward analogue of the popular
homothetic Epstein--Zin class. For tractability, we mainly consider the
separable class and solve it explicitly, including constant-coefficient and
Heston stochastic-volatility models. To ease the presentation, we provide
the majority of the proofs in the three Appendices.

\section{Forward recursive aggregator systems and forward recursive utilities}\label{sec:2}

We consider a filtered probability space $\left( \Omega ,%
\mathcal{F},(\mathcal{F}_{t})_{t\geq 0},\mathbb{P}\right) $ supporting a
standard $d$-dimensional Brownian motion $W$, and we assume that the
filtration coincides with the one generated by $W$.

\medskip

The market consists of a riskless bond, taken to be the numeraire, and $n$
risky assets whose discounted prices $S=(S^{1},\ldots ,S^{n})$ are specified
by
\begin{equation} \label{eq:marketstocks}
dS_{t}=\mathrm{diag}(S_{t})\! \left( \mu _{t}\,dt+\sigma _{t}\,dW_{t}\right) ,%
\text{ }t\geq 0,\qquad S_{0}\in (\mathbb{R}_{+}^{\ast })^{n}.
\end{equation}%
The coefficients $\mu _{t}\in \mathbb{R}^{n}$ and $\sigma _{t}\in
\mathbb{R}^{n\times d}$ are $\mathcal{F}_t$-predictable processes and $\sigma _{t}$ has
full row rank $n\leq d$, for $dt\otimes d\mathbb P$-a.e.\ $(t,\omega )$. The market price of
risk is the $\mathbb{R}^{d}$%
-valued process
\begin{equation*}
\theta _{t}=\sigma _{t}^{+}\mu _{t},\text{ }t\geq 0,
\end{equation*}%
where $\sigma _{t}^{+}=\sigma _{t}^{\top }(\sigma _{t}\sigma _{t}^{\top
})^{-1}$ denotes the Moore--Penrose pseudo-inverse of $\sigma _{t}$. In
particular, it holds that $\sigma _{t}\theta _{t}=\mu _{t}$, and $\theta
_{t}\in \mathrm{Im}\, \sigma _{t}^{\top }$. We assume that 
\begin{equation}\label{eq:theta}
    \int_{0}^{T}|\theta _{t}|^{2}\,dt<\infty ~~ \mathbb P\text{-a.s. for each }T>0.
\end{equation}

\medskip \noindent The subspace $\mathrm{Im}\, \sigma _{t}^{\top }\subset
\mathbb{R}^{d}$ contains the directions spanned by the tradable assets. When it
is a strict subspace of $\mathbb{R}^{d}$, the market is incomplete. In what
follows, the matrix $\sigma _{t}^{+}\sigma _{t}$ is the orthogonal projection
of $\mathbb{R}^{d}$ onto $\mathrm{Im}\, \sigma
_{t}^{\top }$, and $I_{d}-\sigma _{t}^{+}\sigma _{t}$ the projection onto its
orthogonal complement $(\mathrm{Im}\, \sigma _{t}^{\top })^{\perp }$. Since
$\theta _{t}\in \mathrm{Im}\,\sigma _{t}^{\top }$, the two projections act on
the market price of risk as
\begin{equation}
\sigma _{t}^{+}\sigma _{t}\, \theta _{t}=\theta _{t}\text{ \  \  \ and \  \  \ }%
\left( I_{d}-\sigma _{t}^{+}\sigma _{t}\right) \theta _{t}=0,\text{ \ }t\geq
0.  \label{eq:projections_theta}
\end{equation}

\medskip

\noindent The wealth process $X^{\pi ,c}$ solves 
\begin{equation}
dX_{t}^{\pi ,c}=X_{t}^{\pi ,c}\, \pi _{t}^{\top }\sigma _{t}\! \left(
dW_{t}+\theta _{t}\,dt\right) -c_{t}\,dt,~t\geq 0, ~ X_0^{\pi, c} = x >0, \label{eq:wealth}
\end{equation}%
where the portfolio process $\pi $, valued in $\mathbb{R}^{n}$, denotes the
\emph{proportions} of wealth allocated to the risky assets, so that the amount
invested is $X_{t}^{\pi ,c}\pi _{t}$. The
consumption stream $c$ takes non-negative values. We also use
the \emph{relative consumption} process $\check{c}=c/X^{\pi ,c}$. All
processes are expressed in units discounted by the numeraire.

\begin{definition}\label{defn:2-1}
A control pair $(\pi ,c)$ is
\emph{admissible} if $\pi $ and $c$ are $\mathcal{F}_{t}$-progressively
measurable processes satisfying $c_t>0$, for $dt\otimes d\mathbb P$-a.e.\ $(t,\omega )$,
\begin{equation*}
\int_{0}^{T}\! \left( c_{t}+|\pi _{t}^{\top }\sigma _{t}|^{2}\right)
\,dt<\infty \quad\text{$\mathbb P$-a.s. for each }T>0,
\end{equation*}%
and the associated SDE \eqref{eq:wealth}\ has a unique strong solution $%
X^{\pi ,c}$ which is strictly positive. We denote by $\mathcal{A}$ the set of
admissible controls.
\end{definition}

\subsection{Forward recursive aggregator systems}\label{sec:2-1}

We introduce the new concepts of forward recursive preferences and their aggregators. The forward utility is represented by an It\^{o}
random field whose evolution is linked to a progressively measurable
aggregator through a martingale optimality principle. The regularity classes
for the random fields used below are recalled in Appendix~\ref{app:A}.

\begin{definition}\label{defn:2-2}
Let $\mathcal{U}\subset \mathbb{R}$ be a non-empty
open interval, and $\mathbb{D}_{U}=[0,\infty )\times \Omega \times (0,\infty
)$ and $\mathbb{D}_{F}^{\mathcal{U}}=[0,\infty )\times \Omega \times
(0,\infty )\times \mathcal{U}\times \mathbb{R}^{d}.$ A pair $\left(
U,F\right) $ of random fields
\begin{equation*}
U:\mathbb{D}_{U}\longrightarrow \mathcal{U}\text{ \  \  \ and \  \  \ }F:%
\mathbb{D}_{F}^{\mathcal{U}}\longrightarrow \mathbb{R},
\end{equation*}%
is called a \emph{forward aggregator system on $\mathcal{U}$} if:

\smallskip

i) for each $x>0$, the process $U(\cdot ,\cdot ,x)$ is $\mathcal{F}_{t}$%
-progressively measurable, and for each $(c,u,z)\in (0,\infty )\times
\mathcal{U}\times \mathbb{R}^{d}$, the process $F(\cdot ,\cdot ,c,u,z)$ is $%
\mathcal{F}_{t}$-progressively measurable,\smallskip

ii) $\mathbb{P}$-almost surely, for every $t\geq 0$, the map $%
x\mapsto U(t,\omega ,x)$ is strictly increasing and strictly
concave,\smallskip

iii) for $dt\otimes d\mathbb{P}$-a.e.\ $(t,\omega )$, the map $c\mapsto
F(t,\omega ,c,u,z)$ is strictly increasing and strictly concave for every $%
(u,z)\in \mathcal{U}\times \mathbb{R}^{d}$, and the map $(u,z)\mapsto
F(t,\omega ,c,u,z)$ is continuous on $\mathcal{U}\times \mathbb{R}^{d}$ for
each $c>0$.
\end{definition}

The properties of $U$ are stated in ii) for a single version of the
field on one $\mathbb{P}$-full set, while those of $F$ are stated in (iii) for $dt\otimes d\mathbb{P}$-a.e.\ $(t,\omega )$, since $F$ is used only inside time integrals.
Progressive measurability in $(t,\omega )$ together with continuity in the
spatial variables gives, by the Carath\'{e}odory measurability theorem, that
$F$ is jointly measurable in $(t,\omega ,c,u,z)$. We suppress the $\omega $-argument and write $U(t,x)$ and $%
F(t,c,u,z)$, except where the dependence on $\omega $ has to be displayed.

\begin{definition}\label{defn:2-3}
Let $(U,F)$ be a forward aggregator system on
$\mathcal U$ and let $U$ be a $\mathcal K_{\mathrm{loc}}^{2,\varepsilon }$%
-semimartingale random field, for some $\varepsilon \in (0,1]$, with local
characteristics $(b,\delta )$ in the sense that
\begin{equation}
dU(t,x)=b(t,x)\,dt+\delta (t,x)^{\top }dW_{t},\text{ }t\geq 0,\;x>0,
\label{eq:U-decomp}
\end{equation}%
for a suitable initial condition
\begin{equation}
U(0,x)=u_{0}(x),\text{ \ }x>0.  \label{eq:initial-condition}
\end{equation}%
\smallskip
\noindent For every admissible pair $(\pi,c)\in\mathcal A$, define
\begin{equation}
Z_{t}^{\pi ,c}=\delta (t,X_{t}^{\pi ,c})
+U_{x}(t,X_{t}^{\pi ,c})\,X_{t}^{\pi ,c}\,\sigma _{t}^{\top }\pi _{t},
\quad t\geq 0.  \label{eq:Z}
\end{equation}%
We stress that the process $Z^{\pi ,c}$ is the diffusion coefficient obtained by composing
the random field with the controlled wealth, and not a state process with an
independently chosen initial value. Its value at zero is determined by the
data, $Z_{0}^{\pi ,c}=\delta (0,x)+U_{x}(0,x)\,x\, \sigma _{0}^{\top }\pi
_{0}$, and no separate condition on $Z_{0}^{\pi ,c}$ is imposed.

\smallskip

\noindent Whenever
\begin{equation}
\int_{0}^{T}\left|F\! \left( s,c_{s},U(s,X_{s}^{\pi ,c}),
Z_{s}^{\pi ,c}\right)\right|ds<\infty
\quad\text{$\mathbb P$-a.s., for every }T>0,  \label{eq:F-integrability}
\end{equation}%
the associated value process is defined by
\begin{equation}
V_t^{\pi,c}=U(t,X_{t}^{\pi ,c})
+\int_{0}^{t}F\! \left( s,c_{s},U(s,X_{s}^{\pi ,c}),
Z_{s}^{\pi ,c}\right)ds,\qquad t\geq0.
\label{eq:value-process}
\end{equation}%
The forward aggregator system $(U,F)$ is said to be \emph{consistent} if
the following two conditions hold:

\smallskip

\noindent i) for every $(\pi ,c)\in \mathcal{A}$ satisfying
\eqref{eq:F-integrability}, the process $V^{\pi,c}$ is a local
$\mathcal{F}_t$-supermartingale,
and

\noindent ii) there exists $(\pi ^{\star },c^{\star })\in \mathcal{A}$
satisfying \eqref{eq:F-integrability} such that the process
$V^{\pi^{\star },c^{\star }}$ is a local $\mathcal{F}_t$-martingale.
\medskip

If $(U,F)$ is a consistent forward aggregator system, then $U$ is called a
\emph{forward recursive utility}, $F$ a \emph{forward
aggregator}, and any pair $(\pi ^{\star },c^{\star })$ satisfying ii) is
called \emph{optimal}.
\end{definition}

\begin{remark}\label{rem:2-4}
As in the existing forward utility literature, the
volatility process $\delta $ in \eqref{eq:U-decomp} is a modeling input, in
contrast to its analogue in the classical (backward) recursive utility
literature where it is part of the solution process.

Furthermore, the process $Z^{\pi ,c}$ in \eqref{eq:Z} is the diffusion coefficient of $%
U(t,X_{t}^{\pi ,c}),$ produced by the It\^{o}--Ventzel formula. It
combines the modeling input $\delta $ and the term
$U_{x}(t,x)X\sigma^{\top}\pi $ that comes from the dynamics of the underlying wealth
controlled process. Its role as the $z$-argument in $F$ is the (forward)
analogue of the $Z$-process in the classical BSDE formulation of recursive
utility.
\end{remark}

\begin{remark}\label{rem:2-5}
In the classical recursive utility literature, one
often imposes additional Lipschitz assumptions on $F$ with respect to $(u,z)$%
, which however fail for the popular Epstein--Zin recursive utilities. Since later on we focus on the forward analogue of the latter, we do not impose any such conditions. In their absence, however, the integrability requirement
\eqref{eq:F-integrability} may restrict the admissible control pairs to which
the criterion applies. This restriction depends on the pair $(U,F)$ itself, in
analogy with the restricted class of strategies in \citet[Remark~3.14]{Kallblad-2016}. The role of \eqref{eq:F-integrability} is the following:

\noindent i) The value process \eqref{eq:value-process} is defined only for
the admissible control pairs that satisfy \eqref{eq:F-integrability}.

\noindent ii) Integrability of the positive part is a consequence of the
Bellman inequality \eqref{eq:HJB-ineq} below. Indeed, writing $%
v_{t}=X_{t}^{\pi ,c}\, \sigma _{t}^{\top }\pi _{t}$, this inequality gives, for $dt\otimes d\mathbb{P}$-a.e.\ $(t,\omega )$,
\begin{equation}
F\! \left( t,c_{t},U(t,X_{t}^{\pi ,c}),Z_{t}^{\pi ,c}\right) \leq
-b(t,X_{t}^{\pi ,c})-U_{x}(t,X_{t}^{\pi ,c})\! \left( v_{t}^{\top }\theta
_{t}-c_{t}\right) -\tfrac{1}{2}U_{xx}(t,X_{t}^{\pi ,c})|v_{t}|^{2}-v_{t}^{%
\top }\delta _{x}(t,X_{t}^{\pi ,c}).  \label{eq:F-plus-bound}
\end{equation}%
Each term on the right-hand side is $\mathbb{P}$-a.s.\ integrable on $[0,T]$
for each $T>0$, by the continuity in $x$ of $b$, $U_{x}$, $U_{xx}$ and $%
\delta _{x}$ along the continuous process $X^{\pi ,c}$, together with the
integrability requirements of Definition~\ref{defn:2-1} and $\int_{0}^{T}|\theta
_{t}|^{2}dt<\infty $. Hence,
\begin{equation*}
\int_0^T F^+\!\left(s,c_s,U(s,X_s^{\pi,c}),Z_s^{\pi,c}\right)ds<\infty
\qquad\text{$\mathbb P$-a.s. for every }T>0.
\end{equation*}

\noindent iii) Integrability of the negative part is therefore the only part
of \eqref{eq:F-integrability} that is not guaranteed by the consistency requirement, and it remains an
admissibility condition on the control pair. Along the optimal pair of
Theorem~\ref{thm:2-6} it holds automatically, as shown in Appendix~\ref{app:A}, since it is equal to locally integrable terms.
\end{remark}

Consistency is understood locally, in line with the local forward criteria
of \citet[Definition~2.1]{Kallblad-2016}. Whenever expectations at a
deterministic horizon are used, we additionally assume that the relevant
stopped family is uniformly integrable, equivalently, of class $(D)$ in the
setting at hand. These are additional assumptions, under which, the local properties become true
(super)martingale properties and the criterion considers the corresponding
admissible control pairs in the expected-utility sense
\citep[cf.][Remark~3.15]{Kallblad-2016}. They are not imposed in the general
results below, which are therefore, considered only in the local sense.

\subsection{The forward recursive HJB SPDE and verification results}\label{sec:2-2}

The martingale optimality principle leads to the following \emph{forward
recursive HJB SPDE},
\begin{equation}
dU(t,x)=b(t,x)\,dt+\delta (t,x)^{\top }dW_{t},\text{ \  \ }t\geq 0,\text{ }x>0,
\label{Forward-Recursive-SPDE}
\end{equation}%
where the drift $b$ satisfies, for $dt\otimes d\mathbb{P}$-a.e.\ $(t,\omega )$ and each $x>0$, the optimality requirement
\begin{equation*}
\begin{aligned}
b(t,x)+\sup_{(\pi ,c)}\Bigl( &U_{x}(t,x)\bigl(x\, \pi ^{\top }\sigma
_{t}\theta _{t}-c\bigr)+\tfrac{1}{2}U_{xx}(t,x)\,\bigl|x\, \pi ^{\top }\sigma
_{t}\bigr|^{2}+x\, \pi ^{\top }\sigma _{t}\, \delta _{x}(t,x)\\
&+F\bigl(t,c,U(t,x),\delta (t,x)+U_{x}(t,x)\,x\, \sigma _{t}^{\top }\pi
\bigr)\Bigr)=0,
\end{aligned}
\end{equation*}%
together with the initial condition \eqref{eq:initial-condition}. The
It\^{o}--Ventzel formula shows that this drift condition is precisely the
local supermartingale condition. Theorem~\ref{thm:2-6} gives the corresponding
verification result.

\medskip \noindent We introduce the following notation. For $t\geq 0$, $%
\omega \in \Omega$, $x>0$, and $\pi \in \mathbb{R}^{n}$, we define
\begin{equation*}
Z(t,\omega ,x,\pi )=\delta (t,\omega ,x)+U_{x}(t,\omega ,x)\,x\, \sigma
_{t}(\omega )^{\top }\pi ,
\end{equation*}%
and 
\begin{align*}
H(t,\omega ,x,\pi ,c)& =b(t,\omega ,x)+U_{x}(t,\omega ,x)\left( x\pi ^{\top
}\sigma _{t}(\omega )\theta _{t}(\omega )-c\right) \\
& \quad +\tfrac{1}{2}U_{xx}(t,\omega ,x)\,|x\pi ^{\top }\sigma _{t}(\omega
)|^{2}+x\pi ^{\top }\sigma _{t}(\omega )\, \delta _{x}(t,\omega ,x) \\
& \quad +F\! \left( t,\omega ,c,U(t,\omega ,x),Z(t,\omega ,x,\pi )\right) .
\end{align*}%
These are definitions and not additional requirements. In this notation, the
drift condition above reads for $dt\otimes d\mathbb{P}$-a.e.\ $(t,\omega )$
and every $x>0,$
\begin{equation*}
\sup_{(\pi ,c)}H(t,x,\pi ,c)=0.
\end{equation*}

\medskip \noindent In the unnormalized setting, the supremum with respect to $\pi$ is
coupled with the supremum with respect to $c$. Indeed, the $z$-argument of $F$ depends on $\pi $ through $Z(t,x,\pi
)$, and the first-order condition in $\pi $ of the supremum defining the
drift condition reads
\begin{equation} \label{eq:foc}
U_{x}(t,x)\,\sigma _{t}\theta _{t}+U_{xx}(t,x)\,x\, \sigma _{t}\sigma
_{t}^{\top }\pi +\sigma _{t}\, \delta _{x}(t,x)+U_{x}(t,x)\, \sigma _{t}\,
F_{z}\! \left( t,c,U(t,x),Z(t,x,\pi )\right) =0.
\end{equation}%
When $F_{z}\neq 0$, the last term depends on $c$ through $F_{z}$, so that the
optimal $\pi $ depends on $c$ and the optimization in $\pi $ cannot be
separated from the optimization in $c$. Without additional assumptions, no further simplification is possible.

\begin{theorem}[Verification]\label{thm:2-6}
Let $(U,F)$ be a
forward aggregator system, where $U$ is a $\mathcal K_{\mathrm{loc}}^{2,\varepsilon }$%
-semimartingale random field for some $\varepsilon \in (0,1]$, with local
characteristics $(b,\delta )$ as in \eqref{eq:U-decomp}. Moreover, assume that:

i) there exist \emph{feedback controls}, that is, maps $\pi ^{\star
}\colon \lbrack 0,+\infty )\times \Omega \times (0,+\infty )\rightarrow
\mathbb{R}^{n}$ and $c^{\star }\colon \lbrack 0,+\infty )\times \Omega \times
(0,+\infty )\rightarrow (0,+\infty )$ that are $\mathcal{F}_{t}$%
-progressively measurable in $(t,\omega )$ for each $x>0$ and Borel
measurable in $x$, such that, for $dt\otimes d\mathbb{P}$%
-a.e.\ $(t,\omega )$ and every $x>0$,
\begin{equation}
H\! \left( t,\omega ,x,\pi ^{\star }(t,\omega ,x),c^{\star }(t,\omega
,x)\right) =0,  \label{eq:HJB-eq}
\end{equation}%
and 
\begin{equation}
H(t,\omega ,x,\pi ,c)\leq 0\qquad \forall \,(\pi ,c)\in \mathbb{R}^{n}\times
(0,+\infty ),  \label{eq:HJB-ineq}
\end{equation}%
and

ii) for each $x>0$, the wealth equation controlled by the feedback controls
$\left( \pi ^{\star }(t,X_{t}^{\star }),c^{\star }(t,X_{t}^{\star })\right) $%
, namely
\begin{equation*}
dX_{t}^{\star }=X_{t}^{\star }\,  \pi ^{\star }(t,X_{t}^{\star
}) ^{\top }\sigma _{t}\! \left( dW_{t}+\theta _{t}\,dt\right)
-c^{\star }(t,X_{t}^{\star })\,dt,\text{ \ }t\geq 0,\quad X_{0}^{\star }=x,
\end{equation*}%
admits a unique strong solution $X^{\star }$ which is strictly positive, with $%
\big(\pi ^{\star }(t,X_{t}^{\star }),\allowbreak c^{\star }(t,X_{t}^{\star
})\big)_{t\geq 0}\in \mathcal{A}$.

Then, $(U,F)$ is a consistent forward aggregator system, and the feedback
control pair $\big(\pi ^{\star }(t,X_{t}^{\star }),\allowbreak c^{\star
}(t,X_{t}^{\star })\big)_{t\geq 0}$ is optimal.
\end{theorem}

\begin{proof}
The proof is given in Appendix~\ref{app:A}.
\end{proof}

\medskip \noindent Following \citet{Duffie-Epstein}, we introduce a normalization under which optimization with respect to investment and consumption
decouples.

\subsection{Normalization and related results}\label{sec:2-3}

We work with forward aggregators of the form
\begin{equation}
F(t,c,u,z)=f(t,c,u)+\tfrac{1}{2}A(u)|z|^{2},\text{ \  \ }t\geq 0,\text{ }c>0,%
\text{ }u\in \mathcal{U},\text{ }z\in \mathbb{R}^{d},
\label{eq:norm-F}
\end{equation}%
where $A:\mathcal{U}\rightarrow (-\infty ,0]$ is continuous and $f$
does not depend on $z$. Here $A$ is introduced directly as the coefficient of
the quadratic volatility term of the forward aggregator. It represents the forward
analogue of the \textit{variance multiplier} introduced in
\citet{Duffie-Epstein}.

\medskip \noindent We collect here the assumptions on the normalization map.

\begin{assumption}\label{asm:2-7}
Fix $u_{\ast }\in \mathcal{U}$ and
let $\varphi $ be twice continuously differentiable on $\mathcal{U}$, with
\begin{equation}
\varphi ^{\prime }>0,\qquad \varphi ^{\prime \prime }(u)=A(u)\varphi
^{\prime }(u),\qquad \varphi (u_{\ast })=0,\qquad \varphi ^{\prime }(u_{\ast
})=1.  \label{phi-cond}
\end{equation}
\end{assumption}

\medskip

The differential equation in \eqref{phi-cond} is precisely what removes the
volatility term under normalization. Since $A$ is non-positive and $%
\varphi ^{\prime }$ is positive, $\varphi $ is strictly increasing and
concave, and $\bar{\mathcal{U}}=\varphi (\mathcal{U})$ is an open interval.
The map $\varphi $ is determined by the differential equation up to a
positive affine transformation $\varphi \mapsto a\varphi +b$, and the last
two conditions in \eqref{phi-cond} fully determine a representative of the class of solutions. Indeed, affine transformations do not affect the induced performance criterion.

\begin{proposition}\label{prop:2-8}
Let $\varphi $ satisfy Assumption~\ref{asm:2-7} and set
\begin{equation*}
\bar{U}=\varphi \circ U,\text{ \  \  \ }\bar{\mathcal{U}}=\varphi (\mathcal{U%
}),
\end{equation*}
and
\begin{equation*}
\bar{F}(t,c,\bar{u})=\varphi ^{\prime }\! \left( \varphi ^{-1}(\bar{u}%
)\right) f\! \left( t,c,\varphi ^{-1}(\bar{u})\right) ,\quad t\geq0,\ c>0,\ \bar{u}%
\in \bar{\mathcal{U}}.
\end{equation*}
Let $(U,F)$ be a consistent forward aggregator system on $\mathcal{U}$ and let $%
(\pi ^{\star },c^{\star })\in \mathcal{A}$ be an optimal control pair. Then,
the following hold:

\smallskip

\noindent i) for every $(\pi ,c)\in \mathcal{A}$ satisfying
\eqref{eq:F-integrability}, the associated value process $\bar{V}^{\pi,c}$ of $(\bar{U},\bar{F}%
)$ is a local $\mathcal{F}_t$-supermartingale, and\smallskip

\noindent ii) the value process $\bar{V}^{\pi^{\star}, c^{\star}}$ is a local $\mathcal{F}_t$-martingale. 

\smallskip

\noindent Hence, $(\bar U,\bar F)$ is a consistent forward aggregator system
on $\bar{\mathcal{U}}$. 

\noindent Conversely, if $(U,F)$ is a forward aggregator
system of form \eqref{eq:norm-F} and $(\bar U,\bar F)$ is consistent on
$\bar{\mathcal{U}}$, then $(U,F)$ is consistent on $\mathcal{U}$, and the
two systems have the same optimal control pairs.
\end{proposition}

\begin{proof}
The proof is given in Appendix~\ref{app:A}.
\end{proof}

\medskip

A forward aggregator system is called \emph{normalized} if its aggregator is
independent of $z$. Proposition~\ref{prop:2-8} allows us to work with the normalized
system when the dependence on $z$ is purely quadratic, and we do so from here
on. We, henceforth, omit the bars from the normalized objects.

To display the resulting separation, write $v=x\,\sigma _{t}^{\top }\pi $ for
the wealth diffusion coefficient of the portfolio generated by $\pi$. For $dt \otimes d
\mathbb{P}$-a.e. $(t,\omega)$ and every $x>0$, the optimality requirement becomes
\begin{equation*}
b(t,x)+\sup_{v\in \mathrm{Im}\, \sigma _{t}^{\top }}\Bigl(\tfrac{1}{2}\,U_{xx}(t,x)\,|v|^{2}+v^{\top }\bigl(U_{x}(t,x)\theta _{t}+\delta _{x}(t,x)\bigr)\Bigr) +\sup_{c>0}\Bigl(F(t,c,U(t,x))-U_{x}(t,x)\,c\Bigr)=0, 
\quad t\geq0,\ x>0,
\end{equation*}%
in which the first supremum involves only the wealth diffusion $v$ and the
second only the consumption rate $c$. The reason is precisely the one identified in \eqref{eq:foc}. Specifically, for a normalized aggregator, $F_{z}\equiv 0$, so the term $%
U_{x}\,\sigma _{t}F_{z}$ disappears from the first-order condition in $\pi $,
and the consumption rate no longer enters the portfolio choice. Conversely,
the portfolio no longer enters the consumption choice because $\pi $ appears
in the criterion only through $v$ and through the $z$-argument of $F$, and
the latter has been removed. The maximizations can then be performed one at a
time, in either order. Since $v$ is constrained to the subspace
$\mathrm{Im}\,\sigma_t^{\top}$, only the projection
$\sigma_t^{+}\sigma_t\delta_x$ contributes to the first supremum. The second
supremum is the Fenchel transform of $F$ in its consumption argument,
\begin{equation} \label{eq:fenchelF}
\widetilde{F}(t,d,u)=\sup_{c>0}\! \left( F(t,c,u)-d\,c\right) ,\text{ }%
d>0,\;u\in \mathcal{U}.
\end{equation}%
\smallskip
Thus, the forward recursive HJB SPDE \eqref{Forward-Recursive-SPDE} becomes
\begin{equation}
dU(t,x)=\Bigl( \frac{|U_{x}(t,x)\theta _{t}+\sigma_{t}^{+}\sigma_{t}\delta
_{x}(t,x)|^{2}}{2\,U_{xx}(t,x)}-\widetilde{F}\bigl(t,U_{x}(t,x),U(t,x)\bigr)%
\Bigr) dt +\delta (t,x)^{\top }dW_{t},  \text{ \  \ }t\geq 0,\text{ }x>0,
\label{eq:SPDE-norm}
\end{equation}%
with initial condition 
$$ U(0,x) = u_0(x), \quad x >0.$$

\medskip

\noindent\textit{Nondegeneracy.} The strict concavity in Definition~\ref{defn:2-2} and
the regularity of $U$ do not by themselves yield a strict sign for the second
derivative at every point. We therefore impose, as a standing assumption for
the classical form of the forward recursive HJB SPDE and for the
duality results of
Section~\ref{sec:3},
\begin{equation}
U_{xx}(t,x)<0,\qquad t\geq 0,\ x>0,\quad \mathbb{P}\text{-a.s.}
\label{eq:nondegeneracy}
\end{equation}
The strict inequality is what allows us to divide by $U_{xx}$, which is done
in the drift of \eqref{eq:SPDE-norm}, in the portfolio feedback \eqref{eq:feedback}
below, and in the inverse marginal computations of Section~\ref{sec:3}. Under regularity conditions on the convex conjugate $\widetilde{U}$ of $U$,
\eqref{eq:nondegeneracy} is equivalent to
\begin{equation*}
\widetilde U_{yy}(t,y)=-\frac{1}{U_{xx}(t,I(t,y))}>0,\qquad y>0,
\end{equation*}
in the notation of \eqref{eq:conjugacy_relations}. We note that condition
\eqref{eq:nondegeneracy} is not part of the martingale definition of
consistency in Definition~\ref{defn:2-3}, and it is not used in Theorem~\ref{thm:2-6}.

\bigskip

\noindent We introduce the following regularity assumptions for the normalized forward aggregator.

\begin{assumption}\label{asm:2-9}
For $dt\otimes d\mathbb{P}$-a.e.\ $%
(t,\omega )$, the map $(c,u)\mapsto F(t,\omega ,c,u)$ is continuously
differentiable on $(0,\infty )\times \mathcal{U}$, its derivatives $F_{c}$
and $F_{u}$ are jointly continuous in $(c,u)$ on $(0,\infty )\times \mathcal{U%
}$, and the Inada conditions in $c$ hold, namely,
\begin{equation*}
\lim_{c\downarrow 0}F_{c}(t,\omega ,c,u)=\infty \text{ \  \  \ and \  \  \ }%
\lim_{c\uparrow \infty }F_{c}(t,\omega ,c,u)=0,\text{ \  \ for every }u\in
\mathcal{U}.
\end{equation*}
\end{assumption}

\medskip

Recall from Definition~\ref{defn:2-2} that map $c\mapsto F(t,c,u)$ is strictly
increasing and strictly concave. Together with Assumption~\ref{asm:2-9}, this gives,
for $dt\otimes d\mathbb{P}$-a.e.\ $(t,\omega )$ and each $(d,u)\in (0,\infty
)\times \mathcal{U}$, a unique interior maximizer of $c\mapsto F(t,c,u)-d\,c$%
, at which $F_{c}=d$, namely the optimizer $F_{c}^{-1}(t,d,u)$.
The Fenchel transform $\widetilde{F}$ in \eqref{eq:fenchelF} is therefore finite and the maximum is attained
at $c=F_{c}^{-1}(t,d,u)$, and for $dt\otimes d\mathbb{P}$-a.e.\ $(t,\omega )$ and each $(d,u)\in (0,\infty
)\times \mathcal{U}$, we have
\begin{equation*}
\widetilde{F}(t,d,u)=F\bigl(t,F_{c}^{-1}(t,d,u),u\bigr)-d\,F_{c}^{-1}(t,d,u).
\end{equation*}%
Moreover, $\widetilde{F}$ is continuously differentiable on $(0,\infty
)\times \mathcal{U}$ and satisfies the envelope identities
\begin{equation}
\widetilde{F}_{d}(t,d,u)=-F_{c}^{-1}(t,d,u)\quad \text{ and } \quad \widetilde{F}%
_{u}(t,d,u)=F_{u}\bigl(t,F_{c}^{-1}(t,d,u),u\bigr),  \label{eq:F-envelope}
\end{equation}%
for all $(d,u)\in (0,\infty )\times \mathcal{U}$ and $dt\otimes d\mathbb{P}$-a.e.\ $(t,\omega )$. The first identity in
\eqref{eq:F-envelope} is the differentiability of a convex conjugate at a
point with a unique maximizer. The second is the envelope theorem. The joint
continuity of $F_{c}$ and $F_{u}$ assumed above, together with the strict
monotonicity of $F_{c}$ in $c$, makes $F_{c}^{-1}$ continuous in $(d,u)$ and
$\widetilde{F}_{d}$, $\widetilde{F}_{u}$ continuous as well.

\begin{theorem}[Verification]\label{thm:2-10}
Let $(U,F)$ be a
normalized forward aggregator system on $\mathcal{U}$ satisfying
Assumption~\ref{asm:2-9}, where $U$ is a $\mathcal K_{\mathrm{loc}}^{2,\varepsilon }$%
-semimartingale random field for some $\varepsilon \in (0,1]$, with local
characteristics $(b,\delta )$ as in \eqref{eq:U-decomp} and satisfying the
nondegeneracy condition \eqref{eq:nondegeneracy}. Assume that the random
field $U$ satisfies the normalized forward recursive HJB SPDE
\eqref{eq:SPDE-norm}. 

\noindent 
Let $\pi^{\star}: [0,\infty) \times \Omega \times (0, \infty) \mapsto \mathbb{R}^n$, and $ c^{\star}: [0,\infty) \times \Omega \times (0, \infty) \mapsto (0,\infty)$ be defined as 
\begin{equation}
\pi ^{\star }(t,x)
 =-\frac{U_{x}(t,x)}{x\,U_{xx}(t,x)}\,(\sigma _{t}^{\top })^{+}
\left( \theta _{t}+\frac{\delta _{x}(t,x)}{U_{x}(t,x)}\right),\qquad t\geq0,\ x>0,
\label{eq:feedback}
\end{equation}
and
\begin{equation}
c^{\star }(t,x)=F_{c}^{-1}\! \left( t,U_{x}(t,x),U(t,x)\right),\qquad t\geq0,\ x>0,
\label{eq:feedback_consumption}
\end{equation}
and assume that the wealth equation controlled by the above feedback controls admits a
unique strong solution $X^{\star }$ which is strictly positive, and that $%
\big(\pi ^{\star }(t,X_{t}^{\star }),\allowbreak c^{\star }(t,X_{t}^{\star
})\big)_{t\geq 0}\in \mathcal{A}$. Then, $(U,F)$ is a consistent forward
aggregator system and the control pair $\big(\pi ^{\star }(t,X_{t}^{\star
}),\allowbreak c^{\star }(t,X_{t}^{\star })\big)_{t\geq 0}$ is optimal.
\end{theorem}

\begin{proof}
The proof is given in Appendix~\ref{app:A}.
\end{proof}

\medskip

\noindent Explicit conditions on the model inputs guaranteeing the
admissibility of the feedback control pair are provided in Section~\ref{sec:3}
(Assumption~\ref{asm:3-13} and Theorem~\ref{thm:3-14}). They resemble those in Assumption 2.3 in
\citet{LSZ25}, where the risk tolerance $-U_{x}/U_{xx}$ is required to be
bounded above and below away from zero, and the ratio $|\delta _{x}/U_{xx}|$
to be bounded.

\begin{remark}\label{rem:2-11}
We note that in the term $%
\widetilde{F}(t,U_{x}(t,x),U(t,x))$ both the forward recursive utility $%
U(t,x)$ and its marginal $U_{x}(t,x)$ appear. This nonlinear coupling makes
the problem in general non-tractable. However, for specific but rich enough cases, which we
study in Section~\ref{sec:4}, we are able to produce closed form solutions and draw
analogies with their classical counterparts.
\end{remark}

\section{Duality results for normalized forward recursive aggregator systems}\label{sec:3}

In this section, we develop the convex duality theory for
forward recursive aggregator systems. We first introduce the family of state
price densities associated with the market with stock prices as in \eqref{eq:marketstocks} together with the associated
deflated process. We then show that the convex conjugate of a forward
recursive utility solves a \textit{dual forward recursive HJB SPDE}, which is the Legendre
counterpart of the primal one. We establish the corresponding dual
consistency on the family of state price densities and, in turn identify the
state price density selected by the primal optimum, which turns out to be the marginal
utility along the optimal wealth multiplied by the factor that removes its
recursive drift. Finally, we provide a theorem that derives, using
explicit bounds on the model inputs, the existence and uniqueness of the dual
and primal solution families, the admissibility of the feedback control pair, and the
resulting primal and dual consistency properties. This approach
extends the results of \citet{EKM13} to the recursive setting. These results provide the tools used in Section~\ref{sec:4} where we introduce and solve in detail the forward analogue of the classical Epstein--Zin recursive utility maximization problem.

\subsection{State price densities}\label{sec:3-1}

For an admissible control pair $(\pi ,c)\in \mathcal{A}$ and a positive It%
\^{o} process $Y$, the deflated process $H^{\pi ,c}(Y)$ is defined as
\begin{equation}
H_{t}^{\pi ,c}(Y)=X_{t}^{\pi ,c}\,Y_{t}+\int_{0}^{t}Y_{s}\,c_{s}\,ds,\text{
\ }t\geq 0.  \label{eq:budget_process}
\end{equation}

\begin{definition}\label{defn:3-1}
A \emph{state price density} is a
positive It\^{o} process $Y$ such that for every $(\pi ,c)\in \mathcal{A}$, $%
H^{\pi ,c}(Y)$ is a local martingale. We denote by $\mathcal{Y}$ the family
of state price densities.
\end{definition}

\bigskip

\medskip Using the wealth dynamics \eqref{eq:wealth}, $H^{\pi ,c}(Y)$ is a
local martingale for every $(\pi ,c)\in \mathcal{A}$ if and only if $Y$
satisfies, for some initial value $y>0$ and some predictable process $\nu $
with $\nu _{t}\in (\mathrm{Im}\, \sigma _{t}^{\top })^{\perp }$, for $dt\otimes d%
\mathbb{P}$-a.e.\ $(t,\omega )$,
\begin{equation}
dY_{t}^{y,\nu }=Y_{t}^{y,\nu }\bigl(-\theta _{t}+\nu _{t}\bigr)^{\top
}dW_{t},\qquad Y_{0}^{y,\nu }=y,\qquad t\geq0.  \label{eq:SPD}
\end{equation}

\medskip \noindent A dual control $\nu $ is \emph{admissible} if it is a
$\mathcal{F}_{t}$-predictable process satisfying $\nu _{t}\in (\mathrm{Im}%
\, \sigma _{t}^{\top })^{\perp }$, for $dt\otimes d\mathbb{P}$-a.e.\ $(t,\omega )$, and
\begin{equation*}
\int_{0}^{T}\! |\nu _{t}|^{2}
\,dt<\infty \text{ \ $\mathbb P$-a.s., for each }T>0.
\end{equation*}%
We denote the set of admissible dual controls by
$\mathcal A^{\mathrm{dual}}$ and we set
\begin{equation}
\mathcal{Y}(y)=\bigl \{Y^{y,\nu }:\nu \in \mathcal{A}^{\mathrm{dual}}\bigr \}%
,\text{ \ }y>0,\text{ \ and \ }\mathcal{Y}=\bigcup_{y>0}\mathcal{Y}%
(y).  \label{eq:SPD_family}
\end{equation}%
This formulation distinguishes the dual control $\nu $, the equation \eqref{eq:SPD} that
it parametrizes, and the resulting density $Y^{y,\nu }$. When the initial
value plays no role, we write $Y^{\nu }$ for an element of $\mathcal{Y}$.  
The set $\mathcal{Y}$ depends only on the market processes $(\mu
,\sigma )$ through the market price of risk $\theta $ and the projection $%
\sigma _{t}^{+}\sigma _{t}$. Clearly, if the market is complete, $\sigma
_{t}^{+}\sigma _{t}\equiv I_{d}$, $\nu \equiv 0$, and $\mathcal{Y}$ consists of a single state price density, up to a
positive scaling.

\medskip

\subsection{Conjugate forward recursive utility and dual forward recursive SPDE}\label{sec:3-2}

We analyse the conjugate random field $\widetilde{U}$ and derive the equation
that it satisfies. As in the literature of forward utilities
\citep[see, e.g.,][]{EKM13,LSZ25},
we impose additional regularity conditions on the random field $U$ for the dual forward recursive HJB SPDE to be well defined.

\begin{assumption}\label{asm:3-2}
For some $\varepsilon \in (0,1]$,

i) $U$ is a $\mathcal K_{\mathrm{loc}}^{2,\varepsilon }$-semimartingale
random field and, $\mathbb P$-almost surely, the map
$x\mapsto U(t,x)$ is three times continuously differentiable for every
$t\geq0$, with third derivative locally $\varepsilon$-H\"{o}lder on compact
subsets of $(0,\infty)$,

ii) $\mathbb P$-almost surely, for every $t\geq 0$, the marginal
$U_{x}(t,\cdot )$ satisfies the
Inada conditions
\begin{equation*}
\lim_{x\downarrow 0}U_{x}(t,x)=\infty \text{ \ and \ }\lim_{x\uparrow \infty
}U_{x}(t,x)=0,
\end{equation*}%

iii) the inverse marginal field
\begin{equation*}
I(t,\omega ,y)=\bigl(x\mapsto U_{x}(t,\omega ,x)\bigr)^{-1}(y),\qquad t\geq0,\ \omega \in \Omega ,\ y>0,
\end{equation*}%
is a $\mathcal K_{\mathrm{loc}}^{2,\varepsilon }$-semimartingale random field.
\end{assumption}

\medskip

Strict concavity and the Inada conditions in ii) make $U_x(t,\cdot)$ a
bijection from $(0,\infty)$ onto $(0,\infty)$, so that $I(t,y)$ is
well defined for all $y>0$. The regularity of this specific field is assumed
in iii), and is not derived from the dual equation. The It\^{o}--Ventzel
calculations of Appendix~\ref{app:B} use (i) for the field $U$ and its derivative field
$U_{x}$, together with their local characteristics $(b,\delta )$ and
$(b_{x},\delta _{x})$, and (iii) for the field $I$ along which the conjugate
is computed. No It\^{o} decomposition of $\delta $ or of $U_{xxx}$ is
required. This resembles Assumption~2.3 in \citet{LSZ25}, and the
inverse field condition is standard in the It\^{o} random field setting, see
\citep{EKM13}.

\medskip

Next, we define the convex dual
\begin{equation}
\widetilde{U}(t,y)=\sup_{x>0}\left( U(t,x)-x\,y\right)
=U(t,I(t,y))-y\,I(t,y),\quad t\geq0,\ y>0.  \label{eq:fenchel_U}
\end{equation}
The technical conjugacy calculation is given in
Lemma~\ref{lem:3-3} in Appendix~\ref{app:B}. In particular, $\widetilde U$ is a semimartingale
random field with local characteristics
$(\widetilde b,\widetilde\delta)$. The conjugacy identities hold $\mathbb P$-a.s.\ for every $t\geq0$ and $y>0$,
\begin{equation}
I(t,y)=-\widetilde{U}_{y}(t,y),\qquad U_{xx}(t,I(t,y))=-\frac{1}{\widetilde{U%
}_{yy}(t,y)}.
\label{eq:conjugacy_relations}
\end{equation}%
The differential characteristics satisfy, for $dt\otimes d\mathbb P$-a.e.
$(t,\omega)$ and every $y>0$,
\begin{equation}
\widetilde{\delta }(t,y)=\delta (t,I(t,y))\quad \text{and}\quad \widetilde{b}
(t,y)=b(t,I(t,y))-\frac{1}{2}\, \frac{|\delta _{x}(t,I(t,y))|^{2}}{
U_{xx}(t,I(t,y))},
\label{eq:fenchel_diff_rule}
\end{equation}%
and, as a consequence,
\begin{equation}
\delta _{x}(t,I(t,y))=-\frac{\widetilde{\delta }_{y}(t,y)}{\widetilde{U}
_{yy}(t,y)}.  \label{eq:delta_x_via_dual}
\end{equation}

\medskip

\noindent To ease the presentation, we write
\begin{equation}
\hat u(t,y)=\widetilde U(t,y)-y\widetilde U_y(t,y)=U(t,I(t,y)),\qquad t\geq0,\ y>0.
\label{eq:dual_continuation_level}
\end{equation}
In other words, $\hat u(t,y)$ is the primal utility level expressed at the
dual state $y$.

\medskip

\noindent When $(U,F)$ solves the primal SPDE \eqref{eq:SPDE-norm}, the
drift characteristic of $\widetilde U$ is given, for
$dt\otimes d\mathbb P$-a.e.\ $(t,\omega)$ and every $y>0$, by
\begin{equation}
\widetilde{b}(t,y)=-\widetilde{F}\bigl(t,y,\hat u(t,y)\bigr)
-\tfrac{1}{2}\,y^{2}\, \widetilde{U}%
_{yy}(t,y)\,|\theta _{t}|^{2} +y\, \theta _{t}^{\top }\widetilde{\delta }_{y}(t,y)+\frac{%
|(I_{d}-\sigma_{t}^{+}\sigma_{t})\, \widetilde{\delta }_{y}(t,y)|^{2}}{2\,
\widetilde{U}_{yy}(t,y)}.
\label{eq:dual_drift}
\end{equation}%

\begin{proposition}[{Primal recursive SPDE $\implies $ dual recursive SPDE}]\label{prop:3-4}
Let $(U,F)$ satisfy the
normalized primal forward recursive HJB SPDE \eqref{eq:SPDE-norm},
Assumption~\ref{asm:2-9} and Assumption~\ref{asm:3-2}. Then, the convex dual $\widetilde{U}$ solves the
\emph{dual normalized forward recursive HJB SPDE}
\begin{equation}
\begin{aligned} d\widetilde U(t,y) &=\Bigl( -\widetilde
F\bigl(t,y,\hat u(t,y)\bigr)
-\tfrac{1}{2}\,y^{2}\, \widetilde U_{yy}(t,y)\,|\theta_t|^{2}\\ &\quad +y\,
\theta_t^\top \widetilde \delta_y(t,y) +\frac{|(I_{d}-\sigma_{t}^{+}\sigma_{t})\, \widetilde
\delta_y(t,y)|^{2}}{2\, \widetilde U_{yy}(t,y)} \Bigr)\,dt +\widetilde
\delta(t,y)^\top dW_t, \text{ \ }t\geq 0,\text{ }y>0, \end{aligned}  \label{eq:dual_SPDE}
\end{equation}
with initial condition
$$ \widetilde{U}(0,y) = \Tilde{u}_0(y), \quad y >0.$$
\end{proposition}

\begin{proof}
The proof is given in Appendix~\ref{app:B}.
\end{proof}

\subsection{Dual consistency on \texorpdfstring{$\mathcal{Y}$}{Y}}\label{sec:3-3}

The dual value process introduced below is defined only along those dual
controls for which the running dual aggregator term is integrable. This is not
a property derived from the dual equation but a requirement imposed on the
control, and we state it as such: for $\nu\in\mathcal A^{\mathrm{dual}}$ satisfying
\begin{equation}
\int_0^T\left|\widetilde F\bigl(s,Y_s^\nu,\hat u(s,Y_s^\nu)\bigr)\right|ds
<\infty\quad\text{$\mathbb P$-a.s., for every }T>0,
\label{eq:Ftilde-integrability}
\end{equation}
the associated \emph{dual value process} is well defined and given by
\begin{equation}
\mathcal{Z}_{t}^{\nu }=\widetilde{U}(t,Y_{t}^{\nu })+\int_{0}^{t}\widetilde{%
F}\bigl(s,Y_{s}^{\nu },\hat u(s,Y_s^\nu)\bigr)\,ds,\quad t\geq0.  \label{eq:Zcal_def}
\end{equation}

\begin{definition}\label{defn:3-5}
The conjugate pair
$(\widetilde U,\widetilde F)$ is \emph{dual consistent on $\mathcal Y$} if:

\smallskip

\noindent \textit{(i)} for every $\nu \in \mathcal{A}^{\mathrm{dual}}$
satisfying \eqref{eq:Ftilde-integrability}, the process $\mathcal Z^\nu$ is
a local $\mathcal{F}_t$-submartingale,

and

\noindent \textit{(ii)} there exists a dual control
$\nu^{\mathrm{dual}}\in\mathcal A^{\mathrm{dual}}$ satisfying
\eqref{eq:Ftilde-integrability} for
which $\mathcal Z^{\nu^{\mathrm{dual}}}$ is a local $\mathcal{F}_t$-martingale.

\smallskip

\noindent If the conjugate pair
$(\widetilde U,\widetilde F)$ is \emph{dual consistent on $\mathcal Y$}, we call $\widetilde U$ a \emph{dual forward
recursive utility}, $\widetilde F$ its \emph{dual forward aggregator}, and $\nu^{\mathrm{dual}}$ the \textit{optimal dual control}.
\end{definition}

\begin{remark}[One-sided integrability]\label{rem:3-6}
If
$\widetilde U$ satisfies
\eqref{eq:dual_SPDE}, then
\begin{equation*}
\int_0^T\widetilde F^-
\bigl(s,Y_s^\nu,\hat u(s,Y_s^\nu)\bigr)ds<\infty
\qquad\text{$\mathbb P$-a.s. for every }T>0,
\end{equation*}
as shown in Appendix~\ref{app:B}. Consequently,
\eqref{eq:Ftilde-integrability} may fail only through divergence of
the positive part. We stress that this is a consequence of the dual equation
and not part of Definition~\ref{defn:3-5}.
\end{remark}

\begin{theorem}[Dual verification]\label{thm:3-7}
Let
Assumptions~\ref{asm:2-9} and~\ref{asm:3-2} hold, and suppose that $\widetilde U$ satisfies the
dual forward recursive HJB SPDE \eqref{eq:dual_SPDE}. Define the dual
feedback map
\begin{equation}
\nu_t^{\mathrm{dual}}(y)
=-\frac{(I_d-\sigma_t^+\sigma_t)\widetilde\delta_y(t,y)}
{y\widetilde U_{yy}(t,y)},\qquad t\geq0,\ y>0.
\label{eq:nu_dual_def}
\end{equation}
Assume that the state price density equation \eqref{eq:SPD}, controlled by
$\nu_t=\nu_t^{\mathrm{dual}}(Y_t^{\nu^{\mathrm{dual}}})$, has a unique,
strictly positive, strong solution and that the resulting control belongs to
$\mathcal A^{\mathrm{dual}}$.
Then, the pair $(\widetilde U,\widetilde F)$ is dual consistent on $\mathcal Y$, with dual
optimal control $\nu^{\mathrm{dual}}$ and state price density
$Y^{\nu^{\mathrm{dual}}}$.
\end{theorem}

\begin{proof}
The proof is given in Appendix~\ref{app:B}.
\end{proof}

\medskip \noindent The above dual consistency uses only the wealth conjugate
$\widetilde U$ and the Fenchel transform $\widetilde F$ in the consumption
argument, and it does not require the convexity of $F$ in $u$. The density
$Y^{\nu^{\mathrm{dual}}}$ is the one for which the dual value process
\eqref{eq:Zcal_def} is a local $\mathcal{F}_t$-martingale. It should be
distinguished from the density $Y^{\star }$ generated by the primal optimizer
in Theorem~\ref{thm:3-12}, and we study their relation there.

\medskip

\noindent Next, we introduce the variational representation of
\citet{MX17}. It requires the convexity of the aggregator $F$ in its
utility argument $u$, and it involves an auxiliary process $\alpha $, which
plays the role of a stochastic discount rate. Its role is the following. In
the recursive setting, the marginal utility along the optimal wealth is not
itself a state price density, because the dependence of the aggregator on the
utility level adds a finite variation term to its dynamics. Discounting by
$\alpha $ removes precisely this term and restores the time additive
structure, in which the discounted marginal utility does become a state price
density. Any admissible discount process $\alpha $, combined with any state
price density, produces in this way a Fenchel type upper bound for the primal
criterion. We show in Proposition~\ref{prop:3-10} and Theorem~\ref{thm:3-12} that this bound is
attained, and we identify where: equality holds at the discount process
$\alpha ^{\star }=-F_{u}$, evaluated along the optimal consumption and utility
levels, and at the state price density $Y^{\star }$ generated by the primal
optimizer.

\subsection{Convex duality using a variational transform}\label{sec:3-4}

We establish a link between the convex duality developed above and a
family of time-additive
discounted dual forward utilities $(\widetilde{U},G)$ for investment and consumption,
indexed by an auxiliary process $\alpha $. For an ordinary time-additive
forward utility, the marginal utility evaluated along the optimal wealth is
a state price density. In the present recursive setting, the marginal utility
includes an additional finite variation term and, as we show, it must be
multiplied by the recursive discount factor before a state price density is obtained. This
is why the auxiliary process $\alpha $ is needed.

\begin{assumption}\label{asm:3-8}
In addition to Assumption~\ref{asm:2-9},
we assume that, for $dt\otimes d\mathbb{P}$-a.e.\ $(t,\omega )$ and each $%
c>0 $, the map $u\mapsto F(t,c,u)$ is convex on $\mathcal{U}$.
\end{assumption}

Motivated by the variational representation used by
\citet{MX17} for the classical recursive utility problem, we define,
under Assumption~\ref{asm:3-8},
\begin{equation}
\widehat{F}(t,\omega ,c,\alpha )=\inf_{u\in \mathcal{U}}\left( F(t,\omega
,c,u)+\alpha \,u\right) ,\qquad t\geq0,\ \omega\in\Omega,\ c>0,\ \alpha\in\mathbb R.
\label{eq:utility_variational_transform}
\end{equation}%
We suppress the $\omega $-argument from now on, for notational convenience.
For $dt\otimes d\mathbb P$-a.e.\ $(t,\omega)$ and every
$(c,u)\in(0,\infty)\times\mathcal U$, the definition of the infimum gives
$\widehat{F}(t,c,\alpha )-\alpha u\leq F(t,c,u)$. Since $u\mapsto F(t,c,u)$
is differentiable and convex, equality at $u$ holds for
\begin{equation}
\alpha ^{\star }(t,c,u)=-F_{u}(t,c,u),\qquad t\geq0,\ c>0,\ u\in\mathcal U,  \label{eq:utility_variational_FOC}
\end{equation}%
in which case the infimum in \eqref{eq:utility_variational_transform} is
attained at $u$ itself. Equivalently,
\begin{equation}
F(t,c,u)=\sup_{\alpha \in \mathbb{R}}\bigl \{ \widehat{F}(t
,c,\alpha )-\alpha \,u\bigr \},\qquad t\geq 0,\ c>0,\ u\in \mathcal{U},
\label{eq:utility_variational_identity}
\end{equation}%
and the maximizer is $\alpha ^\star(t,c,u)$.

\medskip \noindent \textbf{Concave case.} When $u\mapsto F(t,c,u)$ is
concave, the pointwise variational identity has a formal analogue obtained by
replacing the infimum in \eqref{eq:utility_variational_transform} by a
supremum and the supremum in \eqref{eq:utility_variational_identity} by an
infimum, as in \citet[Remark~2.10]{MX17}. However, the joint saddle relation
with the consumption transform $G$, defined below in
\eqref{eq:G_alpha_definition}, does not follow from the concavity in $u$
alone. A concave analogue of the results below requires separate
joint-concavity, finiteness, and attainment assumptions. We therefore state
the results of this section only under the convexity condition in
Assumption~\ref{asm:3-8}.

\medskip \noindent To this end, we first introduce the class of discount
processes. A
progressively measurable real valued process $\alpha $ belongs to
$\mathcal{A}^{\mathrm{disc}}$ if, for every $T>0$,
\begin{equation*}
\int_0^T|\alpha_t|\,dt<\infty \text{ \  \ $\mathbb P$-a.s.,}
\qquad \text{and}\qquad
\widehat F(t,c,\alpha_t)>-\infty ,
\end{equation*}
for $dt\otimes d\mathbb P$-a.e.\ $(t,\omega)$ and every $c>0$. For
$\alpha \in \mathcal{A}^{\mathrm{disc}}$, set
\begin{equation}
\kappa _{s,t}^{\alpha }=\exp \Bigl(-\int_{s}^{t}\alpha _{r}\,dr\Bigr),\qquad
0\leq s\leq t.  \label{eq:kappa_alpha_def}
\end{equation}%
For a control pair $(\pi,c)\in\mathcal A$, define the discounted additive
value process associated with the discount $\alpha $,
\begin{equation}
S_{t}^{\alpha ,\pi ,c}=\kappa _{0,t}^{\alpha }\,U(t,X_{t}^{\pi
,c})+\int_{0}^{t}\kappa _{0,s}^{\alpha }\, \widehat{F}(s,c_{s},\alpha
_{s})\,ds,\qquad t\geq0.  \label{eq:discounted_primal_process}
\end{equation}%
We also define the Fenchel transform of $\widehat F$ with respect to the
consumption argument and the deflated density,
\begin{equation}
G(t,d,\alpha )=\sup_{c>0}\bigl( \widehat{F}(t,c,\alpha )-d\,c\bigr)%
,\qquad t\geq0,\ d>0,\ \alpha\in\mathbb R,
\label{eq:G_alpha_definition}
\end{equation}
and
\begin{equation}
D_{t}^{\alpha ,\nu }=\frac{Y_{t}^{\nu }}{\kappa _{0,t}^{\alpha }},\qquad t\geq0.
\label{eq:deflated_density_definition}
\end{equation}
The corresponding dual process is
\begin{equation}  \label{eq:dual_envelope_def}
\mathcal{Z}_t^{\alpha,\nu}
=\kappa_{0,t}^{\alpha}\widetilde U(t,D_t^{\alpha,\nu})
+\int_0^{t}\kappa_{0,s}^{\alpha}
G(s,D_s^{\alpha,\nu},\alpha_s)\,ds,\qquad t\geq0.
\end{equation}

\medskip

\textit{Integrability issue.} The fact that $\alpha $ already belongs to
$\mathcal{A}^{\mathrm{disc}}$ makes the
integrands in \eqref{eq:discounted_primal_process} and
\eqref{eq:dual_envelope_def} finite. Indeed, $\widehat{F}(t,c,\alpha
_{t})>-\infty $ by the definition of $\mathcal{A}^{\mathrm{disc}}$, while
$\widehat{F}(t,c,\alpha _{t})\leq F(t,c,u)+\alpha _{t}u<\infty $, for every
$u\in \mathcal{U}$. Subtracting $d\,c$ from the latter bound and taking the
supremum over $c>0$ gives $G(t,d,\alpha _{t})\leq \widetilde{F}%
(t,d,u)+\alpha _{t}u$, which is finite because $\widetilde{F}$ is, under
Assumption~\ref{asm:2-9}, while $G(t,d,\alpha _{t})\geq \widehat{F}(t,1,\alpha
_{t})-d>-\infty $. Under Assumption~\ref{asm:3-2}, the conjugate
$\widetilde{U}(t,\cdot )$ is in turn real valued on $(0,\infty )$, the
supremum defining the conjugate $\widetilde{U}(t,\cdot )$ being attained at
$I(t,\cdot )$.

\medskip \noindent However, the absolute convergence of the corresponding time
integrals is not, in general, guaranteed. It has to be imposed
separately, for the particular discount process, control
pair and density used. We say that a triple $(\alpha ,(\pi ,c),Y^{\nu })\in
\mathcal{A}^{\mathrm{disc}}\times \mathcal{A}\times \mathcal{Y}$ is
\emph{integrable} if, for every $T>0$,
\begin{equation}
\int_{0}^{T}\kappa _{0,s}^{\alpha }\Bigl( \bigl|\widehat{F}(s,c_{s},\alpha
_{s})\bigr|+\bigl|G(s,D_{s}^{\alpha ,\nu },\alpha _{s})\bigr|\Bigr) \,ds
<\infty \qquad \text{$\mathbb{P}$-a.s.}
\label{eq:envelope-integrability}
\end{equation}%
For an integrable triple, the discounted value process
\eqref{eq:discounted_primal_process} and the dual process
\eqref{eq:dual_envelope_def} are real valued, and so is the budget process
\eqref{eq:budget_process}, since $\int_{0}^{T}c_{s}\,ds<\infty $ by
Definition~\ref{defn:2-1} and $Y^{\nu }$ has continuous paths. All three sides of the
bound below are then finite, and the two Fenchel equality conditions can be
stated as a characterization.

\begin{proposition}[Convex duality bound]\label{prop:3-9}
Let
Assumptions~\ref{asm:2-9} and~\ref{asm:3-2} hold and let $\bigl(\alpha ,(\pi ,c),Y^{\nu }\bigr) \in
\mathcal{A}^{\mathrm{disc}}\times \mathcal{A}\times \mathcal{Y}$
be an integrable triple. Then, for every $t\geq 0$,
\begin{equation}  \label{eq:discounted_fenchel_bound}
S_t^{\alpha,\pi,c}-H_t^{\pi,c}(Y^\nu)
\leq\mathcal Z_t^{\alpha,\nu},
\end{equation}
with equality $\mathbb P$-a.s.\ at time $t$ if and only if
\begin{equation}
D_t^{\alpha,\nu}=U_x\bigl(t,X_t^{\pi,c}\bigr),
\label{eq:fenchel_eq_wealth}
\end{equation}
and, for $ds\otimes d\mathbb P$-a.e.\ $(s,\omega )$ with $s\leq t$,
\begin{equation}
\widehat F(s,c_s,\alpha_s)-D_s^{\alpha,\nu}c_s
=G(s,D_s^{\alpha,\nu},\alpha_s).
\label{eq:fenchel_eq_consumption}
\end{equation}
\end{proposition}

\begin{proof}
Fix an integrable triple
$\bigl(\alpha ,(\pi ,c),Y^{\nu }\bigr) \in
\mathcal{A}^{\mathrm{disc}}\times \mathcal{A}\times \mathcal{Y}$ and $t\geq0$. Since
$\kappa _{0,s}^{\alpha }D_{s}^{\alpha ,\nu }=Y_{s}^{\nu }$ for every
$s\geq 0$, the two Fenchel inequalities \eqref{eq:fenchelF} and
\eqref{eq:utility_variational_transform} yield
\begin{equation}
\kappa _{0,t}^{\alpha }\Bigl(U\bigl(t,X_{t}^{\pi ,c}\bigr)-X_{t}^{\pi
,c}D_{t}^{\alpha ,\nu }\Bigr)\leq \kappa _{0,t}^{\alpha }\widetilde{U}%
\bigl(t,D_{t}^{\alpha ,\nu }\bigr),
\label{eq:proof_fenchel_wealth}
\end{equation}%
which is the definition of the wealth conjugate $\widetilde{U}$ at
$D_{t}^{\alpha ,\nu }$, and
\begin{equation}
\kappa _{0,s}^{\alpha }\Bigl(\widehat{F}(s,c_{s},\alpha _{s})-D_{s}^{\alpha
,\nu }c_{s}\Bigr)\leq \kappa _{0,s}^{\alpha }G(s,D_{s}^{\alpha ,\nu },\alpha
_{s}),\qquad s\in \lbrack 0,t],
\label{eq:proof_fenchel_consumption}
\end{equation}%
which is the definition of $G$ at $D_{s}^{\alpha ,\nu }$. Integrating
\eqref{eq:proof_fenchel_consumption} over $[0,t]$ and adding
\eqref{eq:proof_fenchel_wealth} gives
\begin{equation*}
\kappa _{0,t}^{\alpha }U\bigl(t,X_{t}^{\pi ,c}\bigr)+\int_{0}^{t}\kappa
_{0,s}^{\alpha }\widehat{F}(s,c_{s},\alpha _{s})\,ds-X_{t}^{\pi ,c}Y_{t}^{\nu
}-\int_{0}^{t}Y_{s}^{\nu }c_{s}\,ds\leq \kappa _{0,t}^{\alpha }\widetilde{U}%
\bigl(t,D_{t}^{\alpha ,\nu }\bigr)+\int_{0}^{t}\kappa _{0,s}^{\alpha
}G(s,D_{s}^{\alpha ,\nu },\alpha _{s})\,ds,
\end{equation*}%
which yields \eqref{eq:discounted_fenchel_bound}, by
\eqref{eq:budget_process}, \eqref{eq:discounted_primal_process} and
\eqref{eq:dual_envelope_def}.

By \eqref{eq:envelope-integrability}, every term in
\eqref{eq:proof_fenchel_wealth} and in the integral of
\eqref{eq:proof_fenchel_consumption} over $[0,t]$ is finite. Equality in
\eqref{eq:discounted_fenchel_bound} therefore holds if and only if it holds
in \eqref{eq:proof_fenchel_wealth} and, for $ds\otimes d\mathbb{P}$-a.e.\ $(s,\omega )$
with $s\leq t$, in \eqref{eq:proof_fenchel_consumption}.

Equality in \eqref{eq:proof_fenchel_wealth} holds if and only if
$D_{t}^{\alpha ,\nu }=U_{x}(t,X_{t}^{\pi ,c})$, because $x\mapsto U(t,x)$ is
strictly concave and differentiable. Equality in
\eqref{eq:proof_fenchel_consumption} for $ds\otimes d\mathbb{P}$-a.e.\ $(s,\omega )$
with $s\leq t$ holds if and only if $c_{s}$ attains the supremum defining
$G(s,D_{s}^{\alpha ,\nu },\alpha _{s})$, that is, if and only if
$\widehat{F}(s,c_{s},\alpha _{s})-D_{s}^{\alpha ,\nu }c_{s}=G(s,D_{s}^{\alpha
,\nu },\alpha _{s})$. This concludes the proof.
\end{proof}

\medskip

Inequality \eqref{eq:discounted_fenchel_bound} is an upper bound for an
arbitrary discount process $\alpha$ in $\mathcal{A}^{\mathrm{disc}}$. The optimum occurs at $\alpha^\star=-F_u(\cdot,c^\star,U^\star)$ where the Fenchel inequality in $u$ holds, with equality
along the optimal control pair, by \eqref{eq:utility_variational_FOC}.
Consequently, whenever the first integral in
\eqref{eq:envelope-integrability} is finite for every $T>0$, the discounted
process $S^{\alpha^\star,\pi,c}$ is a local supermartingale for every
$(\pi,c)\in\mathcal A$ that also satisfies \eqref{eq:F-integrability}, and it
is a local martingale at the optimum.
Proposition~\ref{prop:3-10} and Theorem~\ref{thm:3-12} make these equalities precise. 

A diagram illustrating the relationship between the various functions introduced
below is presented in Figure~\ref{fig:square}, starting from the primal problem in
the upper left corner and ending at the dual problem in the lower left corner. In
contrast with \citet{MX17}, the two recursive corners are related directly, by the
Legendre transformation in the wealth variable.
 
\begin{figure}[h]
\centering
\begin{tikzpicture}[
  font=\small,
  >={Stealth[length=2.2mm,width=1.7mm]},
  link/.style={->, line width=0.5pt},
  pair/.style={<->, line width=0.5pt, dashed},
  sep/.style={line width=0.4pt, dash pattern=on 0.8pt off 2pt},
  lab/.style={font=\scriptsize, align=center},
  hd/.style={font=\small},
  scale=0.86, every node/.style={transform shape}
]
 
\draw[sep] (4.5,1.55) -- (4.5,0.80);
\draw[sep] (4.5,-0.50) -- (4.5,-4.35);
\draw[sep] (4.5,-5.45) -- (4.5,-6.20);
\draw[sep] (-3.4,-2.5) -- (12.4,-2.5);
 
\node[hd] at (0,1.35)   {Recursive};
\node[hd] at (9,1.35)   {Variational};
\node[hd] at (-2.9,0)   {Primal};
\node[hd] at (-2.9,-5)  {Dual};
 
\node (PR) at (0,0)   {$U(t,x),\ F(t,c,u)$};
\node (PV) at (9,0)   {$U(t,x),\ \widehat F(t,c,\alpha)$};
\node (DR) at (0,-5)  {$\widetilde U(t,y),\ \widetilde F(t,d,u)$};
\node (DV) at (9,-5)  {$\widetilde U(t,y),\ G(t,d,\alpha)$};
 
\draw[link] (PR) -- node[lab, above=2pt]{concave conjugate in $u$\ \ \eqref{eq:utility_variational_identity}} (PV);
\draw[link] (DV) -- node[lab, above=2pt]{convex conjugate in $\alpha$\ \ \eqref{eq:G_saddle}} (DR);
\draw[link] (PV) -- node[lab, right=3pt, pos=0.30]{convex conjugate in $c$\ \ \eqref{eq:G_alpha_definition}} (DV);
 
\draw[pair] (PR) -- node[lab, left=3pt, pos=0.30]{Legendre in $x$\ \ (Prop.~\ref{prop:3-4})\\[1pt]} (DR);
 
\end{tikzpicture}
\caption{Double Fenchel--Legendre transformation.}
\label{fig:square}
\end{figure}
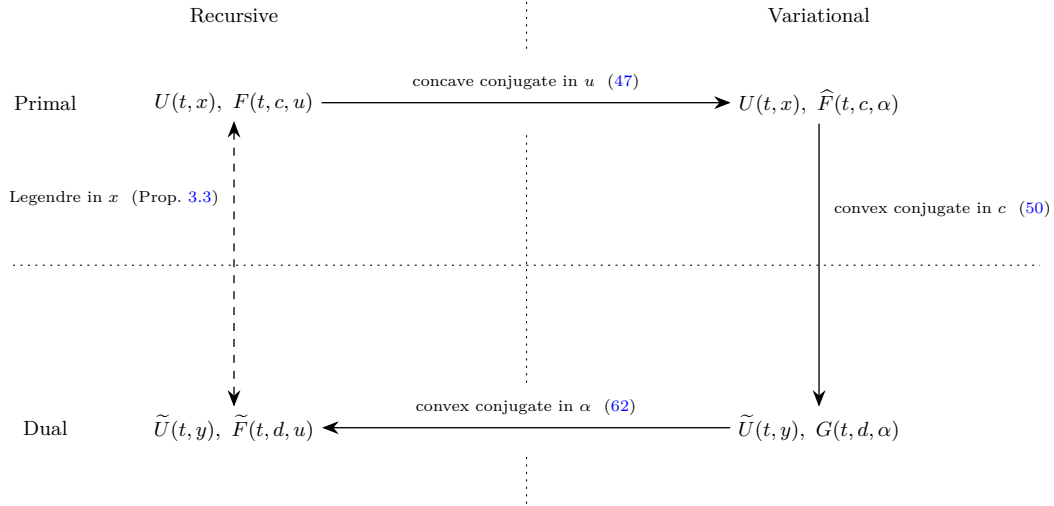
 
\subsection{The state price density generated by the primal optimum}\label{sec:3-5}

We first establish the two equalities in the variational dual bound and then
identify the state price density $Y^\star\in\mathcal{Y}$ generated by
the primal optimum.

For $D>0$, define the feedback discounting map
\begin{equation}  \label{eq:alpha_star_D}
\alpha^\star(t,D) =-F_u\! \left(t,F_c^{-1}(t,D,\hat u(t,D)),\hat
u(t,D)\right),\qquad t\geq0.
\end{equation}

\medskip \noindent A state price density $Y^{\nu }\in \mathcal{Y}$ is said to
\emph{admit a feedback discount} if there exists a progressively measurable
process $\alpha ^{\star }\in \mathcal{A}^{\mathrm{disc}}$ such that
\begin{equation}
D_t=\frac{Y_t^\nu}{\kappa_{0,t}^{\alpha^\star}},
\qquad
\alpha_t^\star
=\alpha^\star(t,D_t),
\qquad t\geq0,
\label{eq:coupled_discount_system}
\end{equation}%
the second identity holding for $dt\otimes d\mathbb{P}$-a.e.\ $(t,\omega )$, and
such that, for every $T>0$,
\begin{equation}  \label{eq:G-integrability}
\int_0^T\kappa_{0,s}^{\alpha^\star}
\bigl|G(s,D_s,\alpha_s^\star)\bigr|\,ds<\infty,
\qquad\text{$\mathbb P$-a.s.}
\end{equation}

\medskip \noindent The first identity in \eqref{eq:coupled_discount_system}
is \eqref{eq:deflated_density_definition} evaluated at $\alpha ^{\star }$, while
the second one evaluates the map \eqref{eq:alpha_star_D} along the deflated
state $D$. Taken together, the two identities constitute a pathwise system
that couples $\alpha ^{\star }$ and $D$. Note that
$D_{0}=Y_{0}^{\nu }=y$, the initial value of the density, since $\kappa
_{0,0}^{\alpha ^{\star }}=1$. Condition \eqref{eq:G-integrability} is
precisely the requirement imposed on the term $G$ in
\eqref{eq:envelope-integrability}. It makes the
finite variation term of \eqref{eq:dual_envelope_def} absolutely convergent,
so that $\mathcal{Z}^{\alpha ^{\star },\nu }$ is a well defined continuous
process. Along the primal optimum, \eqref{eq:G-integrability} is
automatically satisfied, as shown in Step~2 in the proof of Theorem~\ref{thm:3-12}.

\begin{proposition}[Equality in the variational dual bound]\label{prop:3-10}
Let $(U,F)$ satisfy Assumptions~\ref{asm:3-2} and~\ref{asm:3-8}, and
suppose that $\widetilde U$ satisfies the dual forward recursive HJB SPDE
\eqref{eq:dual_SPDE}. Then, the following assertions hold.

\smallskip

\noindent \textit{(i) Pointwise Fenchel equality.} For
$dt\otimes d\mathbb P$-a.e.\ $(t,\omega)$ and every $D>0$, the functions
$\hat u$ and $\alpha^\star$ defined in
\eqref{eq:dual_continuation_level} and \eqref{eq:alpha_star_D} satisfy
\begin{equation}  \label{eq:G_saddle}
G\bigl(t,D,\alpha^\star(t,D)\bigr) =\widetilde F\bigl(t,D,\hat u(t,D)\bigr) %
+\alpha^\star(t,D)\, \hat u(t,D).
\end{equation}

\smallskip

\noindent \textit{(ii) Dynamics of the process $\mathcal{Z}^{\alpha ^{\star
},\nu }$.} Let $Y^{\nu }\in \mathcal{Y}$ admit a feedback discount $\alpha
^{\star }$. Then,
\begin{equation}  \label{eq:deflated_envelope_dynamics}
d\mathcal{Z}_t^{\alpha^\star,\nu} =dM_t^{\nu} +\kappa_{0,t}^{\alpha^\star}\,
\frac{\bigl|(I_{d}-\sigma_{t}^{+}\sigma_{t})\, \widetilde \delta_y(t,D_t)
+D_t\, \widetilde U_{yy}(t,D_t)\, \nu_t\bigr|^{2}}
{2\, \widetilde U_{yy}(t,D_t)}\,dt,
\quad t\geq0,
\end{equation}
where $M^{\nu }$ is a local $\mathcal{F}_{t}$-martingale. In particular,
$\mathcal{Z}^{\alpha ^{\star },\nu }$ is a local $\mathcal{F}_{t}$%
-submartingale, and it is a local $\mathcal{F}_{t}$-martingale if and only if
\begin{equation}  \label{eq:nu_equality_case}
\nu_t=-\frac{(I_{d}-\sigma_{t}^{+}\sigma_{t})\, \widetilde \delta_y(t,D_t)} {D_t\, \widetilde
U_{yy}(t,D_t)},\qquad \text{for }dt\otimes d\mathbb P\text{-a.e. }(t,\omega ).
\end{equation}
\end{proposition}

\begin{proof}
The proof is given in Appendix~\ref{app:B}.
\end{proof}

\subsubsection{Identification of \texorpdfstring{$Y^\star$}{Y*} along the primal optimum}\label{sec:3-5-1}

We now identify the state price density that the primal
optimum generates, together with the dynamics of the marginal utility along
the optimal wealth.

\begin{theorem}[{The state price density generated by the primal optimum}]\label{thm:3-12}
Let Assumption~\ref{asm:3-2} and the hypotheses of
Theorem~\ref{thm:2-10} hold. Let $X^\star$ be the
corresponding wealth process, and let $U^{\star }$ and $c^{\star }$ be the
forward utility and the consumption processes along it,
\begin{equation*}
U_t^\star=U(t,X_t^\star),
\qquad \text{and}\qquad
c_t^\star=c^\star(t,X_t^\star),\quad t\geq0.
\end{equation*}
Introduce the process
\begin{equation}  \label{eq:nu_star_marginal}
\nu_t^{\star}
=\frac{(I_{d}-\sigma_{t}^{+}\sigma_{t})\,
\delta_x(t,X^\star_t)}{U_x(t,X^\star_t)},\quad t\geq0,
\end{equation}
as well as the
processes
\begin{equation}  \label{eq:alpha_kappa_optimal}
\kappa_t^{\star}=\exp \!
\left(\int_0^{t}F_u(s,c_s^{\star},U_s^{\star})\,ds\right)
\qquad \text{and}\qquad
Y_t^{\star}=\kappa_t^{\star}\,U_x(t,X^\star_t),\qquad t\geq0.
\end{equation}
Then, the following assertions hold.

\smallskip

\noindent \textit{(i) Marginal utility dynamics and state price density.}
The process $\nu ^{\star }$ takes values in
$(\operatorname{Im}\sigma_t^\top)^\perp$, and
\begin{equation}  \label{eq:Ux_SDE}
dU_x(t,X^\star_t)
=-F_u(t,c_t^{\star},U_t^{\star})\,U_x(t,X^\star_t)\,dt
+U_x(t,X^\star_t)\,(-\theta_t+\nu^{\star}_t)^\top dW_t,
\quad t\geq0.
\end{equation}
Furthermore, for every $T>0$,
\begin{equation}  \label{eq:marginal_coeff_integrability}
\int_0^T\!\left(
\bigl|F_u(t,c_t^\star,U_t^\star)\bigr|+|\nu_t^\star|^2
\right)dt<\infty ,
\qquad\text{$\mathbb P$-a.s.}
\end{equation}
Process $Y^\star$, defined in \eqref{eq:alpha_kappa_optimal}, is in
$\mathcal Y$ and satisfies
\begin{equation}  \label{eq:optimal_SPD}
dY_t^{\star}=Y_t^{\star}\,(-\theta_t+\nu_t^{\star})^\top dW_t,
\qquad
Y_0^{\star}=U_x(0,x),
\quad t\geq0.
\end{equation}

\smallskip

\noindent \textit{(ii) Equality in the variational dual bound.} Suppose that Assumption~\ref{asm:3-8} also holds and that the discount process
\begin{equation*}
\alpha_t^\star=-F_u(t,c_t^\star,U_t^\star),
\qquad t\geq0,
\end{equation*}
is in $\mathcal{A}^{\mathrm{disc}}$. Let
$D_t^\star=Y_t^\star/\kappa_t^\star=U_x(t,X_t^\star)$. Then, the pair
$(D^{\star },\alpha ^{\star })$ satisfies the feedback relation
\begin{equation*}
\alpha_t^\star=\alpha^\star(t,D_t^\star),
\qquad \text{for }dt\otimes d\mathbb P\text{-a.e. }(t,\omega ),
\end{equation*}
and, consequently, \eqref{eq:coupled_discount_system} is satisfied along the
optimum. This is a conclusion of the theorem and not an additional
requirement. Finally, the process $\nu ^{\star }$ satisfies \eqref{eq:nu_equality_case}, and the corresponding
process $\mathcal{Z}^{\alpha ^{\star },\nu ^{\star }}$ is a local $\mathcal{F}_t$-martingale.
\end{theorem}

\begin{proof}
The proof is given in Appendix~\ref{app:B}.
\end{proof}

\medskip

\noindent In \citet{EKM13}, where there is no aggregator, the marginal utility
along the optimal wealth is itself a state price density. Recursive
aggregation adds the finite variation term $-F_uU_x$ in \eqref{eq:Ux_SDE}. As
a consequence, the following three processes are distinct and must be
distinguished from each other.

\smallskip

\noindent $\bullet $ $U_{x}(t,X_{t}^{\star })$ is the marginal utility along
the optimal wealth.\smallskip

\noindent $\bullet $ $\kappa ^{\star }$ is the factor that removes the
recursive drift of $U_{x}(\cdot ,X^{\star })$.\smallskip

\noindent $\bullet $ $Y_{t}^{\star }=\kappa _{t}^{\star }U_{x}(t,X_{t}^{\star
})$ is the resulting state price density, with \emph{dual control} $\nu _{t}^{\star }$.

\smallskip

\noindent The process $D^{\star }$ is a deflated dual state and should \emph{not} be
confused with the state price density.

\medskip

\noindent There are two optimization problems here and they should \emph{not} be conflated.
On one hand, the dual forward recursive HJB SPDE selects, for each dual state, the
feedback control $\nu ^{\mathrm{dual}}$ in \eqref{eq:nu_dual_def}, for which
the dual value process is a local $\mathcal{F}_t$-martingale. On the other hand, the primal optimum leads to the state price density
$Y^{\star }=\kappa ^{\star }U_{x}(\cdot ,X^{\star })$. The two controls do not
need to agree in a general nonhomothetic system, since $\nu ^{\mathrm{dual}}$ is
evaluated along its own state process. They do agree when the feedback
\eqref{eq:nu_dual_def} is independent of its state argument, as in the
homothetic Epstein--Zin class of Section~\ref{sec:4}, where the identification follows
easily.

\medskip

\subsection{Admissibility and bounds on the model inputs}\label{sec:3-6}

The admissibility of the optimal feedback controls was assumed in the
verification Theorems~\ref{thm:2-10} and~\ref{thm:3-7}.
We now investigate its connection with explicit bounds on the model inputs. We recall that the starting point is \citet[Theorem~4.2]{EKM13}, where
bounds on the orthogonal volatility characteristics connect the forward utility SPDE
with two solvable SDEs. It was developed for investment and consumption by
\citet{EKHM18}. The orthogonal components of
(A1) below reproduce the bounds of \citet[Theorem~4.2]{EKM13}. Our first bound
is imposed on the full vector $\delta_x$ and is, therefore, stronger than its
counterpart therein. Its projected component is needed below to control the
admissibility of the candidate portfolio process.

We now make the following important observation. In the absence of an
aggregator, the marginal utility along the optimal wealth process
is itself a state price density. Here Lemma~\ref{lem:3-11} shows that
$U_x(\cdot,X^\star)$ has an additional drift $-F_uU_x$. The two flows are
therefore connected through the deflated process $D^\star$, while
$Y^\star=\kappa^\star U_x(\cdot,X^\star)$ is recovered only after applying
the recursive discount. This explains the two additional sets of
bounds. Condition (A2) controls the admissibility of the consumption and portfolio,
whereas (A3) controls the drift generated by the recursive discount. Neither
condition is present in \citet{EKM13}.
\smallskip

We distinguish the \emph{portfolio volatility} $\sigma^{\pi,%
\star}_t(x)=\sigma_t^\top \pi^\star(t,x)$ from the \emph{wealth volatility}
\begin{equation}  \label{eq:sigma_star_def}
\sigma^\star(t,x) = x\, \sigma_t^\top \pi^\star(t,x) = -\frac{1}{U_{xx}(t,x)%
}\, \sigma_{t}^{+}\sigma_{t}\bigl(U_x(t,x)\theta_t+\delta_x(t,x)\bigr),\qquad t\geq0,\ x>0.
\end{equation}
The variational discount in feedback form on wealth is
\begin{equation*}
\alpha^\star(t,x)=-F_u\bigl(t,c^\star(t,x),U(t,x)\bigr),\qquad t\geq0,\ x>0,
\end{equation*}
which is the map \eqref{eq:alpha_star_D} evaluated at the marginal utility
level $D=U_x(t,x)$. The
deflated dual coefficients are
\begin{equation*}
\mu^D(t,y)=\alpha^\star \bigl(t,I(t,y)\bigr)\,y, \qquad \text{and}\qquad \sigma^D(t,y)=-y\,
\theta_t+(I_{d}-\sigma_{t}^{+}\sigma_{t})\, \delta_x\bigl(t,I(t,y)\bigr),\qquad t\geq0,\ y>0.
\end{equation*}
Furthermore, the deflated dual process solves 
\begin{equation}  \label{eq:deflated_dual_SDE}
dD_t^\star=\mu^D(t,D_t^\star)\,dt+\sigma^D(t,D_t^\star)^\top \,dW_t ,
\qquad t\geq0,
\end{equation}
which is the conjugate form of the marginal utility SDE \eqref{eq:Ux_SDE}, since $%
D^\star=Y^\star/\kappa^\star=U_x(\cdot,X^\star)$.

\begin{assumption}[Bounds on the model inputs]\label{asm:3-13}
There exist nonnegative adapted processes $%
K^0,K^1,K^2,K^3,K^4$ such that, for every $T>0$,
\begin{equation*}
\int_0^T\! \bigl(|\theta_t|^2+(K_t^1)^2+(K_t^2)^2\bigr)\,dt<\infty,
\qquad
\int_0^T\! (K_t^4)^2\bigl(|\theta_t|+K_t^1\bigr)^2\,dt<\infty,
\qquad
\int_0^T\! \bigl(K_t^0+K_t^3\bigr)\,dt<\infty,
\end{equation*}
$\mathbb{P}$-a.s., and such that, for $dt\otimes d\mathbb{P}$-a.e.\
$(t,\omega)$ and every $x>0$, the following inequalities hold:
\begin{itemize}
\item[(A1)] \textit{Volatility characteristics.}
\begin{equation}  \label{eq:input_bounds_delta}
|\delta_x(t,x)|\le K_t^1\,|U_x(t,x)| \qquad\text{and}\qquad |(I_{d}-\sigma_{t}^{+}\sigma_{t})\,
\delta_{xx}(t,x)|\le K_t^2\,|U_{xx}(t,x)|,
\end{equation}
\item[(A2)] \textit{Consumption and relative risk tolerance.}
\begin{equation}  \label{eq:input_bounds_c}
c^\star(t,x)\le K_t^3\,x \qquad\text{and}\qquad \left|\frac{U_x(t,x)}{x\,U_{xx}(t,x)}%
\right|\le K_t^4,
\end{equation}
\item[(A3)] \textit{Recursive discount.}
The map $x\mapsto\alpha^\star(t,x)$ is continuously differentiable and
\begin{equation}  \label{eq:input_bounds_alpha}
|\alpha^\star(t,x)| +\left|\frac{U_x(t,x)}{U_{xx}(t,x)}\,
\partial_x\alpha^\star(t,x)\right| \le K_t^0.
\end{equation}
\end{itemize}
\end{assumption}

We note that no separate Lipschitz condition is imposed on either the deflated
dual coefficients or the primal feedback coefficients. For the dual coefficients,
the reason is the following. At $x=I(t,y)$, it holds that $U_{x}(t,x)=y$, and
thus the first bound in (A1) gives the linear growth of $\sigma ^{D}$ in $y$.
In turn, the bound on $\alpha ^{\star }$ in (A3) gives the linear growth of
$\mu ^{D}$. Differentiating in $y$ and using $I_{y}(t,y)=1/U_{xx}(t,I(t,y))$,
the bounds (A1) on $(I_{d}-\sigma _{t}^{+}\sigma _{t})\delta _{xx}$ and (A3)
on $\partial _{x}\alpha ^{\star }$ give a Lipschitz bound in $y$ on every
finite horizon, in analogy with \citet[Theorem~4.2(i)]{EKM13}. The
deflated dual equation \eqref{eq:deflated_dual_SDE} has a unique strong
solution, and its strict positivity is a consequence of the multiplicative
form that the equation takes after dividing
the coefficients by $y$. This is presented in detail in Step~1 of the proof.

For the optimal wealth SDE, uniqueness is likewise established rather than
assumed. Step~3 of the proof shows that \emph{every} strictly positive
solution of the feedback wealth equation satisfies the same marginal flow
identity \eqref{eq:flow_identity}, and therefore coincides with the process
constructed from the dual flow through the inverse marginal field. It is this
argument, rather than a Lipschitz estimate on $\sigma ^{\star }$ and $%
c^{\star }$, that yields pathwise uniqueness.

The first bound in (A1) concerns
the full gradient volatility and controls both projections at once: the
orthogonal part $(I_{d}-\sigma _{t}^{+}\sigma _{t})\delta _{x}$ drives the
dual diffusion $\sigma ^{D}$, while the projected part $\sigma _{t}^{+}\sigma
_{t}\delta _{x}$ enters the optimal portfolio through %
\eqref{eq:sigma_star_def}. The bound \eqref{eq:input_bounds_c} on $%
|U_{x}/(xU_{xx})|$ controls the
relative risk tolerance and converts the wealth volatility coefficient bound on $%
\sigma _{t}^{\star }$ into a portfolio volatility coefficient bound on $\sigma _{t}^{\pi
,\star }(x)=\sigma ^{\star }(t,x)/x$, which is what Definition~\ref{defn:2-1} requires.

\bigskip

\medskip \noindent Before stating the result, we introduce the two families of
solutions that the theorem below connects, namely, the family of solutions of
the deflated dual equation \eqref{eq:deflated_dual_SDE}, indexed by their
initial dual value $y>0$, and the family of solutions of the optimal wealth
equation, indexed by their initial wealth $x>0$,
\begin{equation*}
y\longmapsto D_{t}^{\star }(y)\text{ \ and \ }x\longmapsto
X_{t}^{\star }(x),\qquad t\geq 0.
\end{equation*}

\begin{theorem}[{Admissibility and primal--dual consistency}]\label{thm:3-14}
Suppose that $(U,F)$ is a normalized forward aggregator
system satisfying Assumptions~\ref{asm:2-9}, \ref{asm:3-2} and~\ref{asm:3-13} and the nondegeneracy
condition \eqref{eq:nondegeneracy}, and that $U$ satisfies the normalized
forward recursive HJB SPDE \eqref{eq:SPDE-norm}. Then, the following
assertions hold.

\begin{itemize}
\item[\textrm{(1)}] \textit{Connection between the two solution families.} For every $y>0$, the
deflated dual SDE \eqref{eq:deflated_dual_SDE} has a unique strictly positive
strong solution $D^\star(y)$ with initial value $y$, defined and finite for all
$t\geq0$, and nondecreasing in its initial value. For $x>0$, the
process
\begin{equation}
X_t^\star(x)
=I\bigl(t,D_t^\star(U_x(0,x))\bigr),\qquad t\geq 0,
\label{eq:Xstar_inverse_marginal_stmt}
\end{equation}%
is the unique strictly positive strong solution, defined and finite for all
$t\geq 0$, of
\begin{equation}  \label{eq:feedback_wealth_SDE}
dX_t^\star
=\bigl(\sigma^\star(t,X_t^\star)^\top\theta_t
-c^\star(t,X_t^\star)\bigr)\,dt
+\sigma^\star(t,X_t^\star)^\top dW_t,\qquad t\geq0,\qquad X_{0}^{\star }=x,
\end{equation}
and it is nondecreasing in $x$. The two solution families satisfy the identity
\begin{equation}  \label{eq:flow_identity}
U_x\bigl(t,X_t^\star(x)\bigr)
=D_t^\star\bigl(U_x(0,x)\bigr),
\qquad x>0,\ t\geq0.
\end{equation}

\item[\textrm{(2)}] \textit{Primal consistency.} For
every $T>0$,
\begin{equation*}
\int_0^T\!\left(c^\star(t,X_t^\star)
+|\sigma_t^{\pi,\star}(X_t^\star)|^2\right)dt<\infty
\qquad\text{$\mathbb P$-a.s.}
\end{equation*}
Hence, the feedback control pair
$\bigl(\pi^\star(t,X_t^\star),c^\star(t,X_t^\star)\bigr)_{t\geq0}$ is in
$\mathcal A$ and the process
$Y^\star=\kappa^\star U_x(\cdot,X^\star)$ is in $\mathcal Y$. Furthermore, by
Theorem~\ref{thm:2-10}, $(U,F)$ is a consistent normalized forward recursive aggregator
system and this control pair is optimal.

\item[\textrm{(3)}] \textit{Dual consistency.} The conjugate field
$\widetilde U$ satisfies the dual forward recursive HJB SPDE \eqref{eq:dual_SPDE}. For every $y>0$, the equation
\begin{equation*}
dY_t=Y_t\bigl(-\theta_t+\nu_t^{\mathrm{dual}}(Y_t)\bigr)^\top dW_t,
\qquad Y_0=y,\qquad t\geq0,
\end{equation*}
has a unique strictly positive strong solution, and the resulting dual
control $\nu ^{\mathrm{dual}}(Y)$ is also in $\mathcal{A}^{\mathrm{dual}}$.
Hence, $(\widetilde U,\widetilde F)$ is dual consistent on $\mathcal Y$.

\item[\textrm{(4)}] \textit{The two Fenchel equalities.} Suppose that
Assumption~\ref{asm:3-8} also holds and that
$\alpha^\star_t=-F_u(t,c^\star_t,U^\star_t)$, $t\geq0$, is in
$\mathcal{A}^{\mathrm{disc}}$. Then, along the optimal control pair, we have
\begin{equation}  \label{eq:primal_dual_optimum}
Y_t^\star=\kappa_t^\star U_x(t,X_t^\star),\qquad t\geq0,
\end{equation}
and we also have the equalities \eqref{eq:fenchel_eq_wealth} and
\eqref{eq:fenchel_eq_consumption}.
\end{itemize}
\end{theorem}

\begin{proof}
The proof is given in Appendix~\ref{app:B}.
\end{proof}

\begin{corollary}\label{cor:3-15}
Under the assumptions of
Theorem~\ref{thm:3-14}(4), the bound
\eqref{eq:discounted_fenchel_bound} applies to every control pair in
$\mathcal{A}$ and every density in $\mathcal{Y}$ forming an integrable triple
with $\alpha ^{\star }$, and it is attained at
$\bigl((\pi^\star,c^\star),Y^\star\bigr)$.
\end{corollary}

\begin{proof}
The assertion follows easily from
Proposition~\ref{prop:3-9} and Theorem~\ref{thm:3-14}(4).
\end{proof}

\medskip

\noindent We stress that Theorem~\ref{thm:3-14} provides a sufficient construction. The
identity \eqref{eq:flow_identity} is necessary for every sufficiently smooth
admissible optimum, but no general necessity claim is made for
Assumption~\ref{asm:3-13}.

\section{Forward Epstein--Zin recursive aggregator systems}\label{sec:4}

This section develops the homothetic Epstein--Zin class. We start with the linear Uzawa benchmark and subsequently turn to the separable Epstein--Zin construction, their viability condition and optimal controls. The state price density and conjugate equation follow next, before the Merton and Heston examples. The notation would only be introduced when it is first used.
Throughout the section the letter $u$ is reserved for the deterministic
wealth utility $u(x)$ of the multiplicative forms below, and the letter $y$
for the dual state, except in the generic driver notation of Remark~\ref{rem:4-21}. Thus, we denote the continuation-utility argument in
$F$ and $\widetilde{F}$ by $v$. 

\subsection*{Forward Uzawa utilities}

\noindent Let
$\beta:[0,\infty)\to(0,\infty)$ be a deterministic locally integrable
function satisfying $\int_0^T\beta_t\,dt<\infty$, for each $T>0$, and introduce the
normalized aggregator on $%
\mathcal{U}=\mathbb{R}$,
\begin{equation*}
F(t,c,v)=\frac{c^{1-\gamma }}{1-\gamma }-\beta _{t}\,v,\quad c>0,\text{ }%
v\in \mathbb{R},\text{ }\gamma >0,\text{ }\gamma \neq 1.
\end{equation*}%
Its Fenchel transform is
\begin{equation*}
\widetilde{F}(t,d,v)=\frac{\gamma }{1-\gamma }\,d^{\,1-\frac{1}{\gamma }%
}-\beta _{t}\,v,\text{ \ }d>0,\;v\in \mathbb{R},
\end{equation*}%
and the optimal feedback consumption control is $c^{\star
}(t,x)=\left( U_{x}(t,x)\right) ^{-1/\gamma }$, $t\geq0$, $x>0$.

\medskip

\noindent In turn, the normalized forward recursive HJB SPDE \eqref{eq:SPDE-norm} takes the form
\begin{equation}
\begin{aligned}
dU(t,x)=\Bigl( &\frac{|U_{x}(t,x)\theta _{t}+\sigma_{t}^{+}\sigma_{t}%
\delta _{x}(t,x)|^{2}}{2\,U_{xx}(t,x)}-\frac{\gamma }{1-\gamma }\,
\left( U_{x}(t,x)\right) ^{1-\frac{1}{\gamma }}+\beta _{t}\,U(t,x)\Bigr) dt\\
&+\delta (t,x)^{\top }dW_{t}, \quad t\geq 0,\ x>0,
\end{aligned}  \label{SPDE-Uzawa}
\end{equation}%
for a suitable initial condition $U(0,x)=u_{0}(x)$, $x>0$.

The corresponding dual forward SPDE, obtained by substituting
$\widetilde{F}(t,d,v)$ above in \eqref{eq:dual_SPDE}, becomes
\begin{equation}
\begin{aligned}
d\widetilde{U}(t,y)&=\left( -\frac{\gamma }{1-\gamma }\,y^{\,1-\frac{1}{%
\gamma }}+\beta _{t}\bigl(\widetilde{U}(t,y)-y\, \widetilde{U}%
_{y}(t,y)\bigr)-\tfrac{1}{2}\,y^{2}\, \widetilde{U}_{yy}(t,y)\,|\theta
_{t}|^{2}\right. \\
&\left. \quad +y\, \theta _{t}^{\top }\widetilde{\delta }_{y}(t,y)+\frac{%
|(I_{d}-\sigma_{t}^{+}\sigma_{t})\, \widetilde{\delta }_{y}(t,y)|^{2}}{2\, \widetilde{U}_{yy}(t,y)}%
\right) dt+\widetilde{\delta }(t,y)^{\top }dW_{t},  \quad t\geq0, ~~ y > 0,
\end{aligned}
\label{eq:dual_SPDE-Uzawa}
\end{equation}%
 with $\widetilde{U}(0,y)=\widetilde{u}_{0}(y)$, $y>0$, the convex conjugate of the
initial field $u_{0}$ of \eqref{eq:initial-condition}. We note that the
discount model input $\beta$ enters \eqref{eq:dual_SPDE-Uzawa} linearly, through the term $\beta
_{t}(\widetilde{U}-y\, \widetilde{U}_{y}).$ However, SPDE~\eqref{eq:dual_SPDE-Uzawa}
\emph{remains nonlinear} through the unspanned volatility term
$|(I_d-\sigma_t^+\sigma_t)\widetilde\delta_y|^2/(2\widetilde U_{yy})$.
It \emph{becomes linear} in the non volatile case considered next.

\medskip In the special case of zero recursive volatility,
$\delta (t,x)\equiv 0$, $t\geq0$, $x>0$, SPDE~\eqref{SPDE-Uzawa} reduces to the PDE with random
coefficients,
\begin{equation} \label{eq:Uzawa_PDE}
U_t(t,x) - \frac{ U_{x}(t,x)^2 |\theta _{t}|^{2}}{2\,U_{xx}(t,x)%
}+\frac{\gamma }{1-\gamma }\, \left( U_{x}(t,x)\right) ^{1-\frac{1}{\gamma }%
}-\beta _{t}\,U(t,x) = 0, \qquad t\geq0,\ x>0.
\end{equation}

\medskip

\noindent Since the aggregator is affine in $y$, the recursive criterion reduces to a
discounted additive one. We note that $\beta$ being deterministic is not needed for the Bellman
SPDE \eqref{SPDE-Uzawa}, but it is needed for the explicit non volatile representation below.
Indeed, if $\beta$ were merely progressively measurable, the heat kernel
formula would involve its future values which, however, may not be
$\mathcal{F}_{t}$-measurable. Taking conditional expectations would restore
the proper measurability but would, in general, introduce a martingale component.

When $\delta\equiv0$, the model falls within the class of forward investment
and consumption criteria studied by
\citet{Kallblad-2016}. Its convex dual satisfies a linear equation which,
after a suitable time rescaling, takes the form of the ill-posed inhomogeneous heat
equation used below. The resulting criterion decomposes into an
infinite horizon Merton type consumption component and a pure forward investment
component. We stress that this linear structure does not extend to the forward
Epstein--Zin recursive case,
whose dual equation remains nonlinear through the term
$\widetilde U-y\widetilde U_y$.

\subsubsection*{The convex conjugate}

Equation \eqref{eq:Uzawa_PDE} is quasilinear in $U$, but,
as seen in \citet{Kallblad-2016}, can be linearized by passing to the dual domain.
Using the conjugacy identities
\eqref{eq:conjugacy_relations}, the non volatile dual field satisfies
$d\widetilde{U}(t,y)=\widetilde{U}_{t}(t,y)\,dt$. Thus,
SPDE~\eqref{eq:dual_SPDE-Uzawa} with $\widetilde{\delta }(t,y)\equiv 0$,
$t\geq0$, $y>0$, becomes
\begin{equation}
\widetilde{U}_{t}(t,y)+\frac{|\theta _{t}|^{2}}{2}\,y^{2}\,\widetilde{U}%
_{yy}(t,y)-\beta _{t}\bigl(\widetilde{U}(t,y)-y\,\widetilde{U}_{y}(t,y)\bigr)+%
\tfrac{\gamma }{1-\gamma }y^{\,1-\frac{1}{\gamma }}=0, \qquad t\geq0,\ y>0,
\label{eq:Uzawa_dual}
\end{equation}%
which is the analogue of equation (18) in \citet{Kallblad-2016}, but incorporating a discount factor. We stress that equation
\eqref{eq:Uzawa_dual} is
\emph{linear} in $\widetilde{U}$.

\bigskip

Next, using a suitable transformation, we convert the problem to the one studied in \citet{Kallblad-2016}. Specifically, let
\begin{equation}
V(t,x)=e^{-\int_{0}^{t}\beta _{s}\,ds}\,U(t,x),\qquad t\geq0,\ x>0.
\label{eq:Uzawa_discount_change}
\end{equation}%
From \eqref{eq:Uzawa_PDE}, it follows directly that $V$ solves the random
HJB equation in \citet{Kallblad-2016} \emph{without}
zeroth order term, namely,
\begin{equation*}
V_{t}(t,x)
-\frac{|\theta _{t}|^{2}}{2}\,
\frac{V_{x}(t,x)^{2}}{V_{xx}(t,x)} +\frac{\gamma}{1-\gamma }
e^{-\frac{1}{\gamma }\int_{0}^{t}\beta _{s}\,ds}
V_{x}(t,x)^{\,1-\frac{1}{\gamma }}=0,
\qquad t\geq0,\ x>0.
\end{equation*}%
Then, for the convex conjugate, we deduce that
\begin{equation*}
\widetilde{U}(t,y)
=e^{\int_{0}^{t}\beta _{s}\,ds}
\widetilde{V}\bigl(t,e^{-\int_{0}^{t}\beta _{s}\,ds}y\bigr),\qquad t\geq0,\ y>0.
\end{equation*}
The field $\widetilde{V}$ satisfies the same linear dual equation as in
\citet{Kallblad-2016} with a source depending on $(t,y)$ only. Specifically,
\begin{equation}
\widetilde{V}_{t}(t,y)+\frac{|\theta _{t}|^{2}}{2}\,y^{2}\,\widetilde{V}%
_{yy}(t,y)+\,\tfrac{\gamma e^{-\frac{1}{\gamma }\int_{0}^{t}\beta _{s}\,ds}}{1-\gamma }y^{\,1-\frac{1}{\gamma }}=0,\qquad t\geq0,\ y>0;
\label{eq:Uzawa_dual_V}
\end{equation}%
see equation (18) in \citet{Kallblad-2016}.

\subsubsection*{Time rescaling and reduction to the ill-posed inhomogeneous heat equation}

Following the arguments developed in \citet{Kallblad-2016}, we introduce the
time change,
\begin{equation}
M_{t}=\int_{0}^{t}\theta _{s}^{\top }dW_{s}\quad \text{and}\quad
\Theta_{t}=\langle M\rangle _{t}=\int_{0}^{t}|\theta _{s}|^{2}\,ds,\qquad t\geq 0.
\label{eq:Uzawa_timechange}
\end{equation}%
In this section, we work under the following standing assumption.

\begin{assumption}\label{asm:4-1}
There exists a deterministic,
locally integrable function $a:[0,\infty)\to(0,\infty)$ such that
\begin{equation}  \label{eq:Uzawa_deterministic_clock}
|\theta_t|^2=a(t),\qquad \text{and}\qquad
\Theta_{t}=\int_0^t a(s)\,ds\uparrow\infty,\qquad \text{as }t\uparrow\infty.
\end{equation}
\end{assumption}

In other words, the ``traded clock'' is deterministic and strictly increasing, although
the direction of $\theta_t$ may only remain progressively measurable. This
restriction, together with $\beta$ being deterministic, ensures that the future
convolution product displayed below does not make the random coefficients anticipative. In
turn, in analogy to the Definition~3.4 of \citet{Kallblad-2016}, we introduce two auxiliary functions $H,H_{c}:\mathbb{R}%
\times \lbrack 0,\infty )\to (0,\infty )$ defined analogously
through the marginals
\begin{equation}
V_{x}\bigl(t,H(y,t)\bigr)=\exp \Bigl(-y+\tfrac{1}{2}\Theta_{t}\Bigr),\qquad \text{and}\qquad
u_{x}^{c}\bigl(t,H_{c}(y,t)\bigr)=\exp \Bigl(-y+\tfrac{1}{2}\Theta_{t}\Bigr),
\label{eq:Uzawa_Hdef}
\end{equation}%
where $u^{c}(t,c)=e^{-\int_{0}^{t}\beta _{s}\,ds}\,c^{1-\gamma }/(1-\gamma )$
is the discounted felicity utility. We also write $%
H(y,t)=h(y,\Theta_{t})$ and $H_{c}(y,t)=|\theta _{t}|^{2}h_{c}(y,\Theta_{t})$, which
defines the pair $(h,h_{c})$ on the traded clock through the inverse of $\Theta$,
$T(\tau)=\Theta^{-1}_{\tau}$. In turn, the inhomogeneous term generated by the
consumption part of the criterion is explicitly given by
\begin{equation}  \label{eq:Uzawa_hc_explicit}
h_c(y,\tau)
=\frac{1}{a(T(\tau))}
\exp\!\left(
\frac{y-\tau/2-\int_0^{T(\tau)}\beta_s\,ds}{\gamma}
\right),
\qquad y\in\mathbb R,\ \tau\geq0.
\end{equation}
In particular, $h_c$ is deterministic under the above assumptions.
Transformation \eqref{eq:Uzawa_Hdef} turns the operator $\tfrac{1}{%
2}|\theta _{t}|^{2}y^{2}\partial _{yy}$ of \eqref{eq:Uzawa_dual_V} into a
constant coefficient Laplacian, with the shift $\tfrac{1}{2}\Theta _{t}$ absorbing the
accompanying first order drift. Working as in Lemma~3.5 in
\citet{Kallblad-2016}, $(H,H_{c})$ satisfy $H_{t}+\tfrac{1}{2}|\theta
_{t}|^{2}H_{yy}+H_{c}=0$. We obtain that the time changed pair $(h,h_{c})$
satisfies the \emph{constant coefficient} ill-posed inhomogeneous heat equation
\begin{equation}
h_{\tau }(y,\tau )+\tfrac{1}{2}\,h_{yy}(y,\tau )+h_{c}(y,\tau )=0,\qquad
\tau\geq0,\ y\in \mathbb{R},  \label{eq:Uzawa_heat}
\end{equation}%
where $\tau =\Theta_{t}$ is the traded time scale and the subscript $\tau $ denotes
differentiation in \emph{that} clock. It is the transformation
$H(y,t)=h(y,\Theta_{t})$, $%
\partial _{t}H=|\theta _{t}|^{2}h_{\tau }$, that removes the market
coefficient $|\theta _{t}|^{2}$, yielding a diffusion coefficient equal to
$\tfrac{1}{2}$.
Because the Uzawa discount was absorbed into the inhomogeneous term $u^{c}$ in
\eqref{eq:Uzawa_discount_change}, it appears only through $%
h_{c}$ and
\eqref{eq:Uzawa_heat} does not contain the zeroth order term. This is the
ill-posed equation (20) in \citet{Kallblad-2016}.

\subsubsection*{Heat kernel convolution and Widder representation}

Let $k(\tau ,z)=(2\pi \tau )^{-1/2}\exp (-z^{2}/2\tau
)$, $\tau \geq 0$, $z\in \mathbb{R}$, be the fundamental solution of
\eqref{eq:Uzawa_heat} in the homogeneous case, with the convention $
k(z)=k(1,z)$, and let $$(k\star h_{c})(y,\tau )=\int _{\tau }^{\infty }\int
_{-\infty }^{\infty }k(s-\tau ,z-y)\,h_{c}(z,s)\,dz\,ds $$ denote its
space-time convolution with the inhomogeneous consumption term. Under
Assumption~\ref{asm:4-1} and the standing assumption that $\beta $ is deterministic,
the function $h_c$ is the deterministic
function in \eqref{eq:Uzawa_hc_explicit}. Assuming that $k\star h_c$ is finite,
we choose a deterministic positive Borel measure $\varpi$ on $[0,\infty)$
satisfying
\begin{equation*}
\int_0^\infty e^{ry}\varpi(dr)<\infty ,
\qquad\text{for each }y\in\mathbb R.
\end{equation*}
By Theorem~3.7 and Corollary~3.8 in \citet{Kallblad-2016}, the time-changed
dual field admits the representation
\begin{equation}
h(y,\tau )=(k\star h_{c})(y,\tau )+h^{\varpi }(y,\tau ),\qquad \text{with }h^{\varpi
}(y,\tau )\text{ equal to}\quad \int_{0}^{\infty }e^{\,ry-\frac{1}{2}r^{2}\tau }\,\varpi (dr),
\label{eq:Uzawa_Widder}
\end{equation}%
with the term $h^{\varpi }$ being the Widder positive measure representation of the
positive solutions of the homogeneous ill-posed heat equation. We stress that
$h$ is the time-changed inverse marginal field, and not the dual utility itself.
The dual utility is recovered by integrating the identity
$-\widetilde V_y=I$. To make this integration valid also when $\varpi$ has
mass at $r=1$, we introduce
\begin{equation*}
K_r(t,y)=
\begin{cases}
\displaystyle
\frac{y^{1-r}\exp\!\bigl(\tfrac12(r-r^2)\Theta_t\bigr)-1}{r-1},&r\neq1,\\[1.2ex]
\displaystyle-\log y-\tfrac12\Theta_t,&r=1.
\end{cases}
\end{equation*}
Then, direct
differentiation gives
\begin{equation}  \label{eq:Uzawa_Widder_derivative}
-\partial_yK_r(t,y)
=y^{-r}\exp\!\bigl(\tfrac12(r-r^2)\Theta_t\bigr).
\end{equation}
Next, we assume in addition that
\begin{equation}  \label{eq:Uzawa_consumption_finiteness}
J_t = \int_t^\infty
\exp\!\left(
-\frac1\gamma\int_0^s\beta_u\,du
-\frac{\gamma-1}{2\gamma^2}(\Theta_s-\Theta_t)
\right)ds<\infty,
\qquad\text{for each }t\geq0.
\end{equation}
Choosing $K_r$ fixes the integration constant in the dual
variable. We easily deduce that the discounted dual
field $\widetilde V$ decomposes as
\begin{equation}
\begin{aligned}
\widetilde{V}(t,y)&=\int_{t}^{\infty }e^{-\frac{1}{\gamma }\int_{0}^{s}\beta
_{u}\,du}\int _{-\infty }^{\infty }\widetilde{f}\Bigl(y\,e^{\sqrt{\Theta_{s}-\Theta_{t}}%
\,z-\frac{1}{2}(\Theta_{s}-\Theta_{t})}\Bigr)k(z)\,dz\,ds \\
&\quad +\int_{0}^{\infty }K_r(t,y)\,\varpi (dr),
\end{aligned}
\label{eq:Uzawa_decomp}
\end{equation}%
where
$$
\widetilde f(d)=\frac{\gamma}{1-\gamma}d^{1-\frac1\gamma},
\qquad d>0,
$$
is the Fenchel transform of $c^{1-\gamma}/(1-\gamma)$. The first term in
\eqref{eq:Uzawa_decomp} equals $\widetilde f(y)J_t$ and is therefore finite
under \eqref{eq:Uzawa_consumption_finiteness}. The original dual field is then
recovered from
$$
\widetilde U(t,y)
=e^{\int_0^t\beta_s\,ds}
\widetilde V\left(t,e^{-\int_0^t\beta_s\,ds}y\right).
$$
This is the analogue of \citet[formula~(24)]{Kallblad-2016}. The definition
of $K_r$ fixes the sign imposed by the conjugacy convention
\eqref{eq:fenchel_U} and also covers the case $r=1$ by continuity.

Under Assumption~\ref{asm:4-1} and the determinism of $\beta $, the representation is
adapted to the underlying filtration $(\mathcal F_t)_{t\geq0}$.
Its first term is the dual value of the infinite horizon Merton type consumption
problem, whereas the Widder term corresponds to a pure forward investment
criterion determined by the measure $\varpi$. Thus, the criterion separates
into a consumption component and a forward investment component.

We note that $\beta$ being a deterministic process is only a sufficient condition for the appropriate measurability requirements. More generally,
both $h_c$ and $h=k\star h_c+h^\varpi$ must be adapted to the time-changed
filtration. Progressive measurability of $\beta$ alone is insufficient
because the convolution $k\star h_c$ depends on future values of the source.
The regularity, range, and martingale conditions of
\citet[Theorem~3.13]{Kallblad-2016} are also required.

\subsection*{Forward Epstein--Zin recursive aggregators}

We now introduce one of the main contributions in our work, which is the
construction and detailed study of the forward analogue of the celebrated
Epstein--Zin recursive utility.

\medskip

\noindent Let $(\gamma ,\eta ,\lambda )$ satisfy
\begin{equation}
\lambda =\frac{1-\gamma }{1-\frac{1}{\eta }},\text{ \ with }\gamma >0,\gamma
\neq 1,\eta >0,\eta \neq 1,  \label{lambda-Epstein}
\end{equation}%
and set $\mathcal{U}_{\gamma }=(1-\gamma )\mathbb{R}_{+}^{\ast }=\{v\in \mathbb{%
R}:(1-\gamma )v>0\}$, the range on which the powers $((1-\gamma )v)^{1-1/\lambda
}$ are defined in $\mathbb R$. In analogy to \citet{Epstein-Zin-1989}, and
its continuous time formulation in \citet{Duffie-Epstein}, we introduce the
\emph{normalized Epstein--Zin recursive aggregator}
\begin{equation}
f(c,v)=\frac{1}{1-\frac{1}{\eta }}\! \left( c^{\,1-\frac{1}{\eta }}\left(
(1-\gamma )v\right) ^{1-\frac{1}{\lambda }}-(1-\gamma )v\right) ,\quad c>0,%
\text{ }v\in \mathcal{U}_{\gamma }.  \label{eq:EZ_aggregator}
\end{equation}%

\medskip \noindent The model parameters $\gamma $ and $\eta $ play distinct
economic roles. Coefficient $\gamma $ is the relative risk aversion of
the homothetic wealth index, whereas $\eta $ is the \emph{elasticity of
intertemporal substitution} (EIS). We state that the composite parameter
$\lambda $ satisfies
\begin{equation*}
\frac{\eta }{\lambda }=\frac{\eta -1}{1-\gamma }\qquad \text{and}\qquad
\lambda -1=\frac{1-\gamma \eta }{\eta -1}.
\end{equation*}%
Thus, $\lambda =1$ if and only if $\eta =1/\gamma $. This is the additive
CRRA boundary, at which the EIS is constrained to be the reciprocal of risk
aversion. At this boundary,
\begin{equation*}
f(c,v)=\frac{c^{1-\gamma }}{1-\gamma }-v,
\end{equation*}%
so the aggregator becomes additive with respect to both the felicity and
continuation utility. Away from this boundary, Epstein--Zin recursion permits attitudes
toward intertemporal substitution and risk to vary independently \citep[see e.g.][]{KP78,Epstein-Zin-1989,Duffie-Epstein}.

Next, we let $\Psi >0$ be an $\mathcal{F}_{t}$-progressively measurable process
and define
\begin{equation*}
F(t,c,v)=\Psi _{t}\,f(c,v),\text{ \  \ }t\geq 0,\text{ }c>0,\text{ }v\in
\mathcal{U}_{\gamma }.
\end{equation*}%
with $f$ as in (\ref{eq:EZ_aggregator}).

\bigskip

\noindent We note that in the classical Epstein--Zin case, the
multiplier $\Psi $ is a constant, whereas here it is allowed to be a progressively
measurable modeling input. We keep the present notation, as all subsequent
calculations rely on it.

The Fenchel transform of $F$ is 
\begin{equation}
\widetilde{F}(t,d,v)=\frac{d^{\,1-\eta }}{\eta -1}\, \left( (1-\gamma
)v\right) ^{\, \eta \left( 1-\frac{1}{\lambda }\right) }\Psi _{t}^{\eta
}-\lambda \Psi _{t}\,v,\qquad t\geq 0,\text{ }d>0,\text{ }v\in \mathcal{U}%
_{\gamma },  \label{eq:EZ_fenchel}
\end{equation}%
and, in turn, the candidate optimal feedback consumption takes the form
\begin{equation*}
c^{\star }(t,x)=\big((1-\gamma )U(t,x)\big)^{\frac{1-\eta \gamma }{1-\gamma }%
}\! \left( \displaystyle \frac{\Psi _{t}}{U_{x}(t,x)}\right) ^{\! \eta },\qquad t\geq0,\ x>0.
\end{equation*}%
Combining the above with the normalized Bellman SPDE \eqref{eq:SPDE-norm}, we obtain
the (normalized) \textit{forward Epstein--Zin Bellman SPDE},
\begin{equation}
\begin{aligned}
dU(t,x)&=\left( \frac{|U_{x}(t,x)\theta _{t}+\sigma_{t}^{+}\sigma_{t}\delta
_{x}(t,x)|^{2}}{2\,U_{xx}(t,x)} 
-\frac{\Psi _{t}^{\eta }}{\eta -1}\,U_{x}(t,x)^{1-\eta }\bigl(%
(1-\gamma )U(t,x)\bigr)^{\eta \left( 1-\frac{1}{\lambda }\right) }\right. \\
&\left. \quad +\lambda \, \Psi _{t}\,U(t,x)\right) dt+\delta (t,x)^{\top
}dW_{t},  \quad t\geq 0,\ x>0,
\end{aligned}
\label{eq:SPDE-EZ}
\end{equation}%
for a suitable initial condition $U(0,x)=u_{0}(x)$, $x>0$.

Using \eqref{eq:EZ_fenchel} and \eqref{eq:dual_SPDE}, we then obtain the \emph{%
dual forward Epstein--Zin recursive SPDE}. Here $\hat u(t,y)$ from
\eqref{eq:dual_continuation_level} is the primal utility expressed in
the dual state. We have
\begin{equation}
\begin{aligned}
d\widetilde{U}(t,y)&=\left( -\frac{\Psi _{t}^{\eta }}{\eta -1}\,y^{1-\eta }%
\bigl((1-\gamma )\hat u(t,y)\bigr)^{\eta \left( 1-\frac{1}{\lambda }%
\right) }+\lambda \, \Psi _{t}\, \hat u(t,y) -\tfrac{1}{2}\,y^{2}\, \widetilde{U}_{yy}(t,y)\,|\theta
_{t}|^{2}+y\, \theta _{t}^{\top }\widetilde{\delta }_{y}(t,y)\right. \\
&\left. \quad +\frac{|(I_{d}-\sigma_{t}^{+}\sigma_{t})\,\widetilde{\delta }%
_{y}(t,y)|^{2}}{2\, \widetilde{U}_{yy}(t,y)}\right) dt+\widetilde{\delta }%
(t,y)^{\top }dW_{t},  \quad t\geq 0,\ y>0,
\end{aligned}
\label{eq:dual_SPDE-EZ}
\end{equation}%
with $\widetilde{U}(0,y)=\widetilde{u}_{0}(y)$, $y>0$, the convex conjugate of $%
u_{0}$.

Unlike the primal Bellman SPDE \eqref{eq:SPDE-EZ}, the dual equation contains a
nonlinear dependence on $\hat u(t,y)$ and, therefore,
couples $\widetilde{U}$ with $\widetilde{U}_{y}$ in a nonlinear way. Because
of this coupling, a direct linearization of
the dual problem is not, in general, possible. In the next section, we
circumvent this difficulty by working with a special but still rather general
class which allows for closed form solution for the dual problem.

\medskip \noindent The forward separable recursive specification below provides a different
reduction. In turn, factoring out the wealth dependence turns the forward
Epstein--Zin SPDE
into a scalar Bernoulli type equation for $\Phi$. Proposition~\ref{prop:4-2} shows
that the power transformation $\Xi=\Phi^{\eta/\lambda}$ converts this
equation into a linear SDE. Thus, we make the important observation that
conjugacy linearizes the non-volatile Uzawa
problem, whereas homothetic separation and a power transformation linearize
the separable forward Epstein--Zin problem.

\subsection{Separable forward Epstein--Zin recursive aggregator systems}\label{sec:4-1}

We consider forward recursive aggregator pairs $(U,F)$ of both the multiplicative
and separable form 
\begin{equation}
U(t,x)=\Phi _{t}u(x)\text{ \ and \ }F(t,c,v)=\Psi _{t}\,f(c,v),\text{ \  \ }%
t\geq 0,\text{ }x>0,\text{ }c>0,\text{ }v\in \mathcal{U}_{\gamma },
\label{eq:separable_form}
\end{equation}%
with
\begin{equation*}
u(x)=\frac{x^{1-\gamma }}{1-\gamma },\text{ \  \ }x>0,\text{ }\gamma
>0,\gamma \neq 1,
\end{equation*}%
where $\Phi$ is a strictly positive It\^{o} process and $\Psi$ is a strictly
positive $\mathcal F_t$-progressively measurable process. We impose no
semimartingale dynamics on $\Psi$. These processes constitute a modeling
input as follows. Specifically, we allow $\Psi$ to be general under mild integrability
conditions. For process $\Phi$, its volatility process $\delta ^{\Phi }$ is a
modeling input while, as we show, its drift is specified via the martingale
and supermartingale conditions in the definition of forward recursive
aggregators.

\medskip \noindent Because of the homothetic structure, the optimal feedback
policies depend on wealth only through the relative consumption
\begin{equation*}
\check{c}_{t}=\frac{c_{t}}{X_{t}^{\pi ,c}},\text{ \  \ }t\geq 0,
\end{equation*}%
the fraction of wealth consumed per unit time. We, therefore, work with this
quantity throughout the analysis of the separable case.

\medskip \noindent We first construct $\Phi $ explicitly from $\Psi $ and
$\delta ^{\Phi }$, and then characterize consistency and determine the optimal
policies. Corollary~\ref{cor:4-8} verifies the conditions needed to apply the results of Section~\ref{sec:3}.

\subsubsection{The processes \texorpdfstring{$\Phi$ and $\Psi$}{Phi and Psi}}\label{sec:4-1-1}

Let $\delta^\Phi$ be a progressively measurable process, modeling the
volatility of process $\Phi$. Let $\varphi>0$ be an initial multiplier (initial condition of $\Phi$), set
\begin{equation*}
m=\frac{\eta}{\lambda},
\end{equation*}
and define the auxiliary processes
\begin{equation}  \label{eq:EZ_chi}
\chi _{t}=\eta \Psi _{t}+\frac{1-\eta }{2\gamma }\,
\bigl|\theta _{t}+\sigma_{t}^{+}\sigma_{t}\delta _{t}^{\Phi }\bigr|^{2}
+\frac{m(m-1)}{2}|\delta _{t}^{\Phi }|^{2},\qquad t\geq0,
\end{equation}
and
\begin{equation}
\mathcal{E}_{t}=\exp \! \left( \int_{0}^{t}\! \left( \eta \Psi _{s}+\frac{%
1-\eta }{2\gamma }\, \bigl|\theta _{s}+\sigma_{s}^{+}\sigma_{s}\delta _{s}^{\Phi }\bigr|^{2}-%
\frac{\eta }{2\lambda }\,|\delta _{s}^{\Phi }|^{2}\right) ds+\frac{\eta }{%
\lambda }\int_{0}^{t}(\delta _{s}^{\Phi })^{\top }dW_{s}\right) ,\qquad t\geq0.
\label{eq:EZ_stochastic_exponential}
\end{equation}

\medskip \noindent We impose the following assumptions:

\smallskip \noindent \textrm{(V1)} for each $T>0$,
$\int_{0}^{T}\Psi _{s}^{\eta }\,ds<\infty $ $\mathbb P$-a.s.;

\smallskip \noindent \textrm{(V2)} for each $T>0$,
$\int_{0}^{T}(\Psi _{s}+|\delta _{s}^{\Phi }|^{2})ds<\infty $ $\mathbb P$-a.s.;

\smallskip \noindent \textrm{(V3)} it holds that
\begin{equation}  \label{eq:EZ_V3_early}
\mathbb P\!\left(
\varphi ^{\eta /\lambda }
-\int_{0}^{t}\Psi _{u}^{\eta }\mathcal{E}_{u}^{-1}\,du>0
\text{ for every }t\geq0
\right)=1.
\end{equation}

Under \textrm{(V2)} and the standing assumption
\eqref{eq:theta}, the exponent in
\eqref{eq:EZ_stochastic_exponential} is finite on every finite horizon, so
$\mathcal E$ is well defined and strictly positive. We note that no Novikov type condition is
needed because $\mathcal E$ is used only as an integrating factor, and not as a
martingale density. Assumption \textrm{(V1)} makes the integral in
\textrm{(V3)} finite on every interval $[0,T]$, $\mathbb P$-a.s., and \textrm{(V3)} is
the global positivity condition needed for the construction below.

Let process $\Phi $ admit the It\^{o} decomposition
\begin{equation}
d\Phi _{t}=\Phi _{t}\! \left( b_{t}^{\Phi }\,dt+(\delta _{t}^{\Phi })^{\top
}dW_{t}\right) ,\qquad \Phi _{0}=\varphi >0,  \label{eq:SDE_Phi_general}
\end{equation}%
with volatility $\delta ^{\Phi }$ being the modeling input presented earlier, and with drift process $b_{t}^{\Phi }$ given by
\begin{equation}
b_{t}^{\Phi }=\frac{1-\gamma }{1-\eta }\, \frac{\Psi _{t}^{\eta }}{\Phi
_{t}^{\, \eta /\lambda }}-\frac{1-\gamma }{2\gamma }\, \bigl|\theta
_{t}+\sigma_{t}^{+}\sigma_{t}\delta _{t}^{\Phi }\bigr|^{2}+\lambda \Psi _{t},\text{ \  \ }t\geq 0,
\label{eq:EZ_bphi}
\end{equation}%
Then, the It\^{o} decomposition of $\Phi $ is
\begin{equation}
d\Phi _{t}=\Phi _{t}\Bigl( \frac{1-\gamma }{1-\eta }\, \frac{\Psi
_{t}^{\eta }}{\Phi _{t}^{\, \eta /\lambda }}-\frac{1-\gamma }{2\gamma }\,%
\bigl|\theta _{t}+\sigma_{t}^{+}\sigma_{t}\delta _{t}^{\Phi }\bigr|^{2}+\lambda \Psi
_{t}\Bigr) dt +\Phi _{t}(\delta _{t}^{\Phi })^{\top }dW_{t},\quad \Phi
_{0}=\varphi >0,\qquad t\geq0.  \label{eq:EZ_SDE_Phi}
\end{equation}%
As anticipated, the form of drift in \eqref{eq:EZ_bphi} is not postulated
but derived. Theorem~\ref{thm:4-4} shows that it is exactly the condition under which
the separable pair $(U,F)$ constitutes a forward recursive aggregator system.

\medskip

\noindent The next result linearizes \eqref{eq:EZ_SDE_Phi} and gives $\Phi $ in closed
form.

\begin{proposition}\label{prop:4-2}
Let $\Phi $ solve \eqref{eq:EZ_SDE_Phi} and set
$\Xi _{t}=\Phi _{t}^{m},$ $t\geq 0.$ Then,

\noindent i) if $\Phi$ is a strictly positive solution of
\eqref{eq:EZ_SDE_Phi}, then, with $m=\eta/\lambda $ as above, $\Xi$ is a
strictly positive solution of the
linear SDE
\begin{equation}
d\Xi _{t}=\Big(-\Psi _{t}^{\eta }+\chi _{t}\, \Xi _{t}\Big)dt+\frac{\eta }{%
\lambda }\, \Xi _{t}(\delta _{t}^{\Phi })^{\top }dW_{t},\text{ \  \ }t\geq 0,
\label{eq:EZ_linear_Y}
\end{equation}%
with $\Xi _{0}=\varphi ^{m}$ and $\chi$ given by \eqref{eq:EZ_chi}.

Conversely, a solution $\Xi$ of \eqref{eq:EZ_linear_Y} produces the
solution $\Phi=\Xi^{1/m}$ of \eqref{eq:EZ_SDE_Phi} only on a stochastic
interval on which $\Xi$ is strictly positive.

\noindent ii) with $\mathcal E$ as in
\eqref{eq:EZ_stochastic_exponential}, the process $\Phi $ is given by
\begin{equation}
\Phi _{t}=\left( \mathcal{E}_{t}\! \left( \varphi ^{\eta /\lambda
}-\int_{0}^{t}\Psi _{u}^{\eta }\, \mathcal{E}_{u}^{-1}\,du\right) \right)
^{\! \lambda /\eta },\qquad 0\leq t<\tau,  \label{eq:EZ_Phi_explicit}
\end{equation}
where
\begin{equation}  \label{eq:EZ_positive_lifetime}
\tau=\inf\left\{t\geq0:
\varphi^{\eta/\lambda}
-\int_0^t\Psi_u^\eta\mathcal E_u^{-1}\,du\leq0
\right\}.
\end{equation}
\end{proposition}

\begin{proof}
We first derive the SDE satisfied by
$\Xi $, and, in turn, solve it explicitly.

\noindent \textit{Step 1. Derivation of SDE for $\Xi $.}
Recalling that $m=\eta /\lambda $ and $\Xi _{t}=\Phi _{t}^{m}$,
It\^{o}'s formula gives
\begin{equation}
\begin{aligned}
d\Xi _{t}
&=m\Phi _{t}^{m-1}\,d\Phi _{t}
+\frac{m(m-1)}{2}\Phi _{t}^{m-2}\,d\langle \Phi \rangle _{t}\\
&=m\Xi _{t}\frac{d\Phi _{t}}{\Phi _{t}}
+\frac{m(m-1)}{2}\Xi _{t}|\delta _{t}^{\Phi }|^{2}\,dt.
\end{aligned}
\label{eq:proof_Xi_Ito}
\end{equation}
From \eqref{eq:EZ_SDE_Phi},
\begin{equation*}
\frac{d\Phi _{t}}{\Phi _{t}}
=\left(\frac{1-\gamma }{1-\eta }\Psi _{t}^{\eta }
\Phi _{t}^{-m}
-\frac{1-\gamma }{2\gamma }
|\theta _{t}+\sigma _{t}^{+}\sigma _{t}\delta _{t}^{\Phi }|^{2}
+\lambda \Psi _{t}\right)dt
+(\delta _{t}^{\Phi })^{\top }dW_{t},
\end{equation*}%
which in view of \eqref{eq:proof_Xi_Ito} yields
\begin{equation}
\begin{aligned}
d\Xi _{t}
&=m\frac{1-\gamma }{1-\eta }\Psi _{t}^{\eta }
\Xi _{t}\Phi _{t}^{-m}\,dt\\
&\quad+\Xi _{t}\left(
-m\frac{1-\gamma }{2\gamma }
|\theta _{t}+\sigma _{t}^{+}\sigma _{t}\delta _{t}^{\Phi }|^{2}
+m\lambda \Psi _{t}
+\frac{m(m-1)}{2}|\delta _{t}^{\Phi }|^{2}
\right)dt\\
&\quad+m\Xi _{t}(\delta _{t}^{\Phi })^{\top }dW_{t}.
\end{aligned}
\label{eq:proof_Xi_before_simplification}
\end{equation}
Using that
$\Xi _{t}\Phi _{t}^{-m}=1$ and the identities
\begin{equation*}
m\frac{1-\gamma }{1-\eta }=-1,\qquad
m\lambda =\eta ,\qquad
-m(1-\gamma )=1-\eta ,
\end{equation*}%
\eqref{eq:proof_Xi_before_simplification} gives
\begin{equation*}
d\Xi _{t}=\bigl(-\Psi _{t}^{\eta }+\chi _{t}\Xi _{t}\bigr)dt
+\frac{\eta }{\lambda }\Xi _{t}(\delta _{t}^{\Phi })^{\top }dW_{t},
\end{equation*}%
where $\chi _{t}$ is as in \eqref{eq:EZ_chi}. This proves
\eqref{eq:EZ_linear_Y}. The converse direction follows easily, provided
$\Xi $ remains strictly positive.

\noindent \textit{Step 2. Explicit solution of the linear equation.}
The process $\mathcal{E}$ in \eqref{eq:EZ_stochastic_exponential} satisfies
\begin{equation*}
\frac{d\mathcal{E}_{t}}{\mathcal{E}_{t}}
=\chi _{t}\,dt+m(\delta _{t}^{\Phi })^{\top }dW_{t},\qquad
\mathcal{E}_{0}=1,
\end{equation*}%
and, therefore,
\begin{equation*}
d\mathcal{E}_{t}^{-1}
=\mathcal{E}_{t}^{-1}
\left((-\chi _{t}+m^{2}|\delta _{t}^{\Phi }|^{2})dt
-m(\delta _{t}^{\Phi })^{\top }dW_{t}\right).
\end{equation*}%
Thus,
\begin{align*}
d(\mathcal E_t^{-1}\Xi_t)
&=\mathcal E_t^{-1}d\Xi_t+\Xi_t d\mathcal E_t^{-1}
+d\langle\mathcal E^{-1},\Xi\rangle_t\\
&=-\mathcal E_t^{-1}\Psi_t^\eta\,dt.
\end{align*}
Using that $\Xi _{0}=\varphi ^{\eta /\lambda }$, we derive
\begin{equation*}
\Xi _{t}=\mathcal{E}_{t}
\left(\varphi ^{\eta /\lambda }
-\int_{0}^{t}\Psi _{u}^{\eta }\mathcal{E}_{u}^{-1}\,du\right).
\end{equation*}%
Finally, since $\Xi_t>0$ for all $0\leq t < \tau$, we conclude that
$$
\Phi_t=\Xi_t^{\lambda/\eta}, \qquad 0\leq t < \tau,\quad \mathbb{P}\text{-a.s.},
$$
which yields \eqref{eq:EZ_Phi_explicit}.
\end{proof}

\medskip \noindent We note that the squared norms in \eqref{eq:EZ_bphi}
and \eqref{eq:EZ_linear_Y} have distinct origins. The first comes from the
Hamiltonian and uses the projected component $\theta _{t}+\sigma_{t}^{+}\sigma_{t}\delta
_{t}^{\Phi }=\sigma_{t}^{+}\sigma_{t}(\theta _{t}+\delta _{t}^{\Phi })$, by \eqref{eq:projections_theta}. The second is the quadratic-variation correction from It%
\^{o}'s formula applied to $\Xi =\Phi ^{m}$ and uses $|\delta
_{t}^{\Phi }|^{2}.$

\bigskip

\medskip \noindent Under \textrm{(V1)}--\textrm{(V3)}, the process $\Phi$
given by \eqref{eq:EZ_Phi_explicit} is well defined and strictly positive
for all $t\geq0$. Hence, $U(t,\cdot)$ takes values in
$\mathcal U_\gamma$, and the pair $(\Phi,\Psi)$ is globally viable.

\subsubsection{Viability and the classical infinite horizon parameter regimes}\label{sec:4-1-2}

The fact that, in the forward setting, the resolution of the problem is treated for every
$\lambda\neq0$ deserves some clarification.
In the classical
infinite horizon Epstein--Zin case, different ranges of the composite
parameter $\lambda$ require different notions of solution and different
techniques. These difficulties do not disappear in the forward
formulation. Rather, they reappear as a global viability requirement, namely condition~(V3).

In this subsection we keep assumptions \textrm{(V1)} and \textrm{(V2)} and treat
\textrm{(V3)} as the condition to be analyzed. Indeed, the trajectories violating \textrm{(V3)} are
included only to identify the range of global viability.

To compare notations with the existing literature, let $\mathsf{R}$ denote the
relative risk aversion and let $\mathsf{S}$
denote the elasticity of intertemporal complementarity, which is the
reciprocal of the EIS. We use the sans serif font for these two parameters in
order to avoid any confusion with the stock price processes $S$ of
\eqref{eq:marketstocks}. The key parameter used in
\citet{HerdegenHobsonJeromeI,HerdegenHobsonJeromeII,HerdegenHobsonJeromeProper}
is
\begin{equation*}
\vartheta=\frac{1-\mathsf{R}}{1-\mathsf{S}}.
\end{equation*}
With $\mathsf{R}=\gamma$ and $\mathsf{S}=1/\eta$, this parameter is exactly the $\lambda$ herein.
Away from the additive boundary $\lambda=1$, the literature on the classical infinite horizon theory distinguishes three regimes. For
$0<\lambda<1$, a utility process is associated uniquely
with a broad class of consumption streams and in the generalized
formulation of \citet{HerdegenHobsonJeromeII} with every nonnegative
consumption stream. \citet{BayraktarLawless2026}
recently extended the infinite horizon analysis in this range to incomplete
markets with stochastic investment opportunities. They give a
variational characterization of the value function and include stochastic
volatility examples. Their analysis is organized by the range of the parameter
$\lambda$, and the difficulties differ across its range. If $\lambda>1$,
uniqueness fails in general, and the economically relevant optimal consumption process is selected
through the notion of properness. When $\lambda<0$, the risk aversion
parameter and the elasticity of intertemporal complementarity lie on opposite
sides of one. Hence, the usual infinite horizon utility index may display the
``utility bubble'' phenomenon emphasized in \citet{HerdegenHobsonJeromeI}. In
the empirically important case $\gamma>1$ and $\eta>1$, \citet{Shigeta2026} showed that the economic sign problem disappears after an
order-equivalent transformation. Existence, however, still requires a suitable
admissible class of consumption streams. The boundary value $\lambda=1$
corresponds to the additive CRRA aggregator.

In the forward setting developed in this section, the range of parameters for
which the variational duality of Section~\ref{sec:3-4} is available should be
distinguished from the range for which the primal forward construction is
available. Appendix~\ref{app:C} shows that the forward recursive
aggregator is convex in its continuation value when $\gamma\eta\geq1$ and
concave if $\gamma\eta\leq1$. The variational results of Section~\ref{sec:3-4} apply
under the convexity condition of Assumption~\ref{asm:3-8}. In the concave case, a
corresponding variational formulation requires the additional conditions
described there and is not used in the primal construction below. When
$\gamma\eta=1$ the aggregator is affine in its continuation value. Hence, the
primal forward construction developed above covers every $\lambda\neq0$,
subject to viability.

The forward formulation reverses the direction in which the preference
structure is specified. Instead of a terminal condition, it prescribes the
initial field
\begin{equation*}
U(0,x)=\varphi \,u_{\gamma }(x),\qquad u_{\gamma }(x)=\frac{x^{1-\gamma }}{%
1-\gamma },\quad x>0,
\end{equation*}%
where $u_{\gamma }$ is the wealth utility $u$ of \eqref{eq:separable_form},
with the exponent displayed for later comparison,
and propagates this field forward. In contrast, the classical lifetime
formulation must select a solution of a backward recursion without a terminal
condition at infinity. Once $(\Psi,\delta^\Phi)$ and $\varphi$ are fixed, the linear
equation for $\Xi=\Phi^{\eta/\lambda}$ has a unique solution up to its first
exit from $(0,\infty)$. This new forward formulation clarifies why the same requirement covers all cases $\lambda\neq0$. However, it does not imply that
the classical infinite horizon existence questions, and the question of which
solution of the backward recursion is to be singled out, have been resolved
for every parameter range. 

There also exists an exact link with the order-equivalent coordinate used in
\citet{Shigeta2026}, as we discuss next. To this end, let
\begin{equation*}
u_{1/\eta}(x)=\frac{x^{1-1/\eta}}{1-1/\eta},\qquad x>0,
\end{equation*}
and consider the increasing transformation
$u_{1/\eta}\circ u_\gamma^{-1}$ to the homothetic forward field. Since
\begin{equation*}
u_\gamma^{-1}\bigl(U(t,x)\bigr)
=\Phi_t^{1/(1-\gamma)}x,\qquad t\geq0,\ x>0,
\end{equation*}
we obtain
\begin{equation}
\bigl(u_{1/\eta}\circ u_\gamma^{-1}\bigr)
\bigl(U(t,x)\bigr)
=\Phi_t^{(1-1/\eta)/(1-\gamma)}u_{1/\eta}(x)
=\Phi_t^{1/\lambda}u_{1/\eta}(x),\qquad t\geq0,\ x>0.
\label{eq:EZ_order_transform_forward}
\end{equation}
Consequently,
\begin{equation*}
\Xi_t=\Phi_t^{\eta/\lambda}
=\bigl(\Phi_t^{1/\lambda}\bigr)^\eta,\qquad t\geq0.
\end{equation*}
The Bernoulli variable used in Proposition~\ref{prop:4-2} is therefore the
$\eta$-th power of the multiplier of the order-equivalent utility index.
Thus, the linearization and the order transformation are two expressions of the
same change of coordinates. Specifically, when $\lambda<0$, if $\Xi$ vanishes in finite time, the
transformed multiplier vanishes, whereas the original multiplier $\Phi=\Xi^{\lambda/\eta}$ diverges.

\paragraph{A ``preference budget'' interpretation of (V3)}
Recall $m=\eta/\lambda$ and set
\begin{equation}
\mathfrak a_t=\Psi_t^\eta\mathcal E_t^{-1},\qquad\text{and}\qquad
B_t=\varphi^m-\int_0^t \mathfrak a_s\,ds,\qquad t\geq0.
\label{eq:EZ_budget_def}
\end{equation}
Then, $\mathfrak a>0$, and Proposition~\ref{prop:4-2} gives
\begin{equation}
\mathcal E_t^{-1}\Phi_t^m=B_t,\qquad t\geq0,
\qquad
dB_t=-\mathfrak a_t\,dt,\qquad
\check c_t^\star=\frac{\mathfrak a_t}{B_t}
=-\frac{d}{dt}\log B_t,\qquad \text{for }dt\otimes d\mathbb P\text{-a.e. }(t,\omega).
\label{eq:EZ_budget_identity}
\end{equation}
Process $B$ is decreasing while
$\mathfrak a$ is the instantaneous rate at which it decreases. This identity
supports a ``preference budget'' interpretation: $B_t$ represents the part of the
initial preference budget $\varphi ^{m}$ that has not yet been consumed by the
aggregator. Condition~(V3) is then exactly the pathwise solvency constraint
\begin{equation}
\int_0^t \mathfrak a_s\,ds<\varphi^m,
\quad\text{for each }t\geq0.
\label{eq:EZ_budget_V3}
\end{equation}

Let
\begin{equation*}
C_\infty=\int_0^\infty \mathfrak a_s\,ds,
\qquad\text{and}\qquad
B_\infty=\varphi^m-C_\infty.
\end{equation*}
There are three possibilities.

\begin{itemize}
\item If $C_\infty>\varphi^m$, then $B$ reaches zero at a finite time. At
that time $\Phi$ exits from its admissible domain and
$\check c^\star$ explodes, and, hence, the forward system is not globally viable. This, however, is not possible under condition (V3).

\item If $C_\infty=\varphi^m$, then $B_t>0$, $t\geq0$ and
$B_t\downarrow0$ as $t\to\infty$. In this critical case,
\begin{equation*}
B_t=\int_t^\infty \mathfrak a_s\,ds,\qquad\text{and}\qquad
\int_0^\infty\check c_t^\star\,dt=\infty.
\end{equation*}
The initial budget is completely ``used'' at infinity.

\item If $C_\infty<\varphi^m$, then
$B_\infty>0$ and a positive recursive preference budget is never used. In
this case,
\begin{equation*}
\int_0^\infty\check c_t^\star\,dt
=\log\frac{\varphi^m}{B_\infty}<\infty.
\end{equation*}
The residual amount $B_\infty$ is a forward preference component which is not
generated by future aggregator intensity.
\end{itemize}

The above trichotomy also makes the role of the initial multiplier $\varphi $
precise. Since $m=\eta/\lambda$ has the same sign as $\lambda$, the solvency inequality
$C_\infty\leq\varphi^m$ changes direction when it is expressed in terms of
$\varphi$. If $\lambda>0$, a larger $\varphi$ provides a larger budget, while if
$\lambda<0$, a larger $\varphi$ provides a smaller budget. In particular,
when $\gamma>1$ and $\eta>1$, then $\lambda<0$ and
\begin{equation*}
\varphi^{\eta/\lambda}
=\varphi^{-|\eta/\lambda|}.
\end{equation*}
A larger multiplier makes $U(0,x)=\varphi u_\gamma(x)$ more negative, and
increases the initial consumption-to-wealth ratio
\begin{equation*}
\check c_0^\star
=\Psi_0^\eta\varphi^{-\eta/\lambda},
\end{equation*}
and reduces the amount of future preference intensity that the system can
sustain. In this regime, $\varphi^{\eta/\lambda}$ may be interpreted as an
initial ``preference budget''.

\paragraph{Finite-time viability versus transversality}
Condition~(V3) is sufficient for the forward random field and the feedback controls to be
well defined on every finite horizon. However, it is not, by itself, a classical
transversality condition. To see this, let $\tau_n$ be a
localizing sequence for the criterion process and let
$Z_s^\star=Z_s^{\pi^\star,c^\star}$, $s\geq0$. Along the optimal control pair, the stopped
local martingale property gives, for $0\leq t\leq T$,
\begin{align}
&U(t\wedge\tau_n,X_{t\wedge\tau_n}^\star)\notag\\
&\quad=
\mathbb E_{t\wedge\tau_n}\left[
U(T\wedge\tau_n,X_{T\wedge\tau_n}^\star)
+\int_{t\wedge\tau_n}^{T\wedge\tau_n}
F\bigl(s,c_s^\star,U(s,X_s^\star)\bigr)\,ds
\right].
\label{eq:EZ_projective_stopped}
\end{align}
For every $T>0$, the martingale optimality principle yields a finite horizon
recursive problem with terminal field $U(T,\cdot)$. For an arbitrary
admissible control pair, equality is replaced by an inequality. Removing localization
requires class-$D$ or uniform integrability type conditions. An infinite horizon
interpretation also requires a transversality condition, which for
$\lambda<0$ is the transformed condition used in \citet{Shigeta2026}. In the
Merton model, $B_\infty=0$ selects the stationary solution, whereas
$B_\infty>0$ yields a nonstationary forward solution. The condition
$B_\infty=0$ expresses exhaustion of the initial preference budget but does
not by itself imply transversality, which also depends on $\mathcal E$, the
optimal wealth, as well as the recursive discount.

\subsubsection{Characterization and optimal policies}\label{sec:4-1-3}

We first derive the pointwise properties of the aggregator and the resulting
optimal relative consumption.

\begin{proposition}[Aggregator and optimal consumption feedback]\label{prop:4-3}
The forward Epstein--Zin recursive aggregator has the following properties.

\smallskip

\noindent \textit{(i)} For every fixed $(t,\omega,v)$ with
$v\in\mathcal U_\gamma$, the map $c\mapsto F(t,\omega,c,v)$ is strictly
increasing, strictly concave, and satisfies the Inada conditions. Hence,
Assumption~\ref{asm:2-9} holds, and its Fenchel transform is \eqref{eq:EZ_fenchel}.

\smallskip

\noindent \textit{(ii)} Under the separable form
\eqref{eq:separable_form}, the optimal feedback consumption is proportional to wealth
\begin{equation}
\check{c}_{t}^{\star }=\frac{c^{\star }(t,x)}{x}=\frac{\Psi _{t}^{\eta }}{%
\Phi _{t}^{\, \eta /\lambda }},\qquad t\geq 0,\ x>0.
\label{eq:EZ_relative_consumption}
\end{equation}
\end{proposition}

\begin{proof}
Fix $t\geq 0$, $\omega \in \Omega $, and
$v\in \mathcal{U}_{\gamma }$. Since $(1-\gamma )v>0$, differentiation gives
\begin{equation*}
F_{c}(t,c,v)
=\Psi _{t}c^{-1/\eta }\bigl((1-\gamma )v\bigr)^{1-1/\lambda }
\end{equation*}%
and
\begin{equation*}
F_{cc}(t,c,v)
=-\frac{\Psi _{t}}{\eta }c^{-1/\eta -1}
\bigl((1-\gamma )v\bigr)^{1-1/\lambda }.
\end{equation*}%
The first derivative is strictly positive while the second one is
strictly negative. Moreover,
\begin{equation*}
\lim_{c\downarrow 0}F_{c}(t,c,v)=\infty ,\qquad\text{and}\qquad
\lim_{c\uparrow \infty }F_{c}(t,c,v)=0,\qquad t\geq0.
\end{equation*}%
The formulas in Appendix~\ref{app:C} show that $F_c$ and $F_u$ are continuous on
$(0,\infty)\times\mathcal U_\gamma$. These properties prove all assertions
concerning the map $c\mapsto F(t,\omega ,c,v)$ and verify the remaining
regularity in Assumption~\ref{asm:2-9}.

Let $d>0$. The unique maximizer in the definition of
$\widetilde{F}(t,d,v)$ solves $F_{c}(t,c,v)=d$. Therefore,
\begin{equation}
c=\left(\frac{\Psi _{t}
\bigl((1-\gamma )v\bigr)^{1-1/\lambda }}{d}\right)^{\eta },\qquad t\geq0.
\label{eq:proof_EZ_consumption_maximizer}
\end{equation}
Substituting \eqref{eq:proof_EZ_consumption_maximizer} into
$F(t,c,v)-dc$ gives
\begin{equation*}
\widetilde{F}(t,d,v)
=\frac{d^{1-\eta }}{\eta -1}
\bigl((1-\gamma )v\bigr)^{\eta (1-1/\lambda )}\Psi _{t}^{\eta }
-\lambda \Psi _{t}v,\qquad t\geq0,
\end{equation*}%
and \eqref{eq:EZ_fenchel} follows.

Next, we evaluate the maximizer at
\begin{equation*}
d=U_{x}(t,x)=\Phi _{t}x^{-\gamma },\qquad \text{and}\qquad
(1-\gamma )U(t,x)=\Phi _{t}x^{1-\gamma },\qquad t\geq0,\ x>0.
\end{equation*}%
Equation \eqref{eq:proof_EZ_consumption_maximizer} yields
\begin{equation*}
c^{\star }(t,x)
=\left(
\frac{\Psi _{t}
(\Phi _{t}x^{1-\gamma })^{1-1/\lambda }}
{\Phi _{t}x^{-\gamma }}
\right)^{\eta }=\Psi _{t}^{\eta }\Phi _{t}^{-\eta /\lambda }
x^{\eta ((1-\gamma )(1-1/\lambda )+\gamma )},\qquad t\geq0,\ x>0.
\end{equation*}%
The definition of $\lambda $ implies
\begin{equation*}
(1-\gamma )(1-1/\lambda )+\gamma =\frac{1}{\eta }.
\end{equation*}%
Consequently,
\begin{equation*}
c^{\star }(t,x)=x\frac{\Psi _{t}^{\eta }}
{\Phi _{t}^{\eta /\lambda }},\qquad t\geq0,\ x>0,
\end{equation*}%
which proves \eqref{eq:EZ_relative_consumption}.
\end{proof}

\begin{theorem}[Separable characterization]\label{thm:4-4}
Let
$(U,F)$ have the separable form \eqref{eq:separable_form}, and $\Phi$
have dynamics as in \eqref{eq:SDE_Phi_general}.

\smallskip \noindent \textit{(i)} The field $U(t,x)=\Phi_tu(x)$ solves the
forward Epstein--Zin recursive SPDE if and only if the drift of $\Phi$ is
\eqref{eq:EZ_bphi}. In this case, the candidate feedback control processes are
\begin{equation}
\pi _{t}^{\star }=\frac{1}{\gamma }(\sigma _{t}\sigma _{t}^{\top })^{-1}
\left( \mu _{t}+\sigma _{t}\delta _{t}^{\Phi }\right),\qquad \text{and}\qquad
\check{c}_{t}^{\star }=\frac{\Psi _{t}^{\eta }}{\Phi _{t}^{\eta/\lambda }},\qquad t\geq0.
\label{eq:EZ_optimal_controls}
\end{equation}

\smallskip \noindent \textit{(ii)} If the drift condition in part~(i) and
(V1)--(V3) hold, then $(U,F)$ is a
normalized forward recursive aggregator system on $\mathcal{U}_{\gamma }$,
the feedback control pair $\bigl(\pi ^{\star },\check{c}^{\star }X^{\star }\bigr)$
is admissible and optimal, and the optimal wealth process is
\begin{equation}
X_{t}^{\star }=x\, \exp \! \left( \int_{0}^{t}\! \left( (\pi _{s}^{\star
})^{\top }\sigma _{s}\theta _{s}-\check{c}_{s}^{\star }-\tfrac{1}{2}|(\pi
_{s}^{\star })^{\top }\sigma _{s}|^{2}\right) ds+\int_{0}^{t}(\pi _{s}^{\star
})^{\top }\sigma _{s}\,dW_{s}\right) ,\text{ \  \ }t\geq 0.
\label{eq:EZ_optimal_wealth}
\end{equation}
\end{theorem}

\begin{proof}
We prove the two assertions in three
steps.

\noindent \textit{Step 1. Reduction of the Bellman SPDE to the equation for
$\Phi $.} For every $t\geq 0$ and $x>0$,
\begin{equation*}
U(t,x)=\Phi _{t}\frac{x^{1-\gamma }}{1-\gamma }.
\end{equation*}%
Differentiation with respect to $x$ gives
\begin{equation}
U_{x}(t,x)=\Phi _{t}x^{-\gamma },\qquad\text{and}\qquad
U_{xx}(t,x)=-\gamma \Phi _{t}x^{-\gamma -1},\qquad t\geq0,\ x>0.
\label{eq:proof_homothetic_derivatives}
\end{equation}
For a fixed $x>0$, the identity
$dU(t,x)=u(x)d\Phi _{t}$ and the decomposition
\eqref{eq:SDE_Phi_general} give
\begin{equation*}
dU(t,x)=U(t,x)b_{t}^{\Phi }\,dt
+U(t,x)(\delta _{t}^{\Phi })^{\top }dW_{t},\qquad t\geq0.
\end{equation*}%
Therefore,
\begin{equation}
b(t,x)=U(t,x)b_{t}^{\Phi },\qquad
\delta (t,x)=U(t,x)\delta _{t}^{\Phi },\qquad
\delta _{x}(t,x)=U_{x}(t,x)\delta _{t}^{\Phi }, \quad t \geq 0, ~ x>0.
\label{eq:proof_homothetic_characteristics}
\end{equation}

We compute the first nonlinear term in SPDE \eqref{eq:SPDE-EZ}. Equations
\eqref{eq:proof_homothetic_derivatives} and
\eqref{eq:proof_homothetic_characteristics} imply
\begin{equation*}
U_{x}(t,x)\theta _{t}
+\sigma _{t}^{+}\sigma _{t}\delta _{x}(t,x)
=\Phi _{t}x^{-\gamma }
\bigl(\theta _{t}+\sigma _{t}^{+}\sigma _{t}\delta _{t}^{\Phi }\bigr), \quad t \geq 0, ~ x>0.
\end{equation*}%
Therefore,
\begin{align}
&\frac{|U_{x}(t,x)\theta _{t}
+\sigma _{t}^{+}\sigma _{t}\delta _{x}(t,x)|^{2}}
{2U_{xx}(t,x)}\nonumber\\
&\quad=-\frac{\Phi _{t}x^{1-\gamma }}{2\gamma }
|\theta _{t}+\sigma _{t}^{+}\sigma _{t}\delta _{t}^{\Phi }|^{2}\nonumber\\
&\quad=-\frac{1-\gamma }{2\gamma }
|\theta _{t}+\sigma _{t}^{+}\sigma _{t}\delta _{t}^{\Phi }|^{2}
U(t,x), \quad t \geq 0, ~ x>0,
\label{eq:proof_homothetic_portfolio_term}
\end{align}
where in the last equality we used
$\Phi _{t}x^{1-\gamma }=(1-\gamma )U(t,x)$.

For the second nonlinear term we work as follows. From
\eqref{eq:proof_homothetic_derivatives},
\begin{align*}
&U_{x}(t,x)^{1-\eta }
\bigl((1-\gamma )U(t,x)\bigr)^{\eta (1-1/\lambda )}\\
&\quad=(\Phi _{t}x^{-\gamma })^{1-\eta }
(\Phi _{t}x^{1-\gamma })^{\eta (1-1/\lambda )}\\
&\quad=\Phi _{t}^{1-\eta /\lambda }x^{1-\gamma }\\
&\quad=(1-\gamma )\Phi _{t}^{-\eta /\lambda }U(t,x), \quad t \geq 0, ~ x>0.
\end{align*}%
The third equality follows from the definition of $\lambda $, which gives
\begin{equation*}
-\gamma (1-\eta )+\eta (1-\gamma )(1-1/\lambda )=1-\gamma .
\end{equation*}%
It, then, follows that
\begin{align}
&-\frac{\Psi _{t}^{\eta }}{\eta -1}
U_{x}(t,x)^{1-\eta }
\bigl((1-\gamma )U(t,x)\bigr)^{\eta (1-1/\lambda )}\nonumber\\
&\quad=\frac{1-\gamma }{1-\eta }
\frac{\Psi _{t}^{\eta }}{\Phi _{t}^{\eta /\lambda }}U(t,x), \quad t \geq 0, ~ x>0.
\label{eq:proof_homothetic_consumption_term}
\end{align}

\noindent Substituting \eqref{eq:proof_homothetic_portfolio_term} and
\eqref{eq:proof_homothetic_consumption_term} into
\eqref{eq:SPDE-EZ}, and comparing the resulting drift with
$b(t,x)=U(t,x)b_{t}^{\Phi }$, gives
\begin{equation*}
U(t,x)b_{t}^{\Phi }
=U(t,x)\left(
\frac{1-\gamma }{1-\eta }
\frac{\Psi _{t}^{\eta }}{\Phi _{t}^{\eta /\lambda }}
-\frac{1-\gamma }{2\gamma }
|\theta _{t}+\sigma _{t}^{+}\sigma _{t}\delta _{t}^{\Phi }|^{2}
+\lambda \Psi _{t}\right),\qquad t\geq0,\ x>0.
\end{equation*}%
Using that $U(t,x)\neq 0$, dividing by $U(t,x)$ gives
\eqref{eq:EZ_bphi}. Every computation above is reversible and we easily deduce the
equivalence in part (i).

\noindent \textit{Step 2. Computation of the optimal feedback controls.}
The two ratios that enter the candidate optimal portfolio feedback control are
\begin{equation*}
-\frac{U_{x}(t,x)}{xU_{xx}(t,x)}=\frac{1}{\gamma }\qquad \text{and}\qquad
\frac{\delta _{x}(t,x)}{U_{x}(t,x)}=\delta _{t}^{\Phi },\qquad t\geq0,\ x>0,
\end{equation*}%
where we used \eqref{eq:proof_homothetic_derivatives} and
\eqref{eq:proof_homothetic_characteristics}. From
$(\sigma _{t}^{\top })^{+}=(\sigma _{t}\sigma _{t}^{\top })^{-1}\sigma _{t}$
and \eqref{eq:feedback} we obtain
\begin{equation*}
\pi ^{\star }(t,x)
=\frac{1}{\gamma }
(\sigma _{t}\sigma _{t}^{\top })^{-1}\sigma _{t}
(\theta _{t}+\delta _{t}^{\Phi })=\frac{1}{\gamma }
(\sigma _{t}\sigma _{t}^{\top })^{-1}
(\mu _{t}+\sigma _{t}\delta _{t}^{\Phi }),\qquad t\geq0,\ x>0.
\end{equation*}%
The second equality follows from the fact that $\sigma _{t}\theta _{t}=\mu _{t}$, $t\geq0$.
The right-hand side is independent of $x$, and it is precisely the portfolio stated in
\eqref{eq:EZ_optimal_controls}. Proposition~\ref{prop:4-3} gives
\begin{equation*}
c^{\star }(t,x)=x\frac{\Psi _{t}^{\eta }}
{\Phi _{t}^{\eta /\lambda }},\qquad t\geq0,\ x>0,
\end{equation*}%
and, hence,
\begin{equation*}
\check{c}_{t}^{\star }=\frac{c^{\star }(t,x)}{x}
=\frac{\Psi _{t}^{\eta }}{\Phi _{t}^{\eta /\lambda }},\qquad t\geq0,\ x>0,
\end{equation*}%
which proves the second formula in \eqref{eq:EZ_optimal_controls}.

\noindent \textit{Step 3. Consistency, admissibility, and optimal wealth.}
Proposition~\ref{prop:4-3} proves Assumption~\ref{asm:2-9}. Formulas
\eqref{eq:proof_homothetic_derivatives} and
\eqref{eq:proof_homothetic_characteristics}, together with the inverse marginal
\begin{equation*}
I(t,y)=\left(\frac{\Phi _{t}}{y}\right)^{1/\gamma },\qquad y>0,
\end{equation*}%
yield the required spatial regularity and the Inada conditions. Since $\Phi$
is a strictly positive It\^{o} process, the homothetic forms of $U$ and $I$
inherit the random-field regularity of Assumption~\ref{asm:3-2}. Conditions (V1) and
(V2) provide the required local integrability of their characteristics. Under
(V1)--(V3), Corollary~\ref{cor:4-8} verifies every bound in Assumption~\ref{asm:3-13}, including
the integrability condition for which $K^{4}=1/\gamma $. Therefore,
Theorem~\ref{thm:3-14} implies that $(U,F)$ is a normalized forward recursive
aggregator system and that the feedback control pair is admissible and optimal.

It remains to solve the optimal wealth equation. Since
$c_{t}^{\star }=X_{t}^{\star }\check{c}_{t}^{\star }$, equation
\eqref{eq:wealth} becomes
\begin{equation*}
\frac{dX_{t}^{\star }}{X_{t}^{\star }}
=\left((\pi _{t}^{\star })^{\top }\sigma _{t}\theta _{t}
-\check{c}_{t}^{\star }\right)dt
+(\pi _{t}^{\star })^{\top }\sigma _{t}\,dW_{t},\qquad t\geq0.
\end{equation*}%
Applying It\^{o}'s formula gives
\begin{equation*}
d\log X_{t}^{\star }
=\left((\pi _{t}^{\star })^{\top }\sigma _{t}\theta _{t}
-\check{c}_{t}^{\star }
-\frac{1}{2}|(\pi _{t}^{\star })^{\top }\sigma _{t}|^{2}\right)dt
+(\pi _{t}^{\star })^{\top }\sigma _{t}\,dW_{t},\qquad t\geq0.
\end{equation*}%
and \eqref{eq:EZ_optimal_wealth} follows.
\end{proof}

\medskip \noindent Combining Theorem~\ref{thm:4-4} with the explicit form of $\Phi $
obtained in Proposition~\ref{prop:4-2} yields the optimal consumption in closed form.

\begin{corollary}\label{cor:4-5}
Assume that the drift condition %
\eqref{eq:EZ_bphi} and (V1)--(V3) hold. Then, the optimal relative
consumption rate is given by
\begin{equation*}
\check{c}_{t}^{\star }=\Psi _{t}^{\eta }\, \mathcal{E}_{t}^{-1}\! \left(
\varphi ^{\eta /\lambda }-\int_{0}^{t}\Psi _{u}^{\eta }\, \mathcal{E}%
_{u}^{-1}\,du\right) ^{\!-1},\text{ \  \ }t\geq 0,
\end{equation*}%
and the optimal consumption process is $c_{t}^{\star }=X_{t}^{\star }\,%
\check{c}_{t}^{\star }.$ In particular, (V3) is the one that guarantees that $\check{c}^{\star }$ is
finite and positive.
\end{corollary}

\begin{remark}\label{rem:4-6}
The optimal portfolio $\pi ^{\star }$
in \eqref{eq:EZ_optimal_controls} has the same structure as the one in the classical
homothetic setting, but depends on the volatility $\delta ^{\Phi }$. We remind
the reader that the latter is an exogenous modeling input rather than a
quantity determined by the solution of the problem in a fixed
horizon. The optimal relative consumption $\check{c}^{\star }$ depends
nonlinearly on both $\Phi $ and $\Psi .$
\end{remark}

\begin{remark}\label{rem:4-7}
Theorem~\ref{thm:4-4} reduces the forward
recursive problem to the one-dimensional equation \eqref{eq:EZ_SDE_Phi} for $%
\Phi $ and to the viability conditions (V1)--(V3). Indeed, the strict
positivity of $\Phi $ ensures that $U(t,\cdot )\in \mathcal{U}_{\gamma },$
while (V1) and (V2) yield the integrability of the model inputs and, through
Corollary~\ref{cor:4-8} below, the admissibility of the feedback control pair. These are the
properties established in Section~\ref{sec:3}, for general setting.
\end{remark}

\medskip \noindent The final step is to verify the conditions of Section~\ref{sec:3}
directly from the model processes $(\Psi ,\delta ^{\Phi })$. No integrability
beyond (V1)--(V3) is required.

\begin{corollary}\label{cor:4-8}
Under the separable
representation \eqref{eq:separable_form} with the drift condition
\eqref{eq:EZ_bphi} and (V1)--(V3), Assumption~\ref{asm:3-13} holds and
Theorem~\ref{thm:3-14} applies. In particular, the feedback control pair in
\eqref{eq:EZ_optimal_controls} is admissible.
\end{corollary}

\begin{proof}
Using that
\begin{equation*}
U_{x}(t,x)=x^{-\gamma }\Phi _{t},\qquad U_{xx}(t,x)=-\gamma \,x^{-\gamma
-1}\Phi _{t},\qquad \delta _{x}(t,x)=U_{x}(t,x)\, \delta _{t}^{\Phi },\qquad t\geq0,
\end{equation*}%
the bounds in Assumption~\ref{asm:3-13} hold with
\begin{equation*}
K_{t}^{1}=K_{t}^{2}=|\delta _{t}^{\Phi }|,\qquad K_{t}^{3}=\check{c}%
_{t}^{\star }=\frac{\Psi _{t}^{\eta }}{\Phi _{t}^{\, \eta /\lambda }},\qquad
K_{t}^{4}=\frac{1}{\gamma },\qquad K_{t}^{0}=|\alpha _{t}^{\star }|,\qquad t\geq0,
\end{equation*}%
where the discount $\alpha^{\star}$ satisfies
\begin{equation*}
\alpha _{t}^{\star }=-F_{u}(t,c_{t}^{\star },U_{t}^{\star })=\lambda \Psi
_{t}-(\lambda -1)\, \frac{\Psi _{t}^{\eta }}{\Phi _{t}^{\, \eta /\lambda }},
\qquad t\geq0,
\end{equation*}%
We first note that $\partial_x\alpha^\star\equiv0$. Moreover, the optimal
feedback controls are linear in the state and are given by
\begin{equation*}
\sigma^\star(t,x)
=
\frac{x}{\gamma}\sigma_t^+\sigma_t
\bigl(\theta_t+\delta_t^\Phi\bigr),
\qquad
c^\star(t,x)=\check c_t^\star x,
\qquad t\geq0,\ x>0.
\end{equation*}
It remains to check the corresponding integrability conditions. Let $T>0$.
By the standing market condition \eqref{eq:theta} on the market price of
risk and \textrm{(V2)},
\begin{equation*}
\int_0^T
\bigl(
|\theta_t|^2+|\delta_t^\Phi|^2+\Psi_t
\bigr)\,dt
<\infty,
\qquad \mathbb P\text{-a.s.}
\end{equation*}
Hence,
\begin{equation*}
\int_0^T (K_t^4)^2
\bigl(|\theta_t|+K_t^1\bigr)^2\,dt
\leq
\frac{2}{\gamma^2}
\int_0^T
\bigl(
|\theta_t|^2+|\delta_t^\Phi|^2
\bigr)\,dt
<\infty.
\end{equation*}

We now consider the feedback consumption control. To this end, set
\begin{equation*}
B_t
=
\varphi^{\eta/\lambda}
-
\int_0^t
\Psi_u^\eta\mathcal E_u^{-1}\,du,\qquad t\geq0,
\end{equation*}
which is the quantity multiplying $\mathcal E_t$ in the explicit formula
\eqref{eq:EZ_Phi_explicit} for $\Phi $. By
Corollary~\ref{cor:4-5},
\begin{equation*}
\check c_t^\star
=
-\frac{d}{dt}\log B_t,\qquad \text{for }dt\otimes d\mathbb P\text{-a.e. }(t,\omega).
\end{equation*}
Therefore,
\begin{equation*}
\int_0^T \check c_t^\star\,dt
=
\log\left(
\frac{\varphi^{\eta/\lambda}}{B_T}
\right)
<\infty,
\qquad \mathbb P\text{-a.s.},
\end{equation*}
where the last inequality follows from \textrm{(V3)}. Finally,
\begin{equation*}
\int_0^T |\alpha_t^\star|\,dt
\leq
|\lambda|\int_0^T\Psi_t\,dt
+
|\lambda-1|
\int_0^T\check c_t^\star\,dt
<\infty,\qquad \text{for each }T\geq0.
\end{equation*}
Thus, the required integrability conditions hold on every finite horizon.
\end{proof}

\subsubsection{Zero-volatility separable forward Epstein--Zin utilities}\label{sec:4-1-4}

When $\delta ^{\Phi }\equiv 0$, the process $\Phi$ has no martingale part,
although it does not need to be monotone without further assumptions. The stochastic exponential %
\eqref{eq:EZ_stochastic_exponential} is of finite variation,
\begin{equation*}
\mathcal{E}_{t}=\exp \! \left( \int_{0}^{t}\! \left( \eta \Psi _{s}+%
\displaystyle \frac{1-\eta }{2\gamma }\,|\theta _{s}|^{2}\right) ds\right) ,%
\text{ \ }t\geq 0,
\end{equation*}%
and \eqref{eq:EZ_Phi_explicit} gives
\begin{equation}
\Phi _{t}=\left( \mathcal{E}_{t}\! \left( \varphi ^{\eta /\lambda
}-\int_{0}^{t}\Psi _{u}^{\eta }\, \mathcal{E}_{u}^{-1}\,du\right) \right)
^{\! \lambda /\eta },\text{ \  \ }t\geq 0.  \label{eq:EZ_Phi_time_monotone}
\end{equation}%
Direct calculations yield the optimal policies and their associated wealth process in closed
form. Specifically, using that $\delta ^{\Phi }\equiv 0$, the optimal portfolio becomes
\begin{equation*}
\pi _{t}^{\star }=\frac{1}{\gamma }\,(\sigma _{t}\sigma _{t}^{\top })^{-1}\mu
_{t},\text{ \  \ }t\geq 0,
\end{equation*}%
and the optimal consumption
\begin{equation*}
c_{t}^{\star }=\frac{\mathcal{E}_{t}^{-1}\, \Psi _{t}^{\eta }}{\varphi ^{\eta
/\lambda }-\int_{0}^{t}\Psi _{u}^{\eta }\, \mathcal{E}_{u}^{-1}\,du}%
\,X_{t}^{\star },\text{ \  \ }t\geq 0.
\end{equation*}%
We easily deduce that the related optimal wealth process is
\begin{equation*}
X_{t}^{\star }=x\, \exp \! \left( \int_{0}^{t}\! \left( (\pi _{s}^{\star
})^{\top }\sigma _{s}\theta _{s}-\check{c}_{s}^{\star }-\tfrac{1}{2}|(\pi
_{s}^{\star })^{\top }\sigma _{s}|^{2}\right) ds+\int_{0}^{t}(\pi _{s}^{\star
})^{\top }\sigma _{s}\,dW_{s}\right) ,\text{ \  \ }t\geq 0.
\end{equation*}

\begin{remark}\label{rem:4-9}
The time monotone specification in the forward
Epstein--Zin recursive class we are developing is
the natural forward analogue of the well established time monotone homothetic
constructions in the forward performance literature. The optimal
consumption-to-wealth ratio is the ratio of $\Psi _{t}^{\eta }$ to the
integrated factor $\mathcal{E}_{t}(\varphi ^{\eta /\lambda
}-\int_{0}^{t}\Psi _{u}^{\eta }\mathcal{E}_{u}^{-1}\,du).$
\end{remark}

\medskip \noindent When $\delta ^{\Phi }\equiv 0$ the hedging part is eliminated. The portfolio is then governed by
risk aversion alone through the multiplier $1/\gamma $, while the EIS
affects the timing of consumption through the evolution of $\Phi $ and the
ratio $\Psi ^{\eta }/\Phi ^{\eta /\lambda }$. This is the setting in which
the separation between risky investment and intertemporal substitution is
most transparent.

\subsection{State price density}\label{sec:4-2}

\begin{corollary}\label{cor:4-10}
Under the assumptions of
Theorem~\ref{thm:4-4}(ii), the recursive discount rate satisfies
\begin{equation}
\begin{aligned}
F_{u}(t,c_{t}^{\star },U_{t}^{\star })
&=\widetilde{F}_{u}\bigl(t,U_{x}(t,X_{t}^{\star }),U(t,X_{t}^{\star })\bigr)\\
&=(\lambda -1)\check{c}_{t}^{\star }-\lambda\Psi_t\\
&=(\lambda -1)\frac{\Psi_t^\eta}{\Phi_t^{\eta/\lambda}}-\lambda\Psi_t,\qquad t\geq0.
\end{aligned}
\label{eq:EZ_Ftildey}
\end{equation}%
The state price density \eqref{eq:alpha_kappa_optimal} generated by the
primal optimum becomes 
\begin{equation}
Y_{t}^{\star }=U_{x}(t,X_{t}^{\star })\, \exp \! \left( -\int_{0}^{t}\left(
\lambda \Psi _{s}+(1-\lambda )\, \displaystyle \frac{\Psi _{s}^{\eta }}{\Phi
_{s}^{\eta /\lambda }}\right) ds\right) ,\qquad t\geq0.  \label{eq:EZ_pricing_identity}
\end{equation}
\end{corollary}

\begin{remark}\label{rem:4-11}
In \eqref{eq:EZ_pricing_identity},
the exponent of the discount factor separates into the $\Psi$-driven
time-discount term $-\lambda \Psi $ and the correction term
$(\lambda-1)\check{c}^\star$. The latter vanishes in the
additive case $\lambda=1$ but, otherwise, depends on the optimal
consumption-to-wealth ratio. The sign of the correction follows that of $%
\lambda-1$. Since $\lambda =1$ is equivalent to $\eta =1/\gamma $, this
correction is precisely how the separation between the EIS and risk aversion
is reflected in the pricing kernel. When $\lambda=1$, it disappears, and the state price
density reduces to the marginal utility discounted only by the $\Psi $-term. Away from
this case, the optimal consumption also enters the recursive discount. In the homothetic representation, we have $U_{x}(t,X_{t}^{\star })=\Phi _{t}%
(X_{t}^{\star })^{-\gamma }$, and, thus, $\gamma $ governs the direct exposure of
marginal utility to wealth, whereas $\eta $ affects the consumption
policy and the recursive discount. 
\end{remark}

\begin{remark}[Comparison with Matoussi and Xing]\label{rem:4-12}
The identity \eqref{eq:alpha_kappa_optimal} has
the same structural form as the utility-gradient formula of
\citet{MX17} for the classical (backward) Epstein--Zin problem in a finite
horizon $[0,T]$. In their notation,
in which $f$ denotes the aggregator itself and therefore plays the role of
$F$ herein, the state price density selected by optimality satisfies
\citep[cf.][Corollary~3.7, Equation~(3.11)]{MX17}
\begin{equation*}
D_{t}^{\star }\; \propto \;f_{c}(c_{t}^{\star
},U_{t}^{c^{\star }})\exp \! \left( \int_{0}^{t}f_{u}(c_{s}^{\star
},U_{s}^{c^{\star }})\,ds\right) ,\qquad t\in[0,T].
\end{equation*}%
At the optimum, the first order condition gives $f_{c}(c^{\star },U^{\star
})=U_{x}$ while the envelope identity yields $f_{u}(c^{\star },U^{\star
})=F_{u}=\widetilde{F}_{u}$. Thus, in both formulas, the state price density is given by the marginal utility multiplied by the exponential of the derivative of the aggregator with respect to $u$.

The backward and forward pricing kernels are nonetheless different objects.
Specifically, in the backward problem the discount factor $\exp (\int f_{u}\,ds)$
is evaluated along the solution of a BSDE pinned by a given, specific terminal
condition. In contrast, in the forward setting, it is evaluated along the
forward SDE \eqref{eq:EZ_SDE_Phi} for $
\Phi $, whose drift is determined by the market itself through self-generation. 
\end{remark}

\subsection{Duality in forward Epstein--Zin recursive aggregator systems}\label{sec:4-3}

The homothetic structure of $U(t,x)$ reduces the argument $\widetilde{U%
}(t,y)-y\, \widetilde{U}_{y}(t,y)$ inside $\widetilde{F}$ to a scalar
multiple of $\widetilde{U}$. Then, the dual value process of Section~\ref{sec:3-3}
becomes a self-contained Epstein--Zin recursive problem in
$\widetilde{U}$, with an explicit dual aggregator $g$ and the
dual-consistency property of Definition~\ref{defn:3-5}.

\begin{proposition}[forward Epstein--Zin dual factor]\label{prop:4-13}
Let
$U(t,x)=\Phi _{t}\,u(x)$, $x>0$, $t\geq0$, be as in \eqref{eq:separable_form}, with $\Phi $ as
in \eqref{eq:SDE_Phi_general} and with drift $b^\Phi$ satisfying
\eqref{eq:EZ_bphi}, or equivalently, let $U$ solve \eqref{eq:SPDE-EZ}. Set
$\widetilde\Phi_t=\Phi_t^{1/\gamma}$, $t\geq0$.

\smallskip

\noindent \textit{(i) Conjugate field.} The forward recursive convex conjugate
is of separable form,
\begin{equation}
\widetilde{U}(t,y)=\frac{\gamma}{1-\gamma}
\widetilde\Phi_t y^{(\gamma-1)/\gamma},\qquad t\geq0,\ y>0.
\label{eq:homothetic_dual_form}
\end{equation}%
Moreover, \eqref{eq:dual_continuation_level} becomes
\begin{equation}
\hat u(t,y)=\frac{1}{\gamma}\widetilde U(t,y),\qquad t\geq0,\ y>0.
\label{eq:homothetic_identity}
\end{equation}%

\smallskip

\noindent \textit{(ii) Dual aggregator.} The dual forward Epstein--Zin equation
depends on $\widetilde U$ only, with generator
\begin{equation}
g(t,y,v)
=\widetilde{F}\! \left( t,y,\frac{v}{\gamma }\right)=\frac{y^{1-\eta }}{\eta -1}
\left( \frac{1-\gamma }{\gamma }v\right) ^{\eta(1-1/\lambda )}
\Psi _{t}^{\eta }-\frac{\lambda }{\gamma }\Psi _{t}v,\qquad t\geq0,\ y>0,\ (1-\gamma)v>0.
\label{eq:EZ_dual_aggregator}
\end{equation}%

\smallskip

\noindent \textit{(iii) Dual-factor dynamics.} The factor process
$\widetilde\Phi$ satisfies
\begin{equation}
d\widetilde{\Phi }_{t}=\widetilde{\Phi }_{t}\bigl(b_{t}^{\widetilde{\Phi }%
}\,dt+(\delta _{t}^{\widetilde{\Phi }})^{\top }dW_{t}\bigr),\qquad \delta
_{t}^{\widetilde{\Phi }}=\frac{1}{\gamma }\, \delta _{t}^{\Phi },\qquad t\geq0,
\label{eq:dual_factor_SDE}
\end{equation}%
with 
\begin{equation}
b_{t}^{\widetilde{\Phi }}
=\frac{\lambda }{\gamma }\Psi _{t}
-\frac{1-\gamma }{\gamma (\eta -1)}
\Psi _{t}^{\eta }\widetilde{\Phi }_{t}^{-\gamma \eta/\lambda } -\frac{1-\gamma }{2\gamma ^{2}}|\theta _{t}|^{2}
-\frac{1-\gamma }{\gamma }
\theta _{t}^{\top }\delta _{t}^{\widetilde{\Phi }} +\frac{1-\gamma }{2}
|(I_{d}-\sigma_{t}^{+}\sigma_{t})\delta _{t}^{\widetilde{\Phi }}|^{2},\qquad t\geq0.
\label{eq:dual_factor_drift}
\end{equation}
\end{proposition}

\begin{proof}
Fix $t\geq 0$ and $y>0$. The function that is maximized in the definition of
$\widetilde{U}(t,y)$ is
\begin{equation*}
x\longmapsto \Phi _{t}\frac{x^{1-\gamma }}{1-\gamma }-xy,\qquad x>0.
\end{equation*}%
Since it is strictly concave, its unique maximizer is
\begin{equation*}
I(t,y)=\left(\frac{\Phi _{t}}{y}\right)^{1/\gamma },\qquad t\geq0,\ y>0.
\end{equation*}%
Thus,
\begin{align*}
\widetilde{U}(t,y)
&=\Phi _{t}\frac{I(t,y)^{1-\gamma }}{1-\gamma }-yI(t,y)\\
&=\frac{\Phi _{t}^{1/\gamma }y^{(\gamma -1)/\gamma }}{1-\gamma }
-\Phi _{t}^{1/\gamma }y^{(\gamma -1)/\gamma }\\
&=\frac{\gamma }{1-\gamma }
\Phi _{t}^{1/\gamma }y^{(\gamma -1)/\gamma },\qquad t\geq0,\ y>0,
\end{align*}%
and \eqref{eq:homothetic_dual_form} follows. Differentiating the last
expression with respect to $y$ gives
\begin{equation*}
\widetilde{U}_{y}(t,y)
=-\Phi _{t}^{1/\gamma }y^{-1/\gamma },\qquad t\geq0,\ y>0,
\end{equation*}%
and, thus,
\begin{align*}
\widetilde{U}(t,y)-y\widetilde{U}_{y}(t,y)
&=\left(\frac{\gamma }{1-\gamma }+1\right)
\Phi _{t}^{1/\gamma }y^{(\gamma -1)/\gamma }\\
&=\frac{1}{\gamma }\widetilde{U}(t,y),\qquad y>0,\ t\geq0,
\end{align*}%
and \eqref{eq:homothetic_identity} follows. Substituting
\begin{equation*}
\hat u(t,y)=\widetilde{U}(t,y)-y\widetilde{U}_{y}(t,y)
=\frac{1}{\gamma }\widetilde{U}(t,y),\qquad t\geq0,\ y>0,
\end{equation*}%
into \eqref{eq:EZ_fenchel} gives, for every $v$ in the range of
$\widetilde{U}$,
\begin{align*}
g(t,y,v)
&=\widetilde{F}\left(t,y,\frac{v}{\gamma }\right)\\
&=\frac{y^{1-\eta }}{\eta -1}
\left(\frac{1-\gamma }{\gamma }v\right)^{\eta (1-1/\lambda )}
\Psi _{t}^{\eta }
-\frac{\lambda }{\gamma }\Psi _{t}v,\qquad t\geq0,\ y>0,
\end{align*}%
and \eqref{eq:EZ_dual_aggregator} follows.

It remains to compute the dynamics of
$\widetilde{\Phi }_{t}=\Phi _{t}^{1/\gamma }$. It\^{o}'s formula gives
\begin{equation}
\frac{d\widetilde{\Phi }_{t}}{\widetilde{\Phi }_{t}}
=\left(\frac{b_{t}^{\Phi }}{\gamma }
+\frac{1}{2\gamma }\left(\frac{1}{\gamma }-1\right)
|\delta _{t}^{\Phi }|^{2}\right)dt
+\frac{1}{\gamma }(\delta _{t}^{\Phi })^{\top }dW_{t},
\label{eq:proof_dual_factor_Ito}
\end{equation}
and, thus,
\begin{equation*}
\delta _{t}^{\widetilde{\Phi }}=\frac{1}{\gamma }
\delta _{t}^{\Phi },\qquad t\geq0.
\end{equation*}%
To identify the drift, we substitute \eqref{eq:EZ_bphi} into
\eqref{eq:proof_dual_factor_Ito}, and we use
\begin{equation*}
\Phi _{t}^{-\eta /\lambda }
=\widetilde{\Phi }_{t}^{-\gamma \eta /\lambda }\qquad \text{and}\qquad
\delta _{t}^{\Phi }=\gamma \delta _{t}^{\widetilde{\Phi }},\qquad t\geq0,
\end{equation*}%
together with \eqref{eq:projections_theta}. Expanding
the square gives
\begin{align*}
&-\frac{1-\gamma }{2\gamma ^{2}}
|\theta _{t}+\gamma \sigma _{t}^{+}\sigma _{t}
\delta _{t}^{\widetilde{\Phi }}|^{2}
+\frac{1-\gamma }{2}|\delta _{t}^{\widetilde{\Phi }}|^{2}\\
&\quad=-\frac{1-\gamma }{2\gamma ^{2}}|\theta _{t}|^{2}
-\frac{1-\gamma }{\gamma }
\theta _{t}^{\top }\delta _{t}^{\widetilde{\Phi }}
+\frac{1-\gamma }{2}
|(I_{d}-\sigma _{t}^{+}\sigma _{t})
\delta _{t}^{\widetilde{\Phi }}|^{2},\qquad t\geq0.
\end{align*}%
Combining this identity with the remaining terms and
\eqref{eq:proof_dual_factor_Ito} yields
\eqref{eq:dual_factor_drift}, and the drift of $\widetilde\Phi $ is fully
specified.
\end{proof}

\begin{remark}\label{rem:4-14}
The generator $g(t,y,v)=\widetilde{F%
}(t,y,v/\gamma )$ in \eqref{eq:EZ_dual_aggregator} is the forward
self-generating analogue of the finite-horizon dual recursive generator of
\citet{MX17}. The difference is that here the dual factor $%
\widetilde{\Phi }$ evolves forward in time through \eqref{eq:dual_factor_SDE}
and \eqref{eq:dual_factor_drift}, while the analogous object in \citet{MX17}
is fully specified by an exogenous terminal condition. The state price density generated by the primal optimum in
Theorem~\ref{thm:3-12} is $%
Y_{t}^{\star }=\widetilde{\Phi }_{t}^{\, \gamma }\, \kappa _{t}^{\star
}\,(X_{t}^{\star })^{-\gamma }=\Phi _{t}\, \kappa _{t}^{\star
}\,u_{x}(X_{t}^{\star })$, $t\geq0$. In this homothetic class, the feedback control
\eqref{eq:nu_dual_def} is independent of $y$, and the density
$Y^{\nu^{\mathrm{dual}}}$ of Theorem~\ref{thm:3-7}, with an exogenous initial condition
$Y_0^{\nu^{\mathrm{dual}}} = U_x(0,x)$, coincides with $Y^\star$.
\end{remark}

\begin{remark}[Bernoulli variable under duality]\label{rem:4-15}
The Bernoulli variable is unchanged by conjugacy, in that
$\widetilde \Phi_t^{\gamma\eta/\lambda}
=\Phi_t^{\eta/\lambda}=\Xi_t$, $t\geq0$. Thus, the primal and dual reductions use the
same linear state variable.
\end{remark}

\subsection{The Black and Scholes market: explicit closed form solutions}\label{sec:4-4}

Consider a market with constant coefficients: $\mu \in \mathbb{R}^{n}$ and $%
\sigma \in \mathbb{R}^{n\times d}$ of full rank. The market price of risk is
then constant $\bar{\theta}$, given by
\begin{equation*}
\bar{\theta}=\sigma ^{\top }(\sigma \sigma ^{\top })^{-1}\mu \in
\mathbb{R}^{d},\qquad \theta _{t}\equiv \bar{\theta},\qquad t\geq0.
\end{equation*}

\bigskip \noindent Fix a constant $\bar{\Psi}>0$ and set $\Psi _{t}\equiv
\bar{\Psi}$ and $\delta _{t}^{\Phi }\equiv \bar{\delta}^{\Phi }=0$, $t\geq 0$.

\noindent Furthermore, let the constant
\begin{equation*}
\bar\chi =\eta \bar{\Psi}+\displaystyle \frac{1-\eta }{2\gamma }\,|\bar{\theta}%
|^{2},
\end{equation*}%
(see
Proposition~\ref{prop:4-2}). Then,
representation \eqref{eq:EZ_Phi_time_monotone} takes, on the positive lifetime
interval $0\leq t<\tau$ of \eqref{eq:EZ_positive_lifetime}, the following form:

i) if $\bar\chi \neq 0$,  
\begin{equation}
\Phi _{t}=\left( e^{\bar\chi t}\varphi ^{\eta /\lambda }-\displaystyle \frac{%
\bar{\Psi}^{\eta }}{\bar\chi }\big(e^{\bar\chi t}-1\big)\right) ^{\! \lambda /\eta
},\qquad 0\leq t<\tau,  \label{eq:Merton_Phi}
\end{equation}%

ii) if $\bar\chi =0$,%
\begin{equation*}
\Phi _{t}=\big(\varphi ^{\eta /\lambda }-\bar{\Psi}^{\eta }t\big)^{\lambda
/\eta },\qquad 0\leq t<\tau.
\end{equation*}

\bigskip In the case where $\bar\chi >0$ and $\varphi ^{\eta
/\lambda }\geq \bar{\Psi}^{\eta }/\bar\chi $, the optimal
investment and relative consumption controls are
\begin{equation}
\pi _{t}^{\star }=\displaystyle \frac{1}{\gamma }\,(\sigma \sigma ^{\top })^{-1}\mu \qquad \text{and}\qquad \check{c}_{t}^{\star }=\displaystyle \frac{\bar{\Psi}%
^{\eta }}{e^{\bar\chi t}\varphi ^{\eta /\lambda }+\tfrac{\bar{\Psi}^{\eta }}{%
\bar\chi }\,(1-e^{\bar\chi t})},\qquad t\geq0.  \label{eq:Merton_feedback}
\end{equation}
We give more information on the origin of the two inequalities below.

\begin{remark}\label{rem:4-16}
The optimal portfolio process $\pi ^{\star }$ in %
\eqref{eq:Merton_feedback} is constant and resembles the classical Merton
proportion determined by the relative risk aversion $\gamma $ and the market
price of risk $\bar{\theta}$. In contrast, the optimal consumption rate is
not constant, but rather a deterministic function that depends on the
various parameters $\lambda ,\eta ,\varphi ,\bar{\Psi}$ and $\bar\chi $. We
stress, however, that the forward and the backward Merton problems do not
coincide for an arbitrary value of the initial multiplier. The critical
initial value identified below removes the residual forward component and
produces the stationary Merton candidate. Other viable
initial values generate non stationary forward consumption paths.
This benchmark gives the clearest direct separation of
the two Epstein--Zin preference parameters. Specifically, risk aversion determines the
risky allocation through the factor $1/\gamma $, whereas the EIS affects the
consumption path through $\eta $ and the composite parameter $\lambda $.
The coefficient $\bar\chi $ defined above is not an additional preference
parameter and should not be confused with the positive Uzawa discount
$\beta$ used earlier.
Indeed, it is the growth coefficient of the linearized factor $\Xi =\Phi ^{\eta
/\lambda }$. It aggregates the consumption input, investment opportunities,
risk aversion, and the EIS. Consequently, even in this simple market setting, the
optimal portfolio exhibits a transparent, direct dependence on $\gamma $, while the level
and dynamics of the optimal consumption process reflect both parameters through the
forward consistency condition.
\end{remark}

\subsubsection{Infinite horizon viability of the pair \texorpdfstring{$(\Phi,\Psi)$}{(Phi,Psi)}}\label{sec:4-4-1}

\noindent Using the notations of Section~\ref{sec:4-1}, we have
\begin{equation}
C_\infty=\int_0^\infty \mathfrak a_t\,dt
=\begin{cases}
\displaystyle \frac{\bar\Psi^\eta}{\bar\chi},&\bar\chi>0,\\
+\infty,&\bar\chi\leq0.
\end{cases}
\label{eq:Merton_total_preference_cost}
\end{equation}
It follows that (V3) holds on $[0,\infty)$ if and only if
\begin{equation}
\bar\chi>0
\quad\text{and}\quad
\varphi^{\eta/\lambda}\geq
\frac{\bar\Psi^\eta}{\bar\chi}.
\label{eq:Merton_global_viability}
\end{equation}

\noindent The three budget regimes of Section~\ref{sec:4-1} now have an explicit form. Specifically, if
\begin{equation*}
\varphi^{\eta/\lambda}<
\frac{\bar\Psi^\eta}{\bar\chi},
\end{equation*}
then $\Xi_t$ reaches zero in finite time. Consequently, $\Phi_t$ decreases
to zero if $\lambda>0$ and increases to $+\infty$ if $\lambda<0$. If, on the other hand,
\begin{equation}
\varphi^{\eta/\lambda}=
\frac{\bar\Psi^\eta}{\bar\chi},
\label{eq:Merton_stationary_selection}
\end{equation}
then
\begin{equation*}
B_t=\frac{\bar\Psi^\eta}{\bar\chi}e^{-\bar\chi t}
\qquad \text{and}\qquad
\Xi_t=\frac{\bar\Psi^\eta}{\bar\chi},\qquad t\geq0,
\end{equation*}
and therefore
\begin{equation*}
\Phi_t=
\left(\frac{\bar\Psi^\eta}{\bar\chi}\right)^{\lambda/\eta},
\qquad
\check c_t^\star=\bar\chi,\qquad t\geq0.
\end{equation*}
The factor process $\Phi $ and the optimal relative consumption process are
both stationary. Finally, if
\begin{equation*}
\varphi^{\eta/\lambda}>
\frac{\bar\Psi^\eta}{\bar\chi},
\end{equation*}
then
\begin{equation*}
B_\infty
=\varphi^{\eta/\lambda}
-\frac{\bar\Psi^\eta}{\bar\chi}>0.
\end{equation*}
In this case, the forward aggregator is globally viable, but it retains a residual forward preference
component.
$\check c_t^\star$ decreases to zero and the factor $\Phi_t$ is
nonstationary.

As noted in the general case above, the dependence on the initial multiplier changes with the sign of
$\lambda$. If $\lambda>0$, condition
\eqref{eq:Merton_global_viability} imposes a lower bound on $\varphi$, while, if
$\lambda<0$, it imposes an upper bound. Thus, in the range
$\gamma>1$, $\eta>1$, and $\lambda<0$ studied by
\citet{Shigeta2026}, increasing $\varphi$ moves the forward system toward
a finite time preference insolvency rather than away from such an insolvency.

\subsubsection{Relation with the classical (backward) Merton solution}\label{sec:4-4-2}

The optimal portfolio in \eqref{eq:Merton_feedback} resembles the classical
Merton proportion. The forward construction allows every initial multiplier
satisfying \eqref{eq:Merton_global_viability}, whereas the critical value
\eqref{eq:Merton_stationary_selection} yields the stationary Merton
candidate. At this value, $\Phi$ and
$\check c^\star$ are constant and $B_\infty=0$. A strict inequality in
\eqref{eq:Merton_global_viability}, namely, a strictly positive residual
preference budget, yields a
nonstationary forward solution with $B_\infty>0$. Thus, condition~(V3)
imposes initial viability rather than terminal selection and is not itself a
transversality condition.

\subsubsection{The role of the volatility \texorpdfstring{$\delta ^{\Phi }$}{delta^Phi}}\label{sec:4-4-3}

While $(\mu,\sigma,\bar\theta)$ are constant market parameters, the modeling input $%
\delta^\Phi_t\not \equiv0$, $t\geq0$, turns the preference factor $\Phi$ into a random process and, as a result,
the optimal portfolio obtains the hedging correction $\pi^\star_t-\frac{1}{%
\gamma}(\sigma \sigma^{\top})^{-1}\mu =\frac{1}{\gamma}(\sigma
\sigma^{\top})^{-1}\sigma \, \delta^\Phi_t$, $t\geq0$, through the
$\sigma_{t}^{+}\sigma_{t}\delta_x$ term in
the Hamiltonian. The Heston model below provides one of many possible ways to specify the modeling input $\delta^\Phi$ based on a particular market structure. Specifically, setting $\Phi_t=\bar\varphi(\Gamma_t)$ yields a non-trivial
$\delta_t^\Phi$ (as it follows from It\^{o}'s formula) and, in turn, a
state-dependent hedging component in the optimal portfolio process.

\subsection{Heston stochastic volatility}\label{sec:4-5}

We conclude with an incomplete market model with Heston type stochastic
volatility. To this end, we let $W=(W^{S},W^{\perp })$ be an $(n+1)$%
-dimensional Brownian motion on a filtered probability space $\left( \Omega
,\mathcal{F},(\mathcal{F}_{t})_{t\geq 0},\mathbb{P}\right) $.

The variance factor $\Gamma $ follows the Cox--Ingersoll--Ross SDE 
\begin{equation}
d\Gamma _{t}=b(\bar{\kappa}-\Gamma _{t})\,dt+a\sqrt{\Gamma _{t}}\bigl(\rho ^{\top }dW_{t}^{S}+\sqrt{1-|\rho |^{2}}\,dW_{t}^{\perp }\bigr),\qquad t\geq0,\qquad \Gamma _{0}>0,  \label{eq:Heston_CIR}
\end{equation}%
with $b,\bar{\kappa},a>0$, $\rho \in \mathbb{R}^{n}$, $|\rho |\leq 1$, satisfying the Feller condition 
\begin{equation}
2b\bar{\kappa}\geq a^{2},  \label{Feller}
\end{equation}
where $\Gamma_0$ is deterministic.

The price processes of the risky assets $S=\left( S_{1},...,S_{n}\right) $ satisfy 
\begin{equation}
dS_{t}=\mathrm{diag}(S_{t})\bigl(\sigma _{0}\,q\, \Gamma _{t}\,dt+\sqrt{%
\Gamma _{t}}\, \sigma _{0}\,dW_{t}^{S}\bigr),  \label{eq:Heston_S}
\end{equation}%
with $\sigma _{0}\in \mathbb{R}^{n\times n}$ being invertible and $q\in \mathbb{R}%
^{n}$. 

The asset volatility matrix is then given by $\sigma _{t}=\sqrt{\Gamma _{t}}%
\,(\sigma _{0},0)\in \mathbb{R}^{n\times (n+1)}$, the market price of risk $%
\theta $ satisfies 
\begin{equation}
\theta _{t}=\sigma _{t}^{\top }(\sigma _{t}\sigma _{t}^{\top })^{-1}\mu _{t}=%
\sqrt{\Gamma _{t}}\, \bar{\theta},\qquad \bar{\theta}=(q,0)^{\top }\in 
\mathbb{R}^{n+1},\qquad |\theta _{t}|^{2}=\Gamma _{t}\,|q|^{2},
\label{eq:Heston_theta}
\end{equation}%
and the traded subspace projection is $\sigma _{t}^{+}\sigma _{t}=\mathrm{diag%
}(I_{n},0)$, so that $I_{n+1}-\sigma _{t}^{+}\sigma _{t}$ projects onto the
unspanned direction $W^{\perp }$.

We choose a twice continuously differentiable function $\bar{\varphi}%
:(0,\infty )\rightarrow (0,\infty ),$ to be determined below, and a constant
$\bar{\Psi}>0,$ and introduce
\begin{equation}
U(t,x)=\bar{\varphi}(\Gamma _{t})u(x)\qquad \text{and}\qquad F(t,c,u)=\bar{\Psi}%
\,f(c,u),\qquad t\geq0,\ x>0,  \label{eq:Heston_separable_stationary}
\end{equation}%
with $u(x)=x^{1-\gamma }/(1-\gamma )$ and $f$ as in \eqref{eq:EZ_aggregator}.

A direct application of It\^{o}'s formula gives 
\begin{equation}
\begin{aligned}
b_{t}^{\Phi }
&=\frac{b(\bar{\kappa}-\Gamma _{t})\bar{\varphi}^{\prime}(\Gamma _{t})
+\tfrac{1}{2}a^{2}\Gamma _{t}\bar{\varphi}^{\prime \prime}(\Gamma _{t})}
{\bar{\varphi}(\Gamma _{t})},\qquad \text{and}\\
\delta _{t}^{\Phi }
&=\frac{a\sqrt{\Gamma _{t}}\, \bar{\varphi}^{\prime }(\Gamma _{t})}
{\bar{\varphi}(\Gamma _{t})}
\bigl(\rho ,\sqrt{1-|\rho |^{2}}\bigr)^{\top },\qquad t\geq0.
\end{aligned}
\label{eq:Heston_bphi_deltaphi}
\end{equation}%
Substituting \eqref{eq:Heston_bphi_deltaphi} and \eqref{eq:Heston_theta}
into the scalar drift condition \eqref{eq:EZ_bphi} of Theorem~\ref{thm:4-4} and using
\begin{equation*}
|\sigma_{t}^{+}\sigma_{t}\delta _{t}^{\Phi }|^{2}
=\left(\frac{a\bar{\varphi}^{\prime }(\Gamma_t)}
{\bar{\varphi}(\Gamma_t)}\right)^{2}
\Gamma _{t}|\rho |^{2},\qquad t\geq0,
\end{equation*}
we eliminate the common factor $\bar{\varphi}(\Gamma _{t})$ and obtain the
stationary singular ODE
\begin{equation}
\begin{aligned} &\frac{a^2 x}{2}\, \bar \varphi''(x) +\Bigl(b(\bar \kappa-x)
+\tfrac{(1-\gamma)\,a\,q^\top \rho}{\gamma}\,x\Bigr)\bar \varphi'(x)
+\frac{(1-\gamma)a^2|\rho|^2 x}{2\gamma}\, \frac{\bar \varphi'(x)^2}{\bar
\varphi(x)}\\ &\hspace{2cm}
+\Bigl(\tfrac{(1-\gamma)|q|^2}{2\gamma}\,x-\lambda \bar \Psi \Bigr)\bar
\varphi(x) \;=\; \frac{1-\gamma}{1-\eta}\, \bar \Psi^\eta \, \bar
\varphi(x)^{1-\eta/\lambda}. \end{aligned}  \label{eq:Heston_ODE_phi}
\end{equation}%
The substitution $\bar{h}=\bar{\varphi}^{\, \alpha _{\rho }}$ with $\alpha
_{\rho }=1+(1-\gamma )|\rho |^{2}/\gamma $ transforms %
\eqref{eq:Heston_ODE_phi} into the semilinear ODE 
\begin{equation}
\begin{aligned}
&\frac{a^{2}x}{2}\, \bar{h}^{\prime \prime }(x)+\Bigl(b(\bar{\kappa}-x)+\tfrac{%
(1-\gamma )\,a\,q^{\top }\rho }{\gamma }\,x\Bigr)\bar{h}^{\prime }(x)+\alpha
_{\rho }\Bigl(\tfrac{(1-\gamma )|q|^{2}}{2\gamma }\,x-\lambda \bar{\Psi}%
\Bigr)\bar{h}(x) =\frac{\alpha _{\rho }(1-\gamma )}{1-\eta }\, \bar{\Psi}^{\eta }\,
\bar{h}(x)^{1-\eta /(\alpha _{\rho }\lambda )}.
\end{aligned}
\label{eq:Heston_ODE_h}
\end{equation}

\noindent The affine regime below is inspired by the analysis of
\citet{SeiferlingSeifried2016}.

\begin{assumption}[Affine forward Epstein--Zin regime]\label{asm:4-17}
The model parameters satisfy $\gamma >1,$ $q\neq 0,$
\begin{equation}
\frac{\eta }{\lambda }=\alpha _{\rho }=\frac{\gamma +(1-\gamma )|\rho |^{2}}{%
\gamma },  \label{eq:Heston_affine_link}
\end{equation}%
\begin{equation}
\widehat{\kappa }=b-\frac{(1-\gamma )\,a\,q^{\top }\rho }{\gamma }>0,\qquad
\widehat{c}=\alpha _{\rho }\, \frac{(1-\gamma )\,|q|^{2}}{2\gamma },
\label{eq:Heston_affine_struct}
\end{equation}%
and the Feller condition \eqref{Feller}.
\end{assumption}

\noindent Assumption~\ref{asm:4-17} identifies an affine solvability regime for the
stationary ODE \eqref{eq:Heston_ODE_phi}. Note that we make no claim that it represents all possible solutions for this problem. Since $\gamma >1$ and $|\rho
|\leq 1$, one has $\alpha _{\rho }\in (0,1]$, and %
\eqref{eq:Heston_affine_link} forces $\eta =1-\alpha _{\rho }(\gamma -1)\in
(0,1)$, which requires $\alpha _{\rho }(\gamma -1)<1$. Moreover, $\widehat{c%
}<0$, since $\gamma >1$ and $q\neq 0$.

\medskip \noindent This affine restriction also has a transparent preference
interpretation. Under Assumption~\ref{asm:4-17},
\begin{equation*}
\eta =1-\alpha _{\rho }(\gamma -1)
=2-\gamma +\frac{(\gamma -1)^{2}}{\gamma }|\rho |^{2},\qquad
\eta -\frac{1}{\gamma }
=-\frac{(\gamma -1)^{2}}{\gamma }(1-|\rho |^{2}).
\end{equation*}%
Consequently, $0<\eta \leq 1/\gamma <1$ and $0<\lambda \leq 1$. Equality
$\eta =1/\gamma $ occurs when $|\rho |=1$, in which case the variance shock,
that is, the Brownian noise driving the variance process $\Gamma $ in
\eqref{eq:Heston_CIR}, is fully spanned by the traded assets, and the
coefficient multiplying the unspanned Brownian motion $W^{\perp }$ in the
state price density \eqref{eq:Heston_SPD} vanishes. If $|\rho |<1$, the model contains unspanned volatility risk and
the affine restriction requires $\eta <1/\gamma $, or, equivalently, $\lambda
<1$. This connection between spanning and preferences is a property of the
explicit affine solvability regime, and not a general restriction of the
forward Epstein--Zin recursive preferences.

\medskip \noindent In particular, when $|\rho|<1$ then
$\gamma\eta<1$, so the forward Epstein--Zin aggregator is concave, rather than
convex, in its continuation value; see Appendix~\ref{app:C}. Therefore, the convex
variational results of Section~\ref{sec:3-4} do not apply in this incomplete affine
setting. The state price density in Theorem~\ref{thm:4-19} is obtained directly from the discounted marginal utility identity of
Theorem~\ref{thm:3-12}, which, we recall, does not require either convexity or concavity
in the continuation value. At the limiting case $|\rho|=1$, the variance shock is fully spanned, and one has
$\gamma\eta=1$. The aggregator becomes affine in its continuation value.

\begin{proposition}[Stationary Markovian Heston solution]\label{prop:4-18}
Under Assumption~\ref{asm:4-17}, the stationary ODE \eqref{eq:Heston_ODE_phi} admits the following explicit
solution.

\smallskip

\noindent \textit{(i) Linear equation.} The transformed equation for
$\bar h$ becomes
\begin{equation}
\begin{aligned}
&\frac{a^{2}x}{2}\, \bar{h}^{\prime \prime }(x)+\Bigl(b(\bar{\kappa}-x)+\tfrac{%
(1-\gamma )\,a\,q^{\top }\rho }{\gamma }\,x\Bigr)\bar{h}^{\prime }(x)+\alpha
_{\rho }\Bigl(\tfrac{(1-\gamma )|q|^{2}}{2\gamma }\,x-\lambda \bar{\Psi}%
\Bigr)\bar{h}(x) =\frac{\alpha _{\rho }(1-\gamma )}{1-\eta }\, \bar{\Psi}^{\eta }.
\end{aligned}  \label{eq:Heston_ODE_h_linear}
\end{equation}%
\smallskip

\noindent \textit{(ii) Resolvent.} One solution is given by
\begin{equation}
\bar{h}(x)=k\int_{0}^{\infty }e^{A(s)-\mathcal R(s)x}\,ds,\qquad k=\frac{\alpha
_{\rho }\,(\gamma -1)}{1-\eta }\, \bar{\Psi}^{\eta }>0,
\label{eq:Heston_resolvent}
\end{equation}%
where
\begin{equation}
\dot{\mathcal R}(s)=-\widehat{\kappa }\,\mathcal R(s)-\tfrac{1}{2}a^{2}\mathcal R(s)^{2}-\widehat{c}%
,\qquad \mathcal R(0)=0,  \label{eq:Heston_Riccati}
\end{equation}%
and
\begin{equation*}
\dot{A}(s)=-b\bar{\kappa}\,\mathcal R(s)-\eta \bar{\Psi},\qquad A(0)=0.
\end{equation*}%
\smallskip

\noindent \textit{(iii) Properties.} The function $\bar h$ is positive and
twice continuously differentiable on $(0,\infty)$, with
\begin{equation}
-\mathcal R_{\infty }\leq \frac{\bar{h}^{\prime }(x)}{\bar{h}(x)}\leq 0\qquad \text{and}\qquad
\mathcal R_{\infty }=\frac{-\widehat{\kappa }+\sqrt{\widehat{\kappa }^{2}-2a^{2}%
\widehat{c}}}{a^{2}}>0.  \label{eq:Heston_logderiv_bound}
\end{equation}%
Furthermore, the function $\bar\varphi=\bar h^{1/\alpha_\rho}$ is a positive and
twice continuously differentiable solution of \eqref{eq:Heston_ODE_phi}, and
$\bar\varphi^{-\eta/\lambda}=1/\bar h$ is locally bounded.
\end{proposition}

\begin{proof}
We first study the two deterministic functions
$\mathcal R$ and $A$ that enter the resolvent representation and, in turn, verify directly
that the resolvent solves the required differential equation.

\smallskip

\noindent \textit{Step 1. Properties of $\mathcal R$ and $A$.}
By \eqref{eq:Heston_affine_link}, we have $
\frac{\eta}{\alpha _{\rho }\lambda}=1.
$
Consequently, the right-hand side of \eqref{eq:Heston_ODE_h} is the constant
$
\frac{\alpha _{\rho }(1-\gamma )}{1-\eta}\bar{\Psi}^{\eta}=-k,
$
which is negative since $\gamma>1$ and $\eta\in(0,1)$. Thus,
\eqref{eq:Heston_ODE_h} reduces to the linear equation
\eqref{eq:Heston_ODE_h_linear}.

Set
\begin{equation*}
Q(r)=\frac{a^{2}}{2}r^{2}+\widehat{\kappa}r+\widehat{c}.
\end{equation*}
The assumptions $\widehat{c}<0$ and $\widehat{\kappa}>0$ imply that $Q$ has
one negative and one positive root. Its positive root is given by
\begin{equation*}
\mathcal R_{\infty}
=\frac{-\widehat{\kappa}
+\sqrt{\widehat{\kappa}^{2}-2a^{2}\widehat{c}}}{a^{2}}.
\end{equation*}
Equation \eqref{eq:Heston_Riccati} can be written as follows; it holds that
\begin{equation*}
\dot{\mathcal R}(s)=-Q(\mathcal R(s)),\qquad \mathcal R(0)=0,\qquad s\geq0.
\end{equation*}
For every $r\in[0,\mathcal R_{\infty})$, $Q(r)<0$. It follows that
$\mathcal R$ is increasing as long as
$\mathcal R(s)<\mathcal R_{\infty}$. Uniqueness for the Riccati equation
prevents $\mathcal R$ from crossing the stationary solution
$s\mapsto \mathcal R_{\infty}$.
Therefore,
\begin{equation}
0\leq \mathcal R(s)<\mathcal R_{\infty}\quad\text{for each }s\geq0,
\qquad\text{and}\qquad
\lim_{s\to\infty}\mathcal R(s)=\mathcal R_{\infty}.
\label{eq:proof_Riccati_bounds}
\end{equation}

The definition of $A$ gives
\begin{equation*}
A(s)=-\int_{0}^{s}
\bigl(b\bar{\kappa}\mathcal R(r)+\eta\bar{\Psi}\bigr)\,dr,\qquad s\geq0.
\end{equation*}
Since $\mathcal R(r)\geq0$, we have
\begin{equation*}
A(s)\leq-\eta\bar{\Psi}s,\qquad s\geq0.
\end{equation*}
Hence, for every $x>0$ and $s\geq0$,
\begin{equation*}
0<e^{A(s)-\mathcal R(s)x}\leq e^{-\eta\bar{\Psi}s},\qquad s\geq0.
\end{equation*}
The function on the right-hand side above is integrable on $(0,\infty)$. The
resolvent in \eqref{eq:Heston_resolvent} is therefore finite and strictly
positive.

\smallskip

\noindent \textit{Step 2. Regularity and logarithmic derivative.}
The bound in \eqref{eq:proof_Riccati_bounds} and the preceding exponential estimate
allow us to differentiate twice under the integral sign, obtaining, for each $x>0$,
\begin{align*}
\bar h'(x)
&=-k\int_{0}^{\infty}\mathcal R(s)e^{A(s)-\mathcal R(s)x}\,ds
\qquad \text{and}\qquad\\
\bar h''(x)
&=k\int_{0}^{\infty}\mathcal R(s)^{2}e^{A(s)-\mathcal R(s)x}\,ds.
\end{align*}
Thus, $\bar h$ is twice continuously differentiable. Moreover,
\begin{equation*}
\frac{\bar h'(x)}{\bar h(x)}
=-\frac{\int_{0}^{\infty}\mathcal R(s)e^{A(s)-\mathcal R(s)x}\,ds}
{\int_{0}^{\infty}e^{A(s)-\mathcal R(s)x}\,ds},\qquad x>0.
\end{equation*}
The quotient on the right-hand side is the negative of a weighted average
of the values $\mathcal R(s)$. Since
$0\leq \mathcal R(s)\leq \mathcal R_{\infty}$, $s\geq0$, we obtain
\begin{equation*}
-\mathcal R_{\infty}\leq\frac{\bar h'(x)}{\bar h(x)}\leq0, \qquad x>0,
\end{equation*}
and \eqref{eq:Heston_logderiv_bound} follows.

\smallskip

\noindent \textit{Step 3. Derivation of the differential equation.}
Let
\begin{equation*}
K_s(x)=e^{A(s)-\mathcal R(s)x},\qquad s\geq0,\ x>0.
\end{equation*}
Then,
\begin{equation*}
\partial_xK_s(x)=-\mathcal R(s)K_s(x),
\qquad
\partial_{xx}K_s(x)=\mathcal R(s)^{2}K_s(x),\qquad s \geq 0, ~ x>0.
\end{equation*}
Substitution into the differential operator on the left-hand side of
\eqref{eq:Heston_ODE_h_linear} yields
\begin{align*}
&\frac{a^{2}x}{2}\partial_{xx}K_s(x)
+\bigl(b\bar{\kappa}-\widehat{\kappa}x\bigr)\partial_xK_s(x)
+\bigl(\widehat{c}x-\eta\bar{\Psi}\bigr)K_s(x)\\
&\quad=\Bigl[-b\bar{\kappa}\mathcal R(s)-\eta\bar{\Psi}
+x\bigl(\tfrac{a^{2}}{2}\mathcal R(s)^{2}
+\widehat{\kappa}\mathcal R(s)+\widehat{c}\bigr)\Bigr]K_s(x)\\
&\quad=\bigl(\dot A(s)-\dot{\mathcal R}(s)x\bigr)K_s(x)\\
&\quad=\partial_sK_s(x), \qquad s \geq 0, ~ x >0,
\end{align*}
where, in the second equality, we used the defining equations for $A$ and
$\mathcal R$.
Since $\bar h(x)=k\int_{0}^{\infty}K_s(x)\,ds$, the same differential
operator applied to $\bar h$ is equal to
\begin{align*}
k\int_{0}^{\infty}\partial_sK_s(x)\,ds
&=k\left(\lim_{s\to\infty}K_s(x)-K_0(x)\right)\\
&=-k\\
&=\frac{\alpha _{\rho }(1-\gamma )}{1-\eta}\bar{\Psi}^{\eta}.
\end{align*}
Here, $K_0(x)=1$, while the exponential estimate in Step~1 gives
$\lim_{s\to\infty}K_s(x)=0$. This proves that $\bar h$ solves
\eqref{eq:Heston_ODE_h_linear}.

Finally, if we define $\bar\varphi=\bar h^{1/\alpha _{\rho}}$, which inverts the
transformation $\bar h=\bar\varphi^{\alpha _{\rho}}$ which was used to
derive \eqref{eq:Heston_ODE_h}, then $\bar\varphi$ is a positive,
twice continuously differentiable solution of
\eqref{eq:Heston_ODE_phi}. Moreover, \eqref{eq:Heston_affine_link} gives
\begin{equation*}
\bar\varphi^{-\eta/\lambda}=\bar h^{-1},\qquad x>0.
\end{equation*}
The right-hand side is locally bounded because $\bar h$ is continuous and
strictly positive on $(0,\infty)$.
\end{proof}

\begin{theorem}[Forward recursive utility under Heston volatility]\label{thm:4-19}
Assume
the Heston market specification \eqref{eq:Heston_CIR}--%
\eqref{eq:Heston_theta} and recall Assumption~\ref{asm:4-17}. Let $\bar{\varphi}$ be the
positive stationary Markovian solution in Proposition~\ref{prop:4-18}. Then, the pair $(U,F)$ in %
\eqref{eq:Heston_separable_stationary} is a consistent forward aggregator
system on $\mathcal{U}_{\gamma }$. The optimal control policies are
\begin{equation}
\pi _{t}^{\star }=\frac{1}{\gamma }(\sigma _{0}\sigma _{0}^{\top
})^{-1}\sigma _{0}\Bigl(q+a\rho \, \frac{\bar{\varphi}^{\prime }(\Gamma _{t})%
}{\bar{\varphi}(\Gamma _{t})}\Bigr),\qquad \check{c}_{t}^{\star }=\frac{%
\bar{\Psi}^{\eta }}{\bar{\varphi}(\Gamma _{t})^{\eta /\lambda }},
\qquad t\geq0.
\label{eq:Heston_feedback}
\end{equation}%
The state price density $Y^{\star }$ satisfies 
\begin{equation}
\frac{dY_{t}^{\star }}{Y_{t}^{\star }}=-\sqrt{\Gamma _{t}}\,q^{\top
}dW_{t}^{S}+\sqrt{1-|\rho |^{2}}\,a\sqrt{\Gamma _{t}}\, \frac{\bar{\varphi}%
^{\prime }(\Gamma _{t})}{\bar{\varphi}(\Gamma _{t})}\,dW_{t}^{\perp },\qquad
Y_{0}^{\star }=U_{x}(0,x),\qquad t\geq0.  \label{eq:Heston_SPD}
\end{equation}%
Equivalently,
\begin{equation*}
Y_t^\star=\bar\varphi(\Gamma_t)(X_t^\star)^{-\gamma}\kappa_t^\star
\qquad \text{and}\qquad
\kappa_t^\star
=\exp\left(-\int_0^t
\bigl((1-\lambda)\check c_s^\star+\lambda\bar\Psi\bigr)\,ds\right),\qquad t\geq0.
\end{equation*}
\end{theorem}

\begin{proof}
We verify the drift equation, the viability
conditions, the feedback controls, and the state price density separately.

\smallskip

\noindent \textit{Step 1. Local characteristics of $\Phi$.}
Set
\begin{equation*}
\Phi_t=\bar\varphi(\Gamma_t),
\qquad \varphi=\Phi_0=\bar\varphi(\Gamma_0),\qquad t\geq0.
\end{equation*}
Applying It\^{o}'s formula and using
\eqref{eq:Heston_CIR} gives
\begin{align*}
d\Phi_t
&=\left(
b(\bar\kappa-\Gamma_t)\bar\varphi'(\Gamma_t)
+\frac{a^{2}\Gamma_t}{2}\bar\varphi''(\Gamma_t)
\right)dt\\
&\quad
+a\sqrt{\Gamma_t}\bar\varphi'(\Gamma_t)
\left(\rho^{\top}dW_t^S
+\sqrt{1-|\rho|^{2}}\,dW_t^{\perp}\right),\qquad t\geq0.
\end{align*}
Dividing the drift and the Brownian coefficient by
$\Phi_t=\bar\varphi(\Gamma_t)$ yields
\begin{align*}
b_t^{\Phi}
&=\frac{
b(\bar\kappa-\Gamma_t)\bar\varphi'(\Gamma_t)
+\tfrac{a^{2}\Gamma_t}{2}\bar\varphi''(\Gamma_t)}
{\bar\varphi(\Gamma_t)},\qquad t\geq0,\qquad \text{and}\\
\delta_t^{\Phi}
&=\frac{a\sqrt{\Gamma_t}\bar\varphi'(\Gamma_t)}
{\bar\varphi(\Gamma_t)}
\left(\rho,\sqrt{1-|\rho|^{2}}\right)^{\top},\qquad t\geq0,
\end{align*}
and
\eqref{eq:Heston_bphi_deltaphi} follows.

\smallskip

\noindent \textit{Step 2. Verification of the scalar drift equation.}
The projection onto the traded Brownian directions removes the last
component of $\delta_t^{\Phi}$. Hence
\begin{equation*}
\sigma_t^{+}\sigma_t\delta_t^{\Phi}
=\frac{a\sqrt{\Gamma_t}\bar\varphi'(\Gamma_t)}
{\bar\varphi(\Gamma_t)}(\rho,0)^{\top},\qquad t\geq0.
\end{equation*}
Since $\theta_t=\sqrt{\Gamma_t}(q,0)^{\top}$, we obtain
\begin{equation*}
\theta_t+\sigma_t^{+}\sigma_t\delta_t^{\Phi}
=\sqrt{\Gamma_t}\left(
q+a\rho\frac{\bar\varphi'(\Gamma_t)}
{\bar\varphi(\Gamma_t)},0\right)^{\top},\qquad t\geq0.
\end{equation*}
Therefore,
\begin{align*}
\left|\theta_t+\sigma_t^{+}\sigma_t\delta_t^{\Phi}\right|^{2}
&=\Gamma_t\left|
q+a\rho\frac{\bar\varphi'(\Gamma_t)}
{\bar\varphi(\Gamma_t)}\right|^{2}\\
&=\Gamma_t|q|^{2}
+2a\Gamma_t q^{\top}\rho
\frac{\bar\varphi'(\Gamma_t)}{\bar\varphi(\Gamma_t)}
+a^{2}\Gamma_t|\rho|^{2}
\frac{\bar\varphi'(\Gamma_t)^{2}}
{\bar\varphi(\Gamma_t)^{2}}.
\end{align*}

We now substitute the above expression and the drift from Step~1 into the scalar
drift equation \eqref{eq:EZ_bphi}. After multiplying by
$\bar\varphi(\Gamma_t)$, the drift condition is equivalent to
\begin{align*}
&b(\bar\kappa-\Gamma_t)\bar\varphi'(\Gamma_t)
+\frac{a^{2}\Gamma_t}{2}\bar\varphi''(\Gamma_t)
=\frac{1-\gamma}{1-\eta}\bar\Psi^{\eta}
\bar\varphi(\Gamma_t)^{1-\eta/\lambda} -\frac{1-\gamma}{2\gamma}\Gamma_t
\bar\varphi(\Gamma_t)
\left|q+a\rho\frac{\bar\varphi'(\Gamma_t)}
{\bar\varphi(\Gamma_t)}\right|^{2}
+\lambda\bar\Psi\bar\varphi(\Gamma_t),\qquad t\geq0.
\end{align*}
Rearranging terms and expanding the squared norm show that this identity is
precisely \eqref{eq:Heston_ODE_phi} evaluated at $x=\Gamma_t$, where
\begin{align*}
&\frac{a^{2}x}{2}\bar\varphi''(x)
+\left(b(\bar\kappa-x)
+\frac{(1-\gamma)a q^{\top}\rho}{\gamma}x\right)
\bar\varphi'(x)
+\frac{(1-\gamma)a^{2}|\rho|^{2}x}{2\gamma}
\frac{\bar\varphi'(x)^{2}}{\bar\varphi(x)} \\
&\quad+\left(\frac{(1-\gamma)|q|^{2}}{2\gamma}x
-\lambda\bar\Psi\right)\bar\varphi(x)
=\frac{1-\gamma}{1-\eta}\bar\Psi^{\eta}
\bar\varphi(x)^{1-\eta/\lambda},
\qquad x>0.
\end{align*}
By Proposition~\ref{prop:4-18}, $\bar\varphi$ solves this equation on $(0,\infty)$.
Since $\Gamma_t>0$, $t\geq0$, the scalar drift condition \eqref{eq:EZ_bphi} follows.
Therefore, Theorem~\ref{thm:4-4}~(i) implies that
$U(t,x)=\bar\varphi(\Gamma_t)u(x)$ solves the forward Epstein--Zin recursive
SPDE \eqref{eq:SPDE-EZ}.
\smallskip

\noindent \textit{Step 3. Viability and admissibility.}
The process $\Psi_t=\bar\Psi$ is strictly positive and constant, and, hence, it
satisfies (V1) and the part of (V2) that concerns $\Psi$. Proposition~\ref{prop:4-18}
gives
\begin{equation*}
\frac{\bar\varphi'(x)}{\bar\varphi(x)}
=\frac{1}{\alpha_\rho}\frac{\bar h'(x)}{\bar h(x)},
\qquad x>0.
\end{equation*}
The logarithmic derivative on the right-hand side is bounded by
\eqref{eq:Heston_logderiv_bound}. Under the Feller condition, $\Gamma_t$ is
strictly positive and has continuous paths. Thus, on every finite time
interval, $\Gamma$ has a finite pathwise maximum and a strictly positive
pathwise minimum. Then, the formula for $\delta_t^{\Phi}$ in Step~1 shows
that $\delta^{\Phi}$ is pathwise bounded on every finite time interval, and
the remaining part of (V2) follows.

Using that $\Phi_t=\bar\varphi(\Gamma_t)$ solves
\eqref{eq:EZ_SDE_Phi}, Proposition~\ref{prop:4-2} gives, for $0\leq t<\tau$,
\begin{equation*}
\bar\varphi(\Gamma_0)^{\eta/\lambda}
-\int_0^t\bar\Psi^{\eta}\mathcal E_u^{-1}\,du
=\Phi_t^{\eta/\lambda}\mathcal E_t^{-1},\qquad 0\leq t<\tau,
\end{equation*}
where $\tau $ is the positivity lifetime of $\Phi $ introduced in
\eqref{eq:EZ_positive_lifetime}. Suppose that $\tau<\infty$. By continuity and the local finiteness and strict
positivity of $\mathcal E$, letting $t\uparrow\tau$ yields
\begin{equation*}
0=\Phi_\tau^{\eta/\lambda}\mathcal E_\tau^{-1}
=\bar\varphi(\Gamma_\tau)^{\eta/\lambda}\mathcal E_\tau^{-1},
\end{equation*}
which contradicts the strict positivity of $\bar\varphi(\Gamma_\tau)$ and
$\mathcal E_\tau$. Thus, $\tau=\infty$, and (V3) follows.

Proposition~\ref{prop:4-18} also gives
\begin{equation*}
\bar\varphi(x)^{-\eta/\lambda}=\bar h(x)^{-1},\qquad x>0.
\end{equation*}
This function is locally bounded on $(0,\infty)$. Since each path of
$\Gamma$ remains in a compact subset of $(0,\infty)$ on a finite time
interval, the optimal relative consumption
\begin{equation*}
\check c_t^{\star}
=\bar\Psi^{\eta}\bar\varphi(\Gamma_t)^{-\eta/\lambda},\qquad t\geq0,
\end{equation*}
is pathwise bounded on any finite time interval. The same is true for the
discount rate
\begin{equation*}
\alpha_t^{\star}
=(1-\lambda)\check c_t^{\star}+\lambda\bar\Psi,\qquad t\geq0.
\end{equation*}
Together with the pathwise bounds for $\Gamma$, $\delta^{\Phi}$, and
$\theta$, these estimates verify the input bounds of Assumption~\ref{asm:3-13} with
the constants specified in Corollary~\ref{cor:4-8}. In turn, Proposition~\ref{prop:4-3} and the homothetic
relations in the proof of Theorem~\ref{thm:4-4} verify Assumptions~\ref{asm:2-9} and~\ref{asm:3-2}, while
Step~2 verifies the HJB SPDE. Theorem~\ref{thm:3-14} implies that the
feedback control pair is admissible and optimal.

\smallskip

\noindent \textit{Step 4. Feedback controls and state price density.} Substituting the Heston model parameters in \eqref{eq:EZ_optimal_controls} gives
\begin{align*}
\pi_t^{\star}
&=\frac{1}{\gamma}
\bigl(\Gamma_t\sigma_0\sigma_0^{\top}\bigr)^{-1}
\left(\Gamma_t\sigma_0q
+a\Gamma_t\sigma_0\rho
\frac{\bar\varphi'(\Gamma_t)}{\bar\varphi(\Gamma_t)}\right)\\
&=\frac{1}{\gamma}
(\sigma_0\sigma_0^{\top})^{-1}\sigma_0
\left(q+a\rho
\frac{\bar\varphi'(\Gamma_t)}{\bar\varphi(\Gamma_t)}\right),\qquad t\geq0.
\end{align*}
Similarly, we obtain
\begin{equation*}
\check c_t^{\star}
=\frac{\bar\Psi^{\eta}}
{\bar\varphi(\Gamma_t)^{\eta/\lambda}},\qquad t\geq0,
\end{equation*}
and \eqref{eq:Heston_feedback} follows.

For the dual control process, \eqref{eq:nu_star_marginal} and the homothetic identity
$\delta_x(t,x)/U_x(t,x)=\delta_t^{\Phi}$ imply
\begin{align*}
\nu_t^{\star}
&=(I_{n+1}-\sigma_t^{+}\sigma_t)\delta_t^{\Phi}\\
&=\left(0,
\sqrt{1-|\rho|^{2}}\,a\sqrt{\Gamma_t}
\frac{\bar\varphi'(\Gamma_t)}{\bar\varphi(\Gamma_t)}\right)^{\top},\qquad t\geq0.
\end{align*}
Therefore, \eqref{eq:optimal_SPD} gives
\begin{align*}
\frac{dY_t^{\star}}{Y_t^{\star}}
&=(-\theta_t+\nu_t^{\star})^{\top}dW_t\\
&=-\sqrt{\Gamma_t}\,q^{\top}dW_t^{S}
+\sqrt{1-|\rho|^{2}}\,a\sqrt{\Gamma_t}
\frac{\bar\varphi'(\Gamma_t)}{\bar\varphi(\Gamma_t)}\,dW_t^{\perp},\qquad t\geq0,
\end{align*}
and \eqref{eq:Heston_SPD} follows. Finally,
\begin{equation*}
U_x(t,X_t^{\star})
=\bar\varphi(\Gamma_t)(X_t^{\star})^{-\gamma},\qquad t\geq0.
\end{equation*}
Combining this identity with Theorem~\ref{thm:3-12} and Corollary~\ref{cor:4-10} gives
\begin{equation*}
Y_t^{\star}
=\bar\varphi(\Gamma_t)(X_t^{\star})^{-\gamma}\kappa_t^{\star}
\qquad \text{and}\qquad
\kappa_t^{\star}
=\exp\left(-\int_0^t
\bigl((1-\lambda)\check c_s^{\star}+\lambda\bar\Psi\bigr)\,ds\right),\qquad t\geq0,
\end{equation*}
and we easily conclude.
\end{proof}

\begin{remark}\label{rem:4-20}
The optimal portfolio process in %
\eqref{eq:Heston_feedback} decomposes into the Merton component $\tfrac{1}{%
\gamma}(\sigma_0\sigma_0^\top)^{-1}\sigma_0 q$ and an intertemporal hedging
component proportional to $\rho \, \bar \varphi^{\prime }(\Gamma_t)/\bar
\varphi(\Gamma_t)$. The consumption-to-wealth ratio depends on $\Gamma_t$
through $\bar \varphi(\Gamma_t)^{-\eta/\lambda}$; and the orthogonal part 
$\sqrt{1-|\rho|^2}\,a\sqrt{\Gamma_t}\, \bar \varphi^{\prime }/\bar \varphi$
of $Y^\star$ in \eqref{eq:Heston_SPD} is the unspanned Heston component
generated by discounted marginal utility (Theorem~\ref{thm:3-12}). The
myopic demand is scaled directly by $1/\gamma $ and does not depend on $\eta $.
The EIS parameter $\eta$ nevertheless affects the stationary solution $\bar{\varphi}$ through
\eqref{eq:Heston_ODE_phi}. It therefore enters the portfolio indirectly
through the intertemporal hedging term, and affects the optimal consumption directly
through $\bar{\Psi}^{\eta }/\bar{\varphi}^{\eta /\lambda }$. Thus the risk
aversion parameter $\gamma$ governs the immediate market price of risk exposure, whereas the EIS parameter $\eta$
operates through consumption substitution and the hedge against changes in
the forward preferences.
\end{remark}

\begin{remark}\label{rem:4-21}
The nonlinear term in %
\eqref{eq:Heston_ODE_phi} contains $\bar{\varphi}^{-\eta /\lambda }$, and the
logarithmic transform $g=\log \bar{\varphi}$ is not invariant under additive
shifts $g\mapsto g+C$. Therefore, the usual ergodic BSDE reduction is not
natural here. An infinite-horizon BSDE with a level-dependent monotone driver is a
more plausible representation away from the affine locus. Indeed, let
\begin{equation*}
\mathcal L^{\widehat\kappa}v(x)
=\frac{a^{2}x}{2}v''(x)
+\bigl(b\bar\kappa-\widehat\kappa x\bigr)v'(x),\qquad x>0,
\end{equation*}
and set $z=a\sqrt{x}\,g'(x)$. Dividing
\eqref{eq:Heston_ODE_phi} by $\bar\varphi$, we formally deduce that $g$ solves
\begin{equation*}
\mathcal L^{\widehat\kappa}g(x)
+\mathfrak f\bigl(x,g(x),a\sqrt{x}\,g'(x)\bigr)=0,\qquad x>0,
\end{equation*}
where
\begin{equation}
\mathfrak f(x,y,z)
=\frac{\alpha_\rho}{2}|z|^{2}
+\frac{(1-\gamma)|q|^{2}}{2\gamma}x
-\lambda\bar\Psi
-\frac{1-\gamma}{1-\eta}\bar\Psi^{\eta}
e^{-(\eta/\lambda)y}.
\label{eq:Heston_log_BSDE_driver}
\end{equation}
The nonlinear term has the exact monotonicity property, in that
\begin{equation}
\partial_y\mathfrak f(x,y,z)
=-\bar\Psi^{\eta}e^{-(\eta/\lambda)y}<0,
\label{eq:Heston_log_BSDE_monotonicity}
\end{equation}
since $\eta/\lambda=(\eta-1)/(1-\gamma)$. Thus, the recursive term entirely
determines the level that the standard ergodic formulation would not determine. However, the
monotonicity is not uniform since its modulus degenerates in one
tail of $y$. Moreover, the driver has quadratic growth in $z$ and an
unbounded linear dependence on the CIR state. Therefore, establishing an
infinite horizon representation requires a suitable Lyapunov or
properness class. The existing infinite horizon BSDE representation of
homothetic forward criteria in \citet{Liang-Volume1} does not directly cover
this combination, and we leave this problem for future research.

The constant coefficient case makes the associated selection mechanism
transparent. When
\begin{equation*}
\dot\Xi_t=-\bar\Psi^{\eta}+\bar\chi\Xi_t,
\qquad \bar\chi>0,\qquad t\geq0,
\end{equation*}
the critical point $\Xi^{\mathrm{crit}}=\bar\Psi^{\eta}/\bar\chi$ is an
unstable equilibrium of this forward equation, in the sense that any
trajectory started away from it diverges as $t\uparrow \infty $. If, however,
the same deterministic equation is considered
backward from infinity, then boundedness selects this critical solution;
indeed, it is the only solution that remains bounded on $[0,\infty)$. The other
viable forward trajectories
have the form
\begin{equation*}
\Xi_t=\Xi^{\mathrm{crit}}+C e^{\bar\chi t},
\qquad C>0,\qquad t\geq0,\ \bar\chi>0,
\end{equation*}
and carry a residual preference component. Hence the bounded backward
construction selects the critical solution $B_\infty=0$, whereas the forward
formulation also retains the residual solutions $B_\infty>0$. In the
nonlinear Heston case this ``stable--unstable'' interpretation may serve as a prospective
selection principle rather than a direct sum decomposition, since both $Z$ and
$\delta^\Phi$ depend endogenously on the solution.
\end{remark}

The dual stationary equation for $\widetilde{\Phi }_{t}=\bar{%
\varphi}(\Gamma _{t})^{1/\gamma }$ is obtained from %
\eqref{eq:dual_factor_SDE}--\eqref{eq:dual_factor_drift} by setting $%
\widetilde{\Phi }_{t}=\bar{\varphi}(\Gamma _{t})^{1/\gamma }$.

\bigskip

\section*{Conflicts of Interest}

The authors declare no conflicts of interest.

\section*{Data Availability Statement}

Data sharing is not applicable to this article, as no new data were created or
analyzed in this study.

\bibliographystyle{plainnat}
\bibliography{refs}

\appendix

\section{Regular random fields and detailed proofs for the forward recursive utilities}\label{app:A}

We first recall the regularity classes used in Section~\ref{sec:2} and then summarize
the stochastic calculus details underlying the verification and
normalization results. We keep the time and state arguments visible
throughout. For example, when a random field is evaluated along a wealth process, we
write, $U_x(t,X_t)$.

\subsection*{Regular random fields}

We follow \citet{Kunita-1990} and the formulation in
\citet[Appendix~A]{LSZ25}. Let $G(t,x)$, $t\geq0$, $x>0$, be a progressive
random field, $m\in\mathbb N_{0}$ and $\varepsilon\in(0,1]$. We use
the convention $\partial_x^0G=G$.

The field $G$ is called $\mathcal C^{m}$-regular if, $\mathbb P$-almost
surely, the map $x\mapsto G(t,x)$ is $m$-times continuously
differentiable for each $t\geq0$. It is called
$\mathcal C^{m,\varepsilon}$-regular if, on the same almost-sure event,
$\partial_x^mG(t,\cdot)$ is locally $\varepsilon$-H\"{o}lder continuous for
each $t\geq0$. For every compact set $K\subset(0,\infty)$ and $t\geq0$,
define
\begin{equation*}
\begin{aligned}
\|G\|_{m,\varepsilon:K}(t)
={}&\sup_{x\in K}\frac{|G(t,x)|}{1+|x|}
+\sum_{j=1}^{m}\sup_{x\in K}|\partial_x^jG(t,x)| +\sup_{\substack{x,y\in K\\x\neq y}}
\frac{|\partial_x^mG(t,x)-\partial_x^mG(t,y)|}{|x-y|^{\varepsilon}}.
\end{aligned}
\end{equation*}
A $\mathcal C^{m,\varepsilon}$-regular field $G$ belongs to
$\mathcal K_{\mathrm{loc}}^{m,\varepsilon}$ if, for every compact
$K\subset(0,\infty)$ and each $T>0$,
\begin{equation*}
\int_0^T\|G\|_{m,\varepsilon:K}(t)\,dt<\infty
\qquad\text{$\mathbb P$-a.s.},
\end{equation*}
and it belongs to $\overline{\mathcal K}_{\mathrm{loc}}^{m,\varepsilon}$
if
\begin{equation*}
\int_0^T\|G\|_{m,\varepsilon:K}(t)^2\,dt<\infty .
\end{equation*}
For vector-valued random fields, these requirements are understood
componentwise.

Let now $U(t,x)$, $t\geq0$, $x>0$, be a progressively measurable It\^{o}
semimartingale random field with decomposition
\begin{equation*}
dU(t,x)=b(t,x)\,dt+\delta(t,x)^{\top}dW_t,\qquad t\geq0, ~x>0.
\end{equation*}
The field $U$ is called a
$\mathcal K_{\mathrm{loc}}^{m,\varepsilon}$-semimartingale random field if
$U(0,\cdot)$ is $\mathcal C^{m,\varepsilon}$-regular and
\begin{equation*}
\int_0^{\cdot}b(s,\cdot)\,ds
\in\mathcal K_{\mathrm{loc}}^{m,\varepsilon}
\qquad \text{and}\qquad
\int_0^{\cdot}\delta(s,\cdot)^{\top}dW_s
\in\overline{\mathcal K}_{\mathrm{loc}}^{m,\varepsilon}.
\end{equation*}
If $U$ is a
$\mathcal K_{\mathrm{loc}}^{m,\varepsilon}$-semimartingale random field,
then, for every $\varepsilon'<\varepsilon$,
\begin{equation*}
(b,\delta)\in
\mathcal K_{\mathrm{loc}}^{m,\varepsilon'}
\times\overline{\mathcal K}_{\mathrm{loc}}^{m,\varepsilon'}.
\end{equation*}
Conversely, if $(b,\delta)$ is in
$\mathcal K_{\mathrm{loc}}^{m,\varepsilon}
\times\overline{\mathcal K}_{\mathrm{loc}}^{m,\varepsilon}$ then
$U$ is a
$\mathcal K_{\mathrm{loc}}^{m,\varepsilon'}$-semimartingale random field
for every $\varepsilon'<\varepsilon$.

In particular, if $U$ is $\mathcal C^2$-regular and a
$\mathcal K_{\mathrm{loc}}^{1,\varepsilon}$-semimartingale random field, then,
the It\^{o}--Ventzel formula applies along every continuous It\^{o}
process $X$ with values in $(0,\infty)$ and gives, for $t\geq0$,
\begin{equation*}
dU(t,X_t)=b(t,X_t)\,dt+\delta(t,X_t)^{\top}dW_t+U_x(t,X_t)\,dX_t
+\frac12U_{xx}(t,X_t)\,d\langle X\rangle_t
+\delta_x(t,X_t)^{\top}d\langle W,X\rangle_t.
\end{equation*}

\begin{proof}[Proof of Theorem~\ref{thm:2-6}]
Fix an admissible control pair
$(\pi,c)\in\mathcal A$ and write $X_t=X_t^{\pi,c}$. We first compute the
dynamics of the forward recursive criterion process. We, then, consider a
specific family of stopping times to verify the local consistency.

\smallskip

\noindent \textit{Step 1. Dynamics along an admissible wealth process.}
The It\^{o}--Ventzel formula applied to the random field $U$ along $X$
gives
\begin{align*}
dU(t,X_t)
&=b(t,X_t)\,dt+\delta(t,X_t)^{\top}dW_t\\
&\quad+U_x(t,X_t)\,dX_t
+\frac12U_{xx}(t,X_t)\,d\langle X\rangle_t
+\delta_x(t,X_t)^{\top}d\langle W,X\rangle_t,\qquad t\geq0.
\end{align*}
This yields
\begin{align*}
dU(t,X_t)
&=\Bigl(
b(t,X_t)
+U_x(t,X_t)\bigl(X_t\pi_t^{\top}\sigma_t\theta_t-c_t\bigr)\\
&\qquad
+\frac12U_{xx}(t,X_t)
\bigl|X_t\sigma_t^{\top}\pi_t\bigr|^2
+X_t\pi_t^{\top}\sigma_t\delta_x(t,X_t)
\Bigr)dt\\
&\quad+\Bigl(
\delta(t,X_t)
+U_x(t,X_t)X_t\sigma_t^{\top}\pi_t
\Bigr)^{\top}dW_t,\qquad t\geq0.
\end{align*}
After adding the aggregator term, we obtain
\begin{equation}
dV_t^{\pi,c}
=H(t,X_t,\pi_t,c_t)\,dt
+(Z_t^{\pi,c})^{\top}dW_t,
\qquad t\geq0,
\label{eq:proof_primal_criterion_dynamics}
\end{equation}
where
\begin{align*}
H(t,X_t,\pi_t,c_t)
&=b(t,X_t)
+U_x(t,X_t)\bigl(X_t\pi_t^{\top}\sigma_t\theta_t-c_t\bigr)\\
&\quad
+\frac12U_{xx}(t,X_t)
\bigl|X_t\sigma_t^{\top}\pi_t\bigr|^2
+X_t\pi_t^{\top}\sigma_t\delta_x(t,X_t)\\
&\quad
+F\bigl(t,c_t,U(t,X_t),Z(t,X_t,\pi_t)\bigr),\qquad t\geq0.
\end{align*}
Here $Z(t,x,\pi )$ is the pointwise map introduced in Section~\ref{sec:2-2}, so that its
value along the control pair is $Z(t,X_t,\pi _t)=Z_t^{\pi ,c}$, with
$Z^{\pi ,c}$ as in \eqref{eq:Z}.

\smallskip

\noindent \textit{Step 2. Local integrability.}
Fix $T>0$. Since $X$ is continuous and strictly positive, its range on
$[0,T]$ is a compact subset of $(0,\infty)$, $\mathbb P$-a.s. The local
regularity of $U$ and the bounds on the local characteristics established in
Appendix~\ref{app:A} therefore imply that
$\sup_{0\leq t\leq T}\bigl(|U_x(t,X_t)|+|U_{xx}(t,X_t)|\bigr)<\infty$ and
$\int_0^T\bigl(|b(t,X_t)|+|\delta(t,X_t)|^2+
|\delta_x(t,X_t)|^2\bigr)\,dt<\infty$, $\mathbb P$-a.s.

The admissibility of the control pair $(\pi,c)$ gives
\begin{equation*}
\int_0^T
\left(c_t+|\pi_t^{\top}\sigma_t|^2\right)\,dt<\infty.
\end{equation*}
In addition, the standing market assumption gives
\begin{equation*}
\int_0^T|\theta_t|^2\,dt<\infty.
\end{equation*}
Since
$|\pi_t^{\top}\sigma_t\theta_t|
\leq|\pi_t^{\top}\sigma_t|\,|\theta_t|$, $t\geq0$, the Cauchy--Schwarz inequality shows that the corresponding product is
integrable on $[0,T]$. Thus, all terms in the Hamiltonian other than the
aggregator are pathwise integrable on $[0,T]$.

\smallskip

\noindent \textit{Step 3. The local supermartingale property.}
We assume that $(\pi,c)$ also satisfies \eqref{eq:F-integrability}. For $n\geq1$, define
\begin{align*}
\tau_n=\inf\Biggl\{t\geq0\ \Bigg|\ {}
&X_t\notin\left(\frac1n,n\right)\ \text{or}\ {}
\int_0^t|Z_s^{\pi,c}|^2\,ds\geq n\, \ {}
\text{or}\ {}
\int_0^t|H(s,X_s,\pi_s,c_s)|\,ds\geq n
\Biggr\}.
\end{align*}
The local integrability established in Step~2 and the form of
$Z_t^{\pi,c}$ imply
\begin{equation*}
\int_0^T|Z_s^{\pi,c}|^2\,ds<\infty
\qquad \text{and}\qquad
\int_0^T|H(s,X_s,\pi_s,c_s)|\,ds<\infty,
\end{equation*}
$\mathbb P$-a.s.\ for every $T>0$. Since $X$ is continuous and strictly
positive, it follows that $\tau_n$ increases to infinity a.s.

On $[0,\tau_n]$, the stochastic integral in
\eqref{eq:proof_primal_criterion_dynamics} is a square integrable
martingale. The stopped finite variation term has total variation bounded
by $n$, and $V_{\cdot\wedge\tau_n}^{\pi,c}$ is integrable. Moreover,
\eqref{eq:HJB-ineq} holds and, hence,
\begin{equation*}
H(t,X_t,\pi_t,c_t)\leq0,\qquad t\geq0,
\end{equation*}
for $dt\otimes d\mathbb P$-a.e.\ $(t,\omega)$. Therefore,
$V_{\cdot\wedge\tau_n}^{\pi,c}$ is a supermartingale. Since
$\tau_n\uparrow\infty$, the process $V^{\pi,c}$ is a local
supermartingale. This proves condition i) of Definition~\ref{defn:2-3}.

\smallskip

\noindent \textit{Step 4. Equality along the feedback control pair.}
Let $X_t^{\star}$ be the wealth generated by the candidate optimal feedback control pair. Assumption
ii) of the theorem and \eqref{eq:HJB-eq} give
\begin{equation}
H\left(
t,X_t^{\star},
\pi^{\star}(t,X_t^{\star}),
c^{\star}(t,X_t^{\star})
\right)=0,\qquad t\geq0,
\label{eq:proof_H_zero_feedback}
\end{equation}
for $dt\otimes d\mathbb P$-a.e.\ $(t,\omega)$.

All other terms in $H$ are locally integrable by Step~2. Hence, the vanishing
of the Hamiltonian in \eqref{eq:proof_H_zero_feedback} also gives
\eqref{eq:F-integrability} along the feedback control pair.

We now see that equation \eqref{eq:proof_primal_criterion_dynamics} has zero
drift along the feedback control pair. The same stopping times make the stopped criterion process a
true martingale. Therefore,
$V^{\pi^{\star},c^{\star}}$ is a local martingale. This proves condition
ii) of Definition~\ref{defn:2-3}. Both consistency conditions hold, and the feedback
control pair is optimal.
\end{proof}

\begin{proof}[Proof of Proposition~\ref{prop:2-8}]
We first note that the control class $\mathcal{A}$ is unchanged, since Definition~\ref{defn:2-1} involves only the market coefficients and the
wealth equation, and makes no reference to the pair $(U,F)$. Next, we prove
that the transformation preserves consistency in the forward direction, which would then imply (i) and (ii). In turn, we apply the inverse transformation, and verify the spatial shape of the transformed utility.

\smallskip

\noindent \textit{Step 1. Transformation of the criterion dynamics.}
Fix $(\pi,c)\in\mathcal A$ satisfying \eqref{eq:F-integrability}, and set
\begin{equation*}
X_t=X_t^{\pi,c},
\qquad
u_t=U(t,X_t),
\qquad
z_t=Z_t^{\pi,c},\qquad t\geq0.
\end{equation*}
Since $V^{\pi,c}$ is a continuous local supermartingale, there exist a
continuous local martingale $M$ and a continuous non-increasing adapted
process $D$, with $D_0=0$, such that
\begin{equation*}
V_t^{\pi,c}=V_0^{\pi,c}+M_t+D_t,\qquad t\geq0.
\end{equation*}
The computation in Theorem~\ref{thm:2-6} identifies that
\begin{equation*}
M_t=\int_0^t z_s^{\top}dW_s,\qquad t\geq0,
\end{equation*}
and, thus, we have that
\begin{equation}
du_t
=z_t^{\top}dW_t
-F(t,c_t,u_t,z_t)\,dt
+dD_t,\qquad t\geq0.
\label{eq:proof_normalization_u}
\end{equation}
Along an optimal control pair, $D\equiv 0$.

By definition, $\bar U(t,X_t)=\varphi(u_t)$. It\^{o}'s formula gives
\begin{equation}
d\bar U(t,X_t)
=\varphi'(u_t)\,du_t
+\frac12\varphi''(u_t)|z_t|^2\,dt,\qquad t\geq0.
\label{eq:normalization_ito}
\end{equation}
Next, we substitute \eqref{eq:proof_normalization_u} and use the equality
\begin{equation*}
F(t,c_t,u_t,z_t)
=f(t,c_t,u_t)+\frac12A(u_t)|z_t|^2,\qquad t\geq0,
\end{equation*}
into \eqref{eq:normalization_ito}. Then, the coefficient of $|z_t|^2$ becomes $
\frac12\left(
\varphi''(u_t)-A(u_t)\varphi'(u_t)
\right)
$ and vanishes since $\varphi''(u)=A(u)\varphi'(u)$. Therefore, we have
\begin{align*}
d\bar U(t,X_t)
&=\varphi'(u_t)z_t^{\top}dW_t
-\varphi'(u_t)f(t,c_t,u_t)\,dt
+\varphi'(u_t)\,dD_t\\
&=\varphi'(u_t)z_t^{\top}dW_t
-\bar F\bigl(t,c_t,\bar U(t,X_t)\bigr)\,dt
+\varphi'(u_t)\,dD_t.
\end{align*}
The stochastic integral is a continuous local martingale. Since
$\varphi'(u_t)>0$ and $D$ is non-increasing, the process
\begin{equation*}
\int_0^t\varphi'(u_s)\,dD_s,\qquad t\geq0,
\end{equation*}
is non-increasing. It, then, follows that the transformed criterion is a local
supermartingale, $t\geq0$, and we easily conclude.

\smallskip

\noindent \textit{Step 2. Local integrability.}
On a finite time interval, the continuous process $u_t=U(t,X_t)$, has compact
range. Since, on the other hand, $\varphi'$ is continuous and strictly positive, both
$\varphi'(u_t)$ and $1/\varphi'(u_t)$ are bounded on that interval. The
continuity of $A$ also makes $A(u_t)$ bounded on a finite time interval.
Moreover, $
\int_0^T|z_t|^2\,dt<\infty.$ Therefore, the identity
\begin{equation*}
f(t,c_t,u_t)
=F(t,c_t,u_t,z_t)-\frac12A(u_t)|z_t|^2,\qquad t\geq0,
\end{equation*}
implies
\begin{align*}
\int_0^T|f(t,c_t,u_t)|\,dt
&\leq
\int_0^T|F(t,c_t,u_t,z_t)|\,dt -\frac12\int_0^T A(u_t)|z_t|^2\,dt
<\infty,\qquad t\geq0.
\end{align*}
The same identity proves that the local
integrability of $f(t,c_t,u_t)$ implies the local integrability of $F$. Hence,
\begin{equation*}
\int_0^T|F(t,c_t,u_t,z_t)|\,dt<\infty
\quad\Longleftrightarrow\quad
\int_0^T|f(t,c_t,u_t)|\,dt<\infty.
\end{equation*}
On the other hand,
\begin{equation*}
\bar F\bigl(t,c_t,\varphi(u_t)\bigr)
=\varphi'(u_t)f(t,c_t,u_t),\qquad t\geq0.
\end{equation*}
Therefore, the upper and lower pathwise bounds for $\varphi'(u_t)$ yield that
\begin{equation*}
\int_0^T|\bar F(t,c_t,\bar U(t,X_t))|\,dt<\infty
\quad\Longleftrightarrow\quad
\int_0^T|f(t,c_t,u_t)|\,dt<\infty,\qquad \text{for each }T>0.
\end{equation*}
Combining the two equivalences proves that the required local
integrability can be obtained in both directions between the original and the
normalized systems.

\smallskip

\noindent \textit{Step 3. The inverse transformation.}
Let $\psi=\varphi^{-1}$. Differentiating
$\varphi(\psi(\bar u))=\bar u$ gives
\begin{align*}
\psi'(\bar u)
&=\frac{1}{\varphi'(\psi(\bar u))},\\
\psi''(\bar u)
&=-\frac{\varphi''(\psi(\bar u))}
{\varphi'(\psi(\bar u))^3}\\
&=-A(\psi(\bar u))\psi'(\bar u)^2.
\end{align*}
The martingale coefficient of $\bar U(t,X_t)=\varphi(u_t)$ is
$\bar z_t=\varphi'(u_t)z_t$, $t\geq0$.
Applying It\^{o}'s formula to
$U(t,X_t)=\psi(\bar U(t,X_t))$, and using the above formula for
$\psi''$, produces the term $-\frac12A(u_t)|z_t|^2$.
This term eliminates the quadratic part of $F$, while $\psi'>0$ preserves the
sign of the non-increasing finite variation term. Together with Step~2, this
proves consistency in the inverse direction and preserves the equality case.

\smallskip

\noindent \textit{Step 4. Monotonicity and concavity.}
Fix $t\geq0$. Since $\varphi$ is strictly increasing, the strict
monotonicity of $U(t,\cdot)$ implies that of $\bar U(t,\cdot)$. For
$x_1\ne x_2$ and $a\in(0,1)$, strict concavity of $U(t,\cdot)$ and
monotonicity and concavity of $\varphi$ give
\begin{align*}
\bar U\bigl(t,ax_1+(1-a)x_2\bigr)
&>\varphi\bigl(aU(t,x_1)+(1-a)U(t,x_2)\bigr)\\
&\geq a\bar U(t,x_1)+(1-a)\bar U(t,x_2),
\end{align*}
and, thus, $\bar U(t,\cdot)$ is strictly concave. We note that this argument does not require
the pointwise condition $U_{xx}<0$.
\end{proof}

\begin{proof}[Proof of Theorem~\ref{thm:2-10}]
We verify the two Hamiltonian inequalities in
Theorem~\ref{thm:2-6} and determine when they hold as equalities.

To this end, let $x>0$, $\pi\in\mathbb R^n$, $c>0$, and for
$dt\otimes d\mathbb P$-a.e.\ $(t,\omega)$, define
\begin{equation*}
v=x\sigma_t^{\top}\pi
\in\operatorname{Im}\sigma_t^{\top},\qquad t\geq0,
\end{equation*}
and
\begin{equation*}
\Lambda(t,x)
=U_x(t,x)\theta_t
+\sigma_t^{+}\sigma_t\delta_x(t,x),\qquad t\geq0.
\end{equation*}
Since $v$ belongs to $\operatorname{Im}\sigma_t^{\top}$, we have
\begin{equation*}
v^{\top}\delta_x(t,x)
=v^{\top}\sigma_t^{+}\sigma_t\delta_x(t,x),\qquad t\geq0,\ x>0.
\end{equation*}
Therefore, the Hamiltonian can be written as
\begin{equation} \label{eq:Hamiltonian}
H(t,\omega,x,\pi,c)
=b(t,x)
+\frac12U_{xx}(t,x)|v|^2
+v^{\top}\Lambda(t,x)
+F(t,c,U(t,x))
-U_x(t,x)c,\qquad t\geq0,\ x>0.
\end{equation}

\smallskip

\noindent \textit{Step 1. The investment term.}
The equation \eqref{eq:SPDE-norm} gives that the drift satisfies
\begin{equation*}
b(t,x)
=\frac{|\Lambda(t,x)|^2}{2U_{xx}(t,x)}
-\widetilde F\bigl(t,U_x(t,x),U(t,x)\bigr),\qquad t\geq0,\ x>0.
\end{equation*}
Substituting the above into the Hamiltonian \eqref{eq:Hamiltonian} separates the portfolio
term from the consumption one. Specifically, the former
becomes
\begin{align*}
&\frac{|\Lambda(t,x)|^2}{2U_{xx}(t,x)}
+\frac12U_{xx}(t,x)|v|^2
+v^{\top}\Lambda(t,x) 
=\frac12U_{xx}(t,x)
\left|v+\frac{\Lambda(t,x)}{U_{xx}(t,x)}\right|^2,\qquad t\geq0,\ x>0.
\end{align*}
Since $U_{xx}(t,x)<0$, this expression is non-positive and becomes equal to zero
if and only if
\begin{equation}
v=-\frac{\Lambda(t,x)}{U_{xx}(t,x)},\qquad t\geq0,\ x>0.
\label{eq:proof_portfolio_equality}
\end{equation}
The vector $\Lambda(t,x)$ belongs to
$\operatorname{Im}\sigma_t^{\top}$. Because $\sigma_t$ is of full row rank,
\eqref{eq:proof_portfolio_equality} determines a unique portfolio. In turn, solving
$x\sigma_t^{\top}\pi=v$ gives
\begin{align*}
\pi
&=-\frac{1}{xU_{xx}(t,x)}
(\sigma_t^{\top})^{+}\Lambda(t,x)\\
&=-\frac{U_x(t,x)}{xU_{xx}(t,x)}
(\sigma_t^{\top})^{+}
\left(\theta_t+\frac{\delta_x(t,x)}{U_x(t,x)}\right),\qquad t\geq0,
\end{align*}
and \eqref{eq:feedback} follows.

\smallskip

\noindent \textit{Step 2. The consumption term.}
Set $d=U_x(t,x)$, $t\geq0$, $x>0$. Then, $d>0$ by strict monotonicity and concavity. The definition of
Fenchel transform gives, for every $c>0$,
\begin{equation*}
F(t,c,U(t,x))-dc
\leq\widetilde F(t,d,U(t,x)), \qquad t \geq 0, ~ x >0.
\end{equation*}
Then,
\begin{equation*}
F(t,c,U(t,x))-U_x(t,x)c
-\widetilde F\bigl(t,U_x(t,x),U(t,x)\bigr)\leq0, \qquad t \geq 0, ~ x >0.
\end{equation*}
Assumption~\ref{asm:2-9} and
\eqref{eq:F-envelope} show that equality holds if and only if
\begin{equation*}
c=F_c^{-1}\bigl(t,U_x(t,x),U(t,x)\bigr)
=c^{\star}(t,x), \qquad t \geq 0, ~ x >0.
\end{equation*}

\smallskip

\noindent \textit{Step 3. Application of Theorem~\ref{thm:2-6}.}
Steps~1 and~2 prove $H(t,\omega,x,\pi,c)\leq0$ for every $\pi\in\mathbb R^n$ and $c>0$. Equality holds at $\bigl(\pi,c\bigr)
=\bigl(\pi^{\star}(t,x),c^{\star}(t,x)\bigr)$ and, thus, these two statements are
exactly \eqref{eq:HJB-ineq} and \eqref{eq:HJB-eq}.

The optimal feedback portfolio is progressively measurable because it is constructed
from the progressively measurable fields
$U_x(t,x)$, $U_{xx}(t,x)$, and $\delta_x(t,x)$ together with the predictable
processes $\sigma_t$ and $\theta_t$. The optimal feedback consumption control is
progressively measurable because $F_c^{-1}(t,d,u)$ is measurable in
$(t,\omega)$ and continuous in $(d,u)$, while $U_x(t,x)$ and $U(t,x)$ are
progressively measurable.

The theorem assumes that the wealth equation with the candidate optimal feedback controls has a
unique strictly positive strong solution and that the resulting feedback
control pair is admissible. All the hypotheses of Theorem~\ref{thm:2-6} are therefore
satisfied. Its conclusion proves consistency and optimality.
\end{proof}

\section{Technical lemmas and detailed proofs for the duality theory}\label{app:B}

We present the technical lemmas and stochastic calculus computations
underlying Proposition~\ref{prop:3-4}, Theorem~\ref{thm:3-7}, Proposition~\ref{prop:3-10}, Theorem~\ref{thm:3-12}, and
Theorem~\ref{thm:3-14}. Throughout, we work systematically with the dual domain, using the convex conjugate $\widetilde U$, its consumption analogue $%
\widetilde F$, and the inverse marginal $I(t,y)$. For transparency, we retain all time, state,
and dual arguments in the calculations.

\begin{lemma}[Conjugacy identities and characteristics]\label{lem:3-3}
Under Assumption~\ref{asm:3-2}, the convex dual $\widetilde U$
defined in \eqref{eq:fenchel_U} is a semimartingale random field and
satisfies \eqref{eq:conjugacy_relations}--\eqref{eq:delta_x_via_dual}.
\end{lemma}

\begin{proof}
We first establish the deterministic conjugacy
identities and, in turn, compute the semimartingale characteristics.

\smallskip

\noindent \textit{Step 1. The Legendre point and its derivatives.}
Fix $(t,\omega)$ in the a.s. event on which Assumption~\ref{asm:3-2} holds,
let $y>0$.
The strictly concave map $x\mapsto U(t,x)-xy$ has the unique maximizer
$x=I(t,y)$, characterized by
\begin{equation}
U_x(t,I(t,y))=y,\qquad y>0.
\label{eq:proof_inverse_marginal_identity}
\end{equation}
Thus,
\begin{equation}
\widetilde U(t,y)
=U(t,I(t,y))-yI(t,y),\qquad y>0,\ t\geq0.
\label{eq:proof_conjugate_at_I}
\end{equation}
Differentiating \eqref{eq:proof_conjugate_at_I} and using
\eqref{eq:proof_inverse_marginal_identity} gives
\begin{equation*}
\widetilde U_y(t,y)
=U_x(t,I(t,y))I_y(t,y)-I(t,y)-yI_y(t,y)=-I(t,y),\qquad y>0,\ t\geq0.
\end{equation*}
Differentiating \eqref{eq:proof_inverse_marginal_identity} yields
\begin{equation*}
U_{xx}(t,I(t,y))I_y(t,y)=1.
\end{equation*}
\begin{equation*}
I_y(t,y)=\frac{1}{U_{xx}(t,I(t,y))},\qquad
\widetilde U_{yy}(t,y)=-I_y(t,y)
=-\frac{1}{U_{xx}(t,I(t,y))},
\end{equation*}
which are the identities in \eqref{eq:conjugacy_relations}.

\smallskip

\noindent \textit{Step 2. The diffusion coefficient of $I(t,y)$.}
For fixed $y>0$, we write the semimartingale decomposition
\begin{equation*}
dI(t,y)
=\mu^I(t,y)\,dt+\rho^I(t,y)^{\top}dW_t,\qquad t\geq0,\ y>0.
\end{equation*}
Differentiating \eqref{eq:U-decomp} with respect to $x$ gives
\begin{equation*}
dU_x(t,x)
=b_x(t,x)\,dt+\delta_x(t,x)^{\top}dW_t,\qquad t\geq0,\ x>0.
\end{equation*}
Recalling \eqref{eq:proof_inverse_marginal_identity} and using the
It\^{o}--Ventzel formula therefore gives
\begin{align*}
0
&=b_x(t,I(t,y))\,dt
+\delta_x(t,I(t,y))^{\top}dW_t\\
&\quad
+U_{xx}(t,I(t,y))\,dI(t,y)
+\frac12U_{xxx}(t,I(t,y))\,d\langle I(\cdot,y)\rangle_t\\
&\quad
+\delta_{xx}(t,I(t,y))^{\top}
d\langle W,I(\cdot,y)\rangle_t,\qquad t\geq0,\ y>0.
\end{align*}
Its Brownian coefficient vanishes, and hence
\begin{equation*}
\delta_x(t,I(t,y))
+U_{xx}(t,I(t,y))\rho^I(t,y)=0,\qquad t\geq0,\ y>0,
\end{equation*}
which implies
\begin{equation}
\begin{aligned}
\rho^I(t,y)
&=-\frac{\delta_x(t,I(t,y))}
{U_{xx}(t,I(t,y))},\\
d\langle I(\cdot,y)\rangle_t
&=\frac{|\delta_x(t,I(t,y))|^2}
{U_{xx}(t,I(t,y))^2}\,dt,\\
d\langle W,I(\cdot,y)\rangle_t
&=-\frac{\delta_x(t,I(t,y))}
{U_{xx}(t,I(t,y))}\,dt.
\end{aligned}
\label{eq:I_diffusion_aux}
\end{equation}

\smallskip

\noindent \textit{Step 3. Characteristics of $\widetilde U(t,y)$.}
From \eqref{eq:proof_conjugate_at_I} and the It\^{o}--Ventzel formula, we obtain
\begin{align*}
d\widetilde U(t,y)
&=dU(t,I(t,y))-y\,dI(t,y)\\
&=\left(b(t,I(t,y))
+\frac12U_{xx}(t,I(t,y))|\rho^I(t,y)|^2
+\delta_x(t,I(t,y))^\top\rho^I(t,y)\right)dt\\
&\quad+\delta(t,I(t,y))^\top dW_t.
\end{align*}
Substituting $\rho^I(t,y)=-\delta_x(t,I(t,y))/U_{xx}(t,I(t,y))$ yields
\begin{equation*}
d\widetilde U(t,y)
=\left(b(t,I(t,y))
-\frac{|\delta_x(t,I(t,y))|^2}{2U_{xx}(t,I(t,y))}\right)dt
+\delta(t,I(t,y))^\top dW_t,\qquad t\geq0,\ y>0,
\end{equation*}
and \eqref{eq:fenchel_diff_rule} follows.

\smallskip

\noindent \textit{Step 4. Relation between the primal and dual
volatilities.}
Differentiating $\widetilde\delta(t,y)=\delta(t,I(t,y))$ with respect to
$y$ gives
\begin{equation*}
\widetilde\delta_y(t,y)
=\delta_x(t,I(t,y))I_y(t,y)=\frac{\delta_x(t,I(t,y))}
{U_{xx}(t,I(t,y))},\qquad t\geq0,\ y>0.
\end{equation*}
Using that
$U_{xx}(t,I(t,y))=-1/\widetilde U_{yy}(t,y)$ gives
\begin{equation*}
\delta_x(t,I(t,y))
=-\frac{\widetilde\delta_y(t,y)}
{\widetilde U_{yy}(t,y)},
\end{equation*}
and \eqref{eq:delta_x_via_dual} follows. The identities in Step~1 hold
pathwise, while \eqref{eq:fenchel_diff_rule} and
\eqref{eq:delta_x_via_dual}, obtained in Steps~2--4, hold for $dt\otimes d\mathbb P$-a.e.\ $(t,\omega )$ and
every $y>0$, as stated above.
\end{proof}

\begin{proof}[Proof of Proposition~\ref{prop:3-4}]
Fix $(t,\omega)$ in a set of full
$dt\otimes d\mathbb P$-measure on which \eqref{eq:fenchel_diff_rule} and
\eqref{eq:delta_x_via_dual} hold. Let $y>0$, and set $P_t=\sigma_t^+\sigma_t$, $Q_t=I_d-P_t$, and
$x=I(t,y)$.
At $x=I(t,y)$, the conjugacy identities give that
\begin{equation*}
U_x(t,x)=y
\qquad \text{and}\qquad
U(t,x)=\hat u(t,y),\qquad t\geq0,\ y>0.
\end{equation*}
Using \eqref{eq:SPDE-norm}, \eqref{eq:fenchel_diff_rule},
$P_t\theta_t=\theta_t$, and
$|\delta_x|^2=|P_t\delta_x|^2+|Q_t\delta_x|^2$, we obtain
\begin{align*}
\widetilde b(t,y)
&=b(t,x)-\frac{|\delta_x(t,x)|^2}{2U_{xx}(t,x)}\\
&=-\widetilde F\bigl(t,y,\hat u(t,y)\bigr)
+\frac{y^2|\theta_t|^2+2y\theta_t^\top\delta_x(t,x)
-|Q_t\delta_x(t,x)|^2}{2U_{xx}(t,x)}\\
&=-\widetilde F\bigl(t,y,\hat u(t,y)\bigr)
-\frac12y^2\widetilde U_{yy}(t,y)|\theta_t|^2
+y\theta_t^\top\widetilde\delta_y(t,y)
+\frac{|Q_t\widetilde\delta_y(t,y)|^2}
{2\widetilde U_{yy}(t,y)},\qquad t\geq0,\ y>0,
\end{align*}
and \eqref{eq:dual_drift} follows. The volatility identity is the first identity
in \eqref{eq:fenchel_diff_rule}.
\end{proof}

\begin{proof}[Proof of Theorem~\ref{thm:3-7}]
We compute the drift of the dual criterion, verify
its sign, and identify the case where equality is satisfied.

\smallskip

\noindent \textit{Step 1. Characteristics of the state price density.}
Fix $\nu\in\mathcal A^{\mathrm{dual}}$, write $Y=Y^\nu$, and set
$P_t=\sigma_t^+\sigma_t$ and $Q_t=I_d-P_t$. From \eqref{eq:SPD},
\begin{equation*}
d\langle Y\rangle_t
=Y_t^2\bigl(|\theta_t|^2+|\nu_t|^2\bigr)dt
\qquad \text{and}\qquad
d\langle W,Y\rangle_t=Y_t(-\theta_t+\nu_t)dt,\qquad t\geq0,
\end{equation*}
since $\theta_t^\top\nu_t=0$, for $dt\otimes d\mathbb P$-a.e.\ $(t,\omega )$.

\smallskip

\noindent \textit{Step 2. Dynamics of the dual criterion.}
Under Assumption~\ref{asm:3-2}, the local characteristics $(%
\widetilde{b},\widetilde{\delta })$ of $\widetilde{U}$ computed in %
\eqref{eq:fenchel_diff_rule} are continuously differentiable in $y$ with
locally H\"{o}lder derivatives, which are compositions of $(b,\delta )$ and
their spatial derivatives with the field $I$. Therefore, the It\^{o}--Ventzel formula
gives
\begin{equation*}
\begin{aligned}
d\widetilde{U}(t,Y_{t})=\Bigl(&\widetilde{b}(t,Y_{t})+\tfrac{1}{2}\widetilde{%
U}_{yy}(t,Y_{t})\,Y_{t}^{2}\bigl(|\theta _{t}|^{2}+|\nu _{t}|^{2}\bigr)\\
&+Y_{t}\, \widetilde{\delta }_{y}(t,Y_{t})^{\top }\bigl(-\theta _{t}+\nu
_{t}\bigr)\Bigr)dt+\bigl(\widetilde{Z}_{t}^{\nu }\bigr)^{\top }dW_{t},\qquad t\geq0.
\end{aligned}
\end{equation*}%
The volatility process in the above expansion yields
\begin{equation*}
\widetilde{Z}_{t}^{\nu }
=\widetilde{\delta }(t,Y_{t})
+\widetilde{U}_{y}(t,Y_{t})Y_{t}
\bigl(-\theta _{t}+\nu _{t}\bigr),\qquad t\geq0.
\end{equation*}
Together with \eqref{eq:Zcal_def} and
\eqref{eq:dual_drift}, we have
\begin{align}
d\mathcal Z_t^\nu
&=(\widetilde Z_t^\nu)^\top dW_t
+\left(
\frac{|Q_t\widetilde\delta_y(t,Y_t)|^2}
{2\widetilde U_{yy}(t,Y_t)}
+\frac12Y_t^2\widetilde U_{yy}(t,Y_t)|\nu_t|^2
+Y_t\widetilde\delta_y(t,Y_t)^\top\nu_t
\right)dt\notag\\
&=(\widetilde Z_t^\nu)^\top dW_t
+\frac{|Q_t\widetilde\delta_y(t,Y_t)
+Y_t\widetilde U_{yy}(t,Y_t)\nu_t|^2}
{2\widetilde U_{yy}(t,Y_t)}\,dt,\qquad t\geq0.
\label{eq:proof_dual_square}
\end{align}
Here $Q_t\nu_t=\nu_t$. Moreover,
\begin{equation*}
\widetilde{U}_{yy}(t,y)
=-\frac{1}{U_{xx}(t,I(t,y))}>0,\qquad t\geq0,\ y>0.
\end{equation*}
Hence the drift in \eqref{eq:proof_dual_square} is nonnegative and vanishes
if and only if
\begin{equation*}
\nu
=-\frac{(I_d-\sigma_t^{+}\sigma_t)\widetilde\delta_y(t,y)}
{y\widetilde U_{yy}(t,y)},
\end{equation*}
and \eqref{eq:nu_dual_def} follows.

For later use, we denote the nonnegative drift in
\eqref{eq:proof_dual_square} by
\begin{equation*}
g_t^\nu
=\frac{|Q_t\widetilde\delta_y(t,Y_t)
+Y_t\widetilde U_{yy}(t,Y_t)\nu_t|^2}
{2\widetilde U_{yy}(t,Y_t)},\qquad t\geq0.
\end{equation*}

\smallskip

\noindent \textit{Step 3. Integrability of the running term.}
We write
\begin{equation*}
d\widetilde U(t,Y_t)=a_t^\nu\,dt+(\widetilde Z_t^\nu)^\top dW_t,\qquad t\geq0.
\end{equation*}
Then, equation \eqref{eq:proof_dual_square} gives, for
$dt\otimes d\mathbb P$-a.e.\ $(t,\omega)$,
\begin{equation}
a_t^\nu+\widetilde F(t,Y_t,\hat u(t,Y_t))=g_t^\nu\geq0
\qquad \text{and}\qquad
\widetilde F(t,Y_t,\hat u(t,Y_t))^-
\leq (a_t^\nu)^+,\qquad t\geq0.
\label{eq:proof_dual_negative_part}
\end{equation}
For every $T>0$, localization to the compact range of the positive
continuous process $Y$ and the random field bounds in Appendix~\ref{app:A} give
$\int_0^T|a_t^\nu|dt<\infty$ $\mathbb P$-a.s. Thus
\eqref{eq:proof_dual_negative_part} yields that the negative part is
integrable on every finite horizon. Therefore, condition
\eqref{eq:Ftilde-integrability} may fail only through the positive part.

\smallskip

\noindent \textit{Step 4. Localization.}
Consider $\nu\in\mathcal A^{\mathrm{dual}}$ satisfying
\eqref{eq:Ftilde-integrability}, and set, for $n\geq1$,
\begin{equation*}
\tau _{n}=\inf \Bigl \{t\geq 0:Y_{t}\notin \bigl(\tfrac{1}{n},n\bigr)\ \text{%
or}\ \int_{0}^{t}\bigl|\widetilde{Z}_{s}^{\nu }\bigr|^{2}\,ds\geq n\ \text{or%
}\ \int_{0}^{t}g_s^\nu\,ds\geq n\Bigr \},\qquad t>0.
\end{equation*}%

Fix $T>0$. Since $Y$ is continuous and strictly positive, its range on
$[0,T]$ is a compact subset of $(0,\infty)$, $\mathbb P$-a.s.

The local regularity of $\widetilde U_y$ and the integrated bound on
$\widetilde\delta=\delta(\cdot,I(\cdot,\cdot))$ along this range, together
with the pathwise square-integrability of $\theta$ and $\nu$, yield
$\int_0^T|\widetilde Z_s^\nu|^2\,ds<\infty$, $\mathbb P$-a.s.

Moreover, Step~3 gives $\int_0^T|a_s^\nu|\,ds<\infty$,
$\mathbb P$-a.s. Therefore, \eqref{eq:proof_dual_negative_part} and
\eqref{eq:Ftilde-integrability} imply that
$\int_0^Tg_s^\nu\,ds<\infty$, $\mathbb P$-a.s.

Hence, we must have that none of the three conditions defining $\tau_n$ is
reached before $T$ for sufficiently large $n$. Since $T>0$ is arbitrary,
$\tau_n\uparrow\infty$, $\mathbb P$-a.s.

On each $[0,\tau _{n}]$, the
stopped stochastic integral $\int_{0}^{\cdot \wedge \tau _{n}}(\widetilde{Z}%
_{s}^{\nu })^{\top }dW_{s}$ is a square integrable martingale, while
the stopped finite variation part is bounded by $n$ in total variation, and thus
$\mathcal{Z}_{\cdot \wedge \tau _{n}}^{\nu }$ is integrable. Since
$g^\nu\geq0$, the process $\mathcal{Z}^{\nu }$ has nonnegative drift.
Thus, $\mathcal{Z}_{\cdot \wedge \tau _{n}}^{\nu }$ is a submartingale for
every $n$. Letting $n\uparrow \infty $ shows that $\mathcal{Z}^{\nu }$ is a
local submartingale. This proves the first requirement of the consistency condition in Definition~\ref{defn:3-5}(i).

\smallskip

\noindent \textit{Step 5. Equality at the dual optimal density.}
Let $Y^{\nu^{\mathrm{dual}}}$ be the solution assumed in the theorem. By
\eqref{eq:nu_dual_def}, the completed square vanishes along this density,
$g_t^{\nu^{\mathrm{dual}}}=0$, for $dt\otimes d\mathbb P$-a.e.\ $(t,\omega )$. Then, from equation \eqref{eq:proof_dual_negative_part}, we deduce that
\begin{equation*}
\widetilde F\bigl(t,Y_t^{\nu^{\mathrm{dual}}},
\hat u(t,Y_t^{\nu^{\mathrm{dual}}})\bigr)
=-a_t^{\nu^{\mathrm{dual}}},\qquad t\geq0,
\end{equation*}
and thus \eqref{eq:Ftilde-integrability} holds. Using the same
localization as in Step~4, $\mathcal Z^{\nu^{\mathrm{dual}}}$ is a local
martingale. This proves the second requirement for the dual consistency condition in Definition~\ref{defn:3-5}(ii), so the
conjugate pair is dual consistent, with dual optimal control
$\nu^{\mathrm{dual}}$.
\end{proof}

\begin{proof}[Proof of Proposition~\ref{prop:3-10}]
We first prove the pointwise Fenchel equality and then
compute the complete dynamics of $\mathcal{Z}^{\alpha ^{\star },\nu }$.

\smallskip

\noindent \textit{Step 1. The pointwise Fenchel equality.}
Fix $(t,\omega)$ in a set of full $dt\otimes d\mathbb P$-measure on
which Assumptions~\ref{asm:2-9} and~\ref{asm:3-8} hold. Let $D>0$ and write $\hat u=\hat
u(t,D)$, $\alpha ^{\star }=\alpha ^{\star }(t,D)$ and $\hat c%
=F_{c}^{-1}(t,D,\hat u)$. Then,
\begin{equation*}
\alpha^\star=-F_u(t,\hat c,\hat u)\qquad \text{and}\qquad
F_c(t,\hat c,\hat u)=D.
\end{equation*}
Furthermore, the convexity in $u$ and concavity in $c$ give, for each $c>0$,
\begin{align*}
\widehat F(t,\hat c,\alpha^\star)
&=F(t,\hat c,\hat u)+\alpha^\star\hat u,\\
\widehat F(t,c,\alpha^\star)-Dc
&\leq F(t,c,\hat u)+\alpha^\star\hat u-Dc\\
&\leq \widetilde F(t,D,\hat u)+\alpha^\star\hat u.
\end{align*}
Both inequalities become equalities at $c=\hat c$. Taking the supremum over
$c>0$ proves \eqref{eq:G_saddle}.

\smallskip

\noindent \textit{Step 2. Characteristics of the deflated state price
density.}
Write $D_t=Y_t^\nu/\kappa_{0,t}^{\alpha^\star}$ and
$\kappa_t=\kappa_{0,t}^{\alpha^\star}$. Since
$d\kappa_t=-\alpha_t^\star\kappa_tdt$ and
$dY_t^\nu=Y_t^\nu(-\theta_t+\nu_t)^\top dW_t$,
\begin{equation*}
dD_t=\alpha_t^\star D_tdt+D_t(-\theta_t+\nu_t)^\top dW_t,
\qquad
d\langle D\rangle_t=D_t^2(|\theta_t|^2+|\nu_t|^2)dt,
\end{equation*}
and $d\langle W,D\rangle_t=D_t(-\theta_t+\nu_t)dt$.

\smallskip

\noindent \textit{Step 3. Dynamics of the envelope.}
Set $Q_t=I_d-\sigma_t^+\sigma_t$ and
$\hat u_t=\hat u(t,D_t)$. By \eqref{eq:G-integrability}, the
finite variation term of \eqref{eq:dual_envelope_def} converges absolutely on
every finite horizon and, thus, $\mathcal{Z}^{\alpha ^{\star },\nu }$ is a
well defined continuous process for $t\geq0$. The It\^{o}--Ventzel formula and the product
rule give
\begin{equation*}
\begin{aligned}
d\mathcal Z_t^{\alpha^\star,\nu}
&=dM_t^\nu+\kappa_t\Bigl(
\widetilde b(t,D_t)+\alpha_t^\star D_t\widetilde U_y(t,D_t)
-\alpha_t^\star\widetilde U(t,D_t)\\
&\hspace{19mm}
+\tfrac12D_t^2\widetilde U_{yy}(t,D_t)
\bigl(|\theta_t|^2+|\nu_t|^2\bigr)
+D_t\widetilde\delta_y(t,D_t)^\top(-\theta_t+\nu_t)
+G(t,D_t,\alpha_t^\star)\Bigr)dt\\
&=dM_t^\nu+\kappa_t\Bigl(
\widetilde b(t,D_t)+\widetilde F(t,D_t,\hat u_t)
+\tfrac12D_t^2\widetilde U_{yy}(t,D_t)
\bigl(|\theta_t|^2+|\nu_t|^2\bigr)\\
&\hspace{48mm}
+D_t\widetilde\delta_y(t,D_t)^\top(-\theta_t+\nu_t)
\Bigr)dt\\
&=dM_t^\nu+\kappa_t\Bigl(
\frac{|Q_t\widetilde\delta_y(t,D_t)|^2}
{2\widetilde U_{yy}(t,D_t)}
+\tfrac12D_t^2\widetilde U_{yy}(t,D_t)|\nu_t|^2
+D_t\widetilde\delta_y(t,D_t)^\top\nu_t\Bigr)dt\\
&=dM_t^\nu+\kappa_t
\frac{|Q_t\widetilde\delta_y(t,D_t)
+D_t\widetilde U_{yy}(t,D_t)\nu_t|^2}
{2\widetilde U_{yy}(t,D_t)}\,dt,
\end{aligned}
\end{equation*}
where we use \eqref{eq:G_saddle} in the second equality,
\eqref{eq:dual_drift} in the third, and the fact that
\begin{equation*}
dM_t^\nu=\kappa_t\bigl(\widetilde\delta(t,D_t)
+D_t\widetilde U_y(t,D_t)(-\theta_t+\nu_t)\bigr)^\top dW_t,\qquad t\geq0.
\end{equation*}
On every finite horizon, localization to the compact range of the positive
process $D$, together with the local random field bounds and the integrability
of $\alpha^\star$, $\theta$, and $\nu$, makes the stopped process $M^\nu$
square integrable and the displayed finite variation term finite.
By \eqref{eq:nondegeneracy} and \eqref{eq:conjugacy_relations},
$\widetilde U_{yy}(t,D_t)=-1/U_{xx}(t,I(t,D_t))>0$, so the finite variation
term in \eqref{eq:deflated_envelope_dynamics} is nonnegative and
$\mathcal{Z}^{\alpha ^{\star },\nu }$ is a local submartingale. If it is a
local martingale, however, its finite variation term is a continuous local martingale of finite
variation vanishing at the origin and, hence, identically zero. Since $\widetilde
U_{yy}>0$, this occurs exactly when \eqref{eq:nu_equality_case} holds.
Conversely, \eqref{eq:nu_equality_case} makes the term vanish. This proves
\eqref{eq:deflated_envelope_dynamics} and the stated equivalence.
\end{proof}

\begin{lemma}[Marginal utility dynamics]\label{lem:3-11}
Under Assumption~\ref{asm:3-2} and the hypotheses of Theorem~\ref{thm:2-10}, we include the
associated wealth process $X^{\star }$ and set
$U_t^\star=U(t,X_t^\star)$ and
$c_t^\star=c^\star(t,X_t^\star)$, $t\geq0$. Then,
\eqref{eq:Ux_SDE} and \eqref{eq:marginal_coeff_integrability} hold, with the
orthogonal control $\nu^\star$ defined in
\eqref{eq:nu_star_marginal}.
\end{lemma}

\begin{proof}
Set
\begin{equation*}
P_t=\sigma_t^+\sigma_t,\qquad Q_t=I_d-P_t,\qquad
\Lambda(t,x)=U_x(t,x)\theta_t+\delta_x(t,x),\qquad t\geq0.
\end{equation*}
In the formulas below, all fields without an explicit state argument are
evaluated at $(t,X_t^\star)$. We set
\begin{equation*}
p_t=P_t\Lambda(t,X_t^\star),\qquad
\sigma_t^\star=-\frac{p_t}{U_{xx}},\qquad
dX_t^\star=\bigl((\sigma_t^\star)^\top\theta_t-c_t^\star\bigr)dt
+(\sigma_t^\star)^\top dW_t,\qquad t\geq0.
\end{equation*}
The envelope identities \eqref{eq:F-envelope} and
$F_c(t,c_t^\star,U_t^\star)=U_x$ yield
\begin{equation}  \label{eq:bx_at_Xstar}
b_x
=\Lambda^\top\theta_t+\frac{p_t^\top\delta_{xx}}{U_{xx}}
-\frac{|p_t|^2U_{xxx}}{2U_{xx}^2}
+c_t^\star U_{xx}-F_u(t,c_t^\star,U_t^\star)U_x,\qquad t\geq0.
\end{equation}
The It\^{o}--Ventzel formula gives
\begin{equation}  \label{eq:itoventzel_Ux}
\begin{aligned}
dU_x(t,X_t^\star)
&=\Bigl(b_x+U_{xx}\bigl((\sigma_t^\star)^\top\theta_t-c_t^\star\bigr)
+\tfrac12U_{xxx}|\sigma_t^\star|^2
+\delta_{xx}^\top\sigma_t^\star\Bigr)dt\\
&\quad+\bigl(\delta_x+U_{xx}\sigma_t^\star\bigr)^\top dW_t,\qquad t\geq0.
\end{aligned}
\end{equation}
For the drift in \eqref{eq:itoventzel_Ux}, \eqref{eq:bx_at_Xstar} implies
\begin{align*}
&b_x+U_{xx}\bigl((\sigma_t^\star)^\top\theta_t-c_t^\star\bigr)
+\tfrac12U_{xxx}|\sigma_t^\star|^2
+\delta_{xx}^\top\sigma_t^\star\\
&=\Lambda^\top\theta_t+\frac{p_t^\top\delta_{xx}}{U_{xx}}
-\frac{|p_t|^2U_{xxx}}{2U_{xx}^2}
+c_t^\star U_{xx}-F_u(t,c_t^\star,U_t^\star)U_x\\
&\quad-\Lambda^\top\theta_t-c_t^\star U_{xx}
+\frac{|p_t|^2U_{xxx}}{2U_{xx}^2}
-\frac{p_t^\top\delta_{xx}}{U_{xx}}\\
&=-F_u(t,c_t^\star,U_t^\star)U_x,\qquad t\geq0.
\end{align*}
Its diffusion coefficient is given by
\begin{equation*}
\delta_x+U_{xx}\sigma_t^\star
=\delta_x-p_t=-U_x\theta_t+Q_t\delta_x
=U_x(-\theta_t+\nu_t^\star),
\qquad
\nu_t^\star=\frac{Q_t\delta_x}{U_x},\qquad t\geq0.
\end{equation*}
Thus, the proof of \eqref{eq:Ux_SDE} is finished, and
$\nu_t^\star\in(\operatorname{Im}\sigma_t^\top)^\perp$ for
$dt\otimes d\mathbb P$-a.e.\ $(t,\omega)$.

For $T>0$, the paths of $X^\star$ remain in a compact subset of
$(0,\infty)$ on $[0,T]$, and, furthermore, the continuous process
$U_x(t,X_t^\star)$ achieves a strictly positive minimum there. Therefore, the
standing bounds imply
\begin{equation*}
\int_0^T|\nu_t^\star|^2dt<\infty,
\qquad
\int_0^T\bigl|F_u(t,c_t^\star,U_t^\star)U_x(t,X_t^\star)\bigr|dt<\infty
\qquad\text{$\mathbb P$-a.s.}
\end{equation*}
Dividing the second integrand by the pathwise lower bound for
$U_x(t,X_t^\star)$ proves
\eqref{eq:marginal_coeff_integrability}.
\end{proof}

\begin{proof}[Proof of Theorem~\ref{thm:3-12}]
We apply the product rule to the discounted marginal
utility and, then, verify the optimal feedback relation.

\smallskip

\noindent \textit{Step 1. Dynamics of the discounted marginal utility.}
Set $L_t=U_x(t,X_t^\star)$ and
$a_t=F_u(t,c_t^\star,U_t^\star)$, $t\geq0$. By
\eqref{eq:alpha_kappa_optimal} and Lemma~\ref{lem:3-11},
\begin{equation*}
d\kappa_t^\star=a_t\kappa_t^\star dt
\qquad \text{and}\qquad
dL_t=-a_tL_tdt+L_t(-\theta_t+\nu_t^\star)^\top dW_t,\qquad t\geq0.
\end{equation*}
Since $\kappa^\star$ has finite variation,
\begin{align*}
dY_t^\star
&=d(\kappa_t^\star L_t)
=\kappa_t^\star dL_t+L_t d\kappa_t^\star\\
&=\kappa_t^\star L_t(-\theta_t+\nu_t^\star)^\top dW_t
=Y_t^\star(-\theta_t+\nu_t^\star)^\top dW_t,\qquad t\geq0,
\qquad Y_0^\star=U_x(0,x),
\end{align*}
and \eqref{eq:optimal_SPD} follows. The integrability properties in
\eqref{eq:marginal_coeff_integrability} show simultaneously that
$\kappa^\star$ is finite and strictly positive on finite horizons and that
$\nu^\star$ is an admissible dual control. Since
$\nu _{t}^{\star }\in (\mathrm{Im}\,\sigma _{t}^{\top })^{\perp }$, the
process $Y^{\star }$ has the state price density form \eqref{eq:SPD} and,
thus, $Y^{\star }\in \mathcal{Y}$.

\smallskip

\noindent \textit{Step 2. The feedback relation along the optimum.}
Suppose that Assumption~\ref{asm:3-8} holds and that $\alpha^\star$ is admissible.
Since $D_t^\star=U_x(t,X_t^\star)$, $t\geq0$, we obtain
\begin{equation*}
\hat u(t,D_t^\star)=U_t^\star,\qquad
F_c^{-1}(t,D_t^\star,\hat u(t,D_t^\star))=c_t^\star,\qquad
\alpha^\star(t,D_t^\star)=-F_u(t,c_t^\star,U_t^\star)=\alpha_t^\star,\qquad t\geq0,
\end{equation*}
and the feedback relation \eqref{eq:coupled_discount_system} follows. Next, we show
that \eqref{eq:G-integrability} also holds, so that $Y^{\star }$ admits
$\alpha ^{\star }$ as a feedback discounting map. To this end, evaluating \eqref{eq:G_saddle}
at $(D_{t}^{\star },\alpha _{t}^{\star })$, and using that $c_{t}^{\star }$
is the maximizer of the Fenchel transform
$\widetilde{F}(t,D_{t}^{\star },U_{t}^{\star })$, we deduce that
\begin{equation*}
G(t,D_{t}^{\star },\alpha _{t}^{\star })
=F(t,c_{t}^{\star },U_{t}^{\star })-D_{t}^{\star }c_{t}^{\star }
+\alpha _{t}^{\star }U_{t}^{\star },\qquad t\geq 0.
\end{equation*}%
Fix $T>0$. The first term is integrable on $[0,T]$ by
\eqref{eq:F-integrability}, which holds automatically along the optimum, as
noted in Remark~\ref{rem:2-5}. For the second term,
$\int_{0}^{T}c_{t}^{\star }\,dt<\infty $ by Definition~\ref{defn:2-1} and $D^{\star }$ is
pathwise bounded on $[0,T]$ by continuity. For the third term,
$\int_{0}^{T}|\alpha _{t}^{\star }|\,dt<\infty $ by
\eqref{eq:marginal_coeff_integrability} and the fact that $U^{\star }$ is pathwise bounded
on $[0,T]$. Since $\kappa ^{\star }$ is continuous and strictly positive, we
have that
\begin{equation*}
\int_{0}^{T}\kappa _{t}^{\star }\bigl|G(t,D_{t}^{\star },\alpha _{t}^{\star
})\bigr|\,dt<\infty \qquad \text{$\mathbb{P}$-a.s.}
\end{equation*}%
Using
\eqref{eq:delta_x_via_dual}, evaluated at
$D_t^\star=U_x(t,X_t^\star)$,
\begin{equation*}
-\frac{(I_d-\sigma_t^+\sigma_t)\widetilde\delta_y(t,D_t^\star)}
{D_t^\star\widetilde U_{yy}(t,D_t^\star)}
=\frac{(I_d-\sigma_t^+\sigma_t)\delta_x(t,X_t^\star)}
{U_x(t,X_t^\star)}
=\nu_t^\star,\qquad t\geq0,
\end{equation*}
and \eqref{eq:nu_equality_case} follows. In turn, Proposition~\ref{prop:3-10} shows that
$\mathcal{Z}^{\alpha ^{\star },\nu ^{\star }}$ is a local martingale.
\end{proof}

\begin{proof}[Proof of Theorem~\ref{thm:3-14}]
We establish the assertions in several steps. To ease the presentation, we occasionally rewrite some formulae.

\smallskip

\noindent \textit{Step 1. Well-posedness of the deflated dual SDE.}
By Assumption~\ref{asm:3-2},
\begin{equation*}
U_x(t,I(t,y))=y\qquad \text{and}\qquad
I_y(t,y)=\frac{1}{U_{xx}(t,I(t,y))},\qquad t\geq0,\ y>0.
\end{equation*}
The bounds in \eqref{eq:input_bounds_delta} imply
\begin{equation*}
\begin{aligned}
|\sigma^D(t,y)|&\le|y|\bigl(|\theta_t|+K_t^1\bigr), \qquad
|\sigma^D_y(t,y)|\le|\theta_t|+K_t^2,\\
\sigma^D_y(t,y)&=-\theta_t+\frac{(I_{d}-\sigma_{t}^{+}\sigma_{t})\,
\delta_{xx}(t,I(t,y))}{U_{xx}(t,I(t,y))},\qquad t\geq0,\ y>0,
\end{aligned}
\end{equation*}
and \eqref{eq:input_bounds_alpha} gives 
\begin{equation*}
\begin{aligned}
|\mu^D(t,y)|&\le K_t^0|y|, \qquad |\mu^D_y(t,y)|\le K_t^0,\\
\mu^D_y(t,y)&=\alpha^\star(t,I(t,y)) +\frac{U_x(t,I(t,y))}{U_{xx}(t,I(t,y))}%
\, \partial_x\alpha^\star(t,I(t,y)),\qquad t\geq0,\ y>0,\ \text{respectively}.
\end{aligned}
\end{equation*}
The derivative bounds hold for every $y>0$, and, thus, the mean value theorem yields that
the coefficients of \eqref{eq:deflated_dual_SDE} are Lipschitz in $y$ on $%
(0,\infty )$ and of linear growth, with adapted constants that are pathwise
$L^{1}$- and $L^{2}$-integrable on $[0,T]$. This is the exact analogue of the
uniform Lipschitz property in \citet[Theorem~4.2(i)]{EKM13}. Thus, standard
localization gives a unique maximal strong solution of
\eqref{eq:deflated_dual_SDE} with values in $(0,\infty)$. The linear growth
bounds rule out explosion to $\infty$ on a finite time interval.

It remains to exclude exit from zero. To this end, let $\tau_0$ be the first
time at which $D_t^{\star}(y)$ reaches zero. On $[0,\tau_0)$, It\^{o}'s formula gives
\begin{align*}
d\log D_t^{\star}(y)
&=\left(
\frac{\mu^D(t,D_t^{\star}(y))}{D_t^{\star}(y)}
-\frac12
\frac{|\sigma^D(t,D_t^{\star}(y))|^2}
{D_t^{\star}(y)^2}
\right)dt\\
&\quad
+\frac{\sigma^D(t,D_t^{\star}(y))}
{D_t^{\star}(y)}\,dW_t,\qquad t\geq0.
\end{align*}
On $[0,\tau_0)$,
\begin{equation*}
\left|\frac{\mu^D(t,D_t^\star(y))}{D_t^\star(y)}\right|\leq K_t^0
\qquad \text{and}\qquad
\left|\frac{\sigma^D(t,D_t^\star(y))}{D_t^\star(y)}\right|
\leq|\theta_t|+K_t^1,\qquad t\geq0.
\end{equation*}
Thus, the finite variation part of $\log D^\star(y)$ has finite total
variation and its martingale part has finite quadratic variation on every
finite horizon, $\mathbb P$-a.s. Hence, $\log D_t^\star(y)$ cannot diverge to $-\infty$ at a
finite time, and thus $\tau_0=\infty$ $\mathbb P$-a.s. Continuity and pathwise uniqueness
imply that two solutions driven by the same Brownian motion cannot cross.
We easily deduce that $y\mapsto D_t^\star(y)$ is nondecreasing.

\smallskip

\noindent \textit{Step 2. Construction of the primal wealth by inverse
marginality.}
For $x>0$, define
\begin{equation}
X_{t}^{\star }(x)=I\bigl(t,D_{t}^{\star }(U_{x}(0,x))\bigr),\qquad t\geq 0.
\label{eq:Xstar_inverse_marginal}
\end{equation}%
By Assumption~\ref{asm:3-2}(ii), the Inada limits make $I(t,\cdot )$ map $(0,\infty )$
onto $(0,\infty )$, so $X^{\star }(x)$, $x>0$, is strictly positive. Since
$D^{\star }(U_{x}(0,x))$ is continuous and $I$ is jointly continuous, the
paths of $X^{\star }(x)$ are continuous and therefore remain in a compact
subset of $(0,\infty )$ on every finite horizon.

We now use the inverse marginal identity \eqref{eq:flow_identity} rather than a separate
Lipschitz theorem for the primal coefficients. We give the inverse marginal calculation
explicitly. For fixed $y>0$, we write
\begin{equation*}
dI(t,y)=\mu^I(t,y)\,dt+\rho^I(t,y)^\top dW_t,\qquad t\geq0.
\end{equation*}
Applying the It\^{o}--Ventzel formula to the identity
$U_x(t,I(t,y))=y$, as in \eqref{eq:I_diffusion_aux}, gives
\begin{align}
\rho^I(t,y)
&=-\frac{\delta_x}{U_{xx}},\notag\\
\mu^I(t,y)
&=-\frac{b_x}{U_{xx}}
-\frac{U_{xxx}|\delta_x|^2}{2U_{xx}^3}
+\frac{\delta_{xx}^\top\delta_x}{U_{xx}^2}.
\label{eq:inverse_marginal_characteristics}
\end{align}
Here each derivative on the right is evaluated at $(t,I(t,y))$.
The spatial derivatives needed below are
\begin{equation}  \label{eq:inverse_marginal_derivatives}
I_y=\frac1{U_{xx}},\qquad
I_{yy}=-\frac{U_{xxx}}{U_{xx}^3},\qquad
\rho_y^I=-\frac{\delta_{xx}}{U_{xx}^2}
+\frac{\delta_xU_{xxx}}{U_{xx}^3}.
\end{equation}

Set $D_t=D_t^\star(U_x(0,x))$ and $X_t=I(t,D_t)$, $t\geq0$. At the point
$(t,X_t)$, let
\begin{equation*}
P_t=\sigma_t^+\sigma_t,\qquad Q_t=I_d-P_t,\qquad
p_t=U_x(t,X_t)\theta_t+P_t\delta_x(t,X_t),\qquad t\geq0.
\end{equation*}
Since $D_t=U_x(t,X_t)$, the diffusion coefficient in
\eqref{eq:deflated_dual_SDE} is
\begin{equation*}
\sigma^D(t,D_t)=-U_x(t,X_t)\theta_t+Q_t\delta_x(t,X_t),\qquad t\geq0.
\end{equation*}
Applying the It\^{o}--Ventzel formula to $I(t,D_t)$, and using
\eqref{eq:inverse_marginal_characteristics} and
\eqref{eq:inverse_marginal_derivatives}, its volatility coefficient becomes
\begin{align*}
\rho^I(t,D_t)+I_y(t,D_t)\sigma^D(t,D_t)
&=-\frac{\delta_x}{U_{xx}}
+\frac{-U_x\theta_t+Q_t\delta_x}{U_{xx}}\\
&=-\frac{p_t}{U_{xx}}
=\sigma^\star(t,X_t).
\end{align*}

The spatial derivative of the Bellman drift, computed in Lemma~\ref{lem:3-11}, is given
by
\begin{equation*}
b_x
=(U_x\theta_t+\delta_x)^\top\theta_t
+\frac{p_t^\top\delta_{xx}}{U_{xx}}
-\frac{|p_t|^2U_{xxx}}{2U_{xx}^2}
+c^\star U_{xx}-F_uU_x,\qquad t\geq0.
\end{equation*}
On the other hand,
\begin{equation*}
\mu^D(t,D_t)=-F_uU_x,
\qquad
\delta_x-\sigma^D(t,D_t)=p_t,
\qquad
p_t^\top\theta_t=(U_x\theta_t+\delta_x)^\top\theta_t,\qquad t\geq0,
\end{equation*}
and, thus, the It\^{o}--Ventzel drift of $I(t,D_t)$ yields
\begin{align*}
&\mu^I+I_y\mu^D+\tfrac12I_{yy}|\sigma^D|^2
+(\rho_y^I)^\top\sigma^D\\
&=\frac{-b_x+\mu^D}{U_{xx}}
+\frac{\delta_{xx}^\top(\delta_x-\sigma^D)}{U_{xx}^2}
+\frac{U_{xxx}}{2U_{xx}^3}
\bigl(-|\delta_x|^2-|\sigma^D|^2
+2\delta_x^\top\sigma^D\bigr)\\
&=\frac{-b_x-F_uU_x}{U_{xx}}
+\frac{p_t^\top\delta_{xx}}{U_{xx}^2}
-\frac{|p_t|^2U_{xxx}}{2U_{xx}^3}\\
&=-\frac{(U_x\theta_t+\delta_x)^\top\theta_t}{U_{xx}}-c^\star
=\sigma^\star(t,X_t)^\top\theta_t-c^\star(t,X_t),\qquad t\geq0.
\end{align*}
Consequently,
\begin{equation*}
dX_t^\star
=\bigl(\sigma^\star(t,X_t^\star)^\top\theta_t
-c^\star(t,X_t^\star)\bigr)dt
+\sigma^\star(t,X_t^\star)^\top dW_t,\qquad t\geq0.
\end{equation*}
Therefore, \eqref{eq:Xstar_inverse_marginal} is a strong solution of
\eqref{eq:feedback_wealth_SDE} on $[0,\infty)$. Since $U_x(0,\cdot)$ and $I(t,\cdot)$ are
strictly decreasing and $D_t^\star(\cdot)$ is nondecreasing, the map
$x\mapsto X_t^\star(x)$ is nondecreasing.

\smallskip

\noindent \textit{Step 3. Pathwise uniqueness and the inverse marginal identity.}
From \eqref{eq:Xstar_inverse_marginal} we easily deduce that
\begin{equation}  \label{eq:flow_identity_proof}
U_x\bigl(t,X_t^\star(x)\bigr)
=D_t^\star\bigl(U_x(0,x)\bigr),
\qquad t\geq0.
\end{equation}

Next, we prove uniqueness without assuming any local Lipschitz continuity of
$\sigma^\star$ or $c^\star$. To this end, let $\bar X$ be any strictly positive solution
of \eqref{eq:feedback_wealth_SDE} on the stochastic interval $[0,\zeta)$,
where $\zeta$ is its exit time from $(0,\infty)$. The calculation in the
proof of Lemma~\ref{lem:3-11} is local and therefore applies before global
admissibility has been established. Together with the fact that $\alpha^\star=-F_u$, we
deduce that the process $U_x(t,\bar X_t)$ satisfies, on $[0,\zeta)$,
\begin{equation*}
\begin{aligned}
dU_x(t,\bar X_t)
&=\alpha^\star(t,\bar X_t)U_x(t,\bar X_t)dt\\
&\quad+U_x(t,\bar X_t)
\Bigl(-\theta_t
+\frac{(I_d-\sigma_t^+\sigma_t)\delta_x(t,\bar X_t)}
{U_x(t,\bar X_t)}\Bigr)^\top dW_t,\qquad t\geq0.
\end{aligned}
\end{equation*}
We now define the process $D$ by $D_t=U_x(t,\bar X_t)$, $t\in[0,\zeta)$. The
above equation is then precisely equation \eqref{eq:deflated_dual_SDE}, with
initial value $D_0=U_x(0,x)$. Its coefficients were shown in Step~1 to be
locally Lipschitz and of linear growth, and, therefore, this equation has a
unique strong solution. Since $D^\star(U_x(0,x))$ is, by construction, the
solution of \eqref{eq:deflated_dual_SDE} started at $U_x(0,x)$, we conclude
that
\begin{equation*}
U_x(t,\bar X_t)=D_t^\star(U_x(0,x)),\qquad t\in[0,\zeta),
\end{equation*}
and applying $I(t,\cdot)$ yields $\bar X_t=X_t^\star(x)$ on $[0,\zeta)$.
Since $D^\star(U_x(0,x))$ is continuous and strictly positive and $I$ is
jointly continuous, $X^\star(x)$ remains in a compact subset of $(0,\infty)$
on every finite horizon. Hence $\bar X$ cannot exit $(0,\infty)$ in finite
time, so $\zeta=\infty$ $\mathbb P$-a.s.\ and $\bar X=X^\star(x)$. This
completes the proof of~(1).

\smallskip

\noindent \textit{Step 4. Admissibility in the sense of Definition~\ref{defn:2-1}.}
From \eqref{eq:sigma_star_def}, the contraction property of the orthogonal
projection $\sigma _{t}^{+}\sigma _{t}$, and
\eqref{eq:input_bounds_delta} and \eqref{eq:input_bounds_c}, we obtain
\begin{equation*}
|\sigma^{\pi,\star}_t(X^\star_t)| =\frac{|\sigma^\star(t,X^\star_t)|}{%
X^\star_t} \le \bigl(|\theta_t|+K_t^1\bigr)\, \left|\frac{U_x(t,X^\star_t)}{%
X^\star_t\,U_{xx}(t,X^\star_t)}\right| \le K_t^4\bigl(|\theta_t|+K_t^1\bigr),\qquad t\geq0.
\end{equation*}
Hence,
\begin{equation*}
\int_0^T|\sigma^{\pi,\star}_t(X^\star_t)|^2\,dt<\infty,
\end{equation*}
by the integrability requirement on $K^4$ in Assumption~\ref{asm:3-13}. Furthermore,
\begin{equation*}
\int_0^T c^\star(t,X^\star_t)\,dt
\leq\int_0^T K_t^3X^\star_t\,dt<\infty,
\end{equation*}
since $X^\star$ is continuous on $[0,T]$ and $K^3$ is integrable. It, then,
follows that
$\bigl(\pi^\star(t,X^\star_t),c^\star(t,X^\star_t)\bigr)_{t\geq0}\in
\mathcal A$.

\smallskip

\noindent \textit{Step 5. Consistency.}
Multiplying \eqref{eq:flow_identity_proof} with $%
\kappa^\star$ yields $Y^\star=\kappa^\star \,U_x(\cdot,X^\star)\in \mathcal{Y%
}$ as in Theorem~\ref{thm:3-12}. All hypotheses of Theorem~\ref{thm:2-10} are then satisfied, and we deduce
that $(U,F)$ is a normalized forward recursive utility and that the feedback control pair
is optimal. This proves~(2).

\smallskip

\noindent \textit{Step 6. Dual verification.}
From Proposition~\ref{prop:3-4}, $\widetilde U$ satisfies the dual forward recursive
HJB SPDE. The
conjugacy identities give
\begin{equation*}
\nu_t^{\mathrm{dual}}(y)
=\frac{(I_d-\sigma_t^+\sigma_t)\delta_x(t,I(t,y))}{y},\qquad t\geq0,\ y>0.
\end{equation*}
Hence, the state price density equation controlled by $\nu^{\mathrm{dual}}$
has the same diffusion coefficient as in \eqref{eq:deflated_dual_SDE}, but with
zero drift. The Lipschitz, growth, and positivity arguments of Step~1
then imply a unique strictly positive strong solution for every initial
value $y>0$. Moreover, \eqref{eq:input_bounds_delta} yields
$|\nu_t^{\mathrm{dual}}(y)|\leq K_t^1$, $t\geq0$, from which the
admissibility of $\nu^{\mathrm{dual}}$ follows. Theorem~\ref{thm:3-7}
now gives dual consistency on $\mathcal Y$. This proves~(3).

\smallskip

\noindent \textit{Step 7. The two Fenchel equalities along the optimum.}
Equation \eqref{eq:primal_dual_optimum} holds by definition, and
$Y^{\star }\in \mathcal{Y}$ by Step~5.
Under Assumption~\ref{asm:3-8},
Theorem~\ref{thm:3-12} gives the feedback relation, which is the definition
\eqref{eq:alpha_kappa_optimal}. Equation
\eqref{eq:fenchel_eq_wealth} follows from the fact that
$D_t^\star=U_x(t,X_t^\star)$, $t\geq0$. Equation
\eqref{eq:fenchel_eq_consumption} also holds, since $c_t^\star$ is the
maximizer in the supremum that defines $G$, and \eqref{eq:G_saddle} holds at
$(D_t^\star,\alpha_t^\star)$.

We have $\int_{0}^{T}|\alpha_t^\star|dt
\leq\int_{0}^{T}K_t^0dt<\infty$ by \eqref{eq:input_bounds_alpha}. Thus, if
$\widehat F(t,c,\alpha_t^\star)>-\infty$ for every $c>0$, the discount
$\alpha^\star$ belongs to $\mathcal{A}^{\mathrm{disc}}$, and Step~2 of the
proof of Theorem~\ref{thm:3-12} shows that $Y^{\star }$ admits $\alpha ^{\star }$ as a
feedback discount, so that \eqref{eq:G-integrability} holds along the
optimum. Along the optimal control pair, equality in the
continuation utility variable gives
\begin{equation*}
\widehat F(t,c_t^\star,\alpha_t^\star)
=F(t,c_t^\star,U(t,X_t^\star))
+\alpha_t^\star U(t,X_t^\star),\qquad t\geq0,
\end{equation*}
and hence
\begin{equation*}
\int_0^T\kappa_t^\star
\bigl|\widehat F(t,c_t^\star,\alpha_t^\star)\bigr|dt<\infty
\qquad\text{$\mathbb P$-a.s. for every }T>0.
\end{equation*}
Thus, the triple $\bigl(\alpha ^{\star },(\pi ^{\star },c^{\star
}),Y^{\star }\bigr)$ is integrable in the sense of
\eqref{eq:envelope-integrability}.
Then, Proposition~\ref{prop:3-10} yields the local martingale property, while
Proposition~\ref{prop:3-9} gives \eqref{eq:discounted_fenchel_bound}. The two Fenchel
equalities show that this bound is attained along
$\bigl((\pi^\star,c^\star),Y^\star\bigr)$, and (4) follows.
\end{proof}

\section{Properties of the forward Epstein--Zin recursive aggregator}\label{app:C}

We collect here the analytic properties of the aggregator $f$ introduced in %
\eqref{eq:EZ_aggregator}, which are used in Proposition~\ref{prop:4-3} and in the proof of
Theorem~\ref{thm:4-4}.

\subsection*{Partial derivatives and convexity}

Set
\begin{equation*}
z=(1-\gamma)v,
\qquad
p=1-\frac1\lambda
=\frac{1/\eta-\gamma}{1-\gamma},\qquad v\in\mathcal U_\gamma.
\end{equation*}
Then, \eqref{eq:EZ_aggregator} becomes
\begin{equation*}
f(c,v)
=\frac{1}{1-1/\eta}
\left(c^{1-1/\eta}z^p-z\right),\qquad c>0,\ v\in\mathcal U_\gamma,
\end{equation*}
and, in turn,
\begin{equation*}
f_c(c,v)=c^{-1/\eta}z^p\qquad \text{and}\qquad f_{cc}(c,v)=-\frac1\eta c^{-1/\eta-1}z^p,\qquad c>0,\ v\in\mathcal U_\gamma.
\end{equation*}

For the derivative with respect to $v$, we use
\begin{equation*}
\frac{(1-\gamma)p}{1-1/\eta}
=\lambda-1,
\qquad
\frac{1-\gamma}{1-1/\eta}
=\lambda,
\qquad
p-1=-\frac1\lambda.
\end{equation*}
It follows that
\begin{equation*}
f_u(c,v)
=\frac{1}{1-1/\eta}
\left(
c^{1-1/\eta}p z^{p-1}(1-\gamma)
-(1-\gamma)
\right)=(\lambda-1)c^{1-1/\eta}z^{-1/\lambda}-\lambda,\qquad c>0,\ v\in\mathcal U_\gamma.
\end{equation*}
Differentiating this identity once more with respect to $v$, and using that
\begin{equation*}
(\lambda-1)\left(-\frac1\lambda\right)(1-\gamma)
=\gamma-\frac1\eta,
\end{equation*}
gives
\begin{equation*}
f_{uu}(c,v)
=\left(\gamma-\frac1\eta\right)
c^{1-1/\eta}z^{-(\lambda+1)/\lambda},\qquad c>0,\ v\in\mathcal U_\gamma.
\end{equation*}
Finally, differentiating $f_c(c,v)=c^{-1/\eta}z^p$ with respect to $v$
gives
\begin{equation*}
f_{cu}(c,v)
=c^{-1/\eta}p(1-\gamma)z^{p-1}=\left(\frac1\eta-\gamma\right)
c^{-1/\eta}z^{-1/\lambda},\qquad c>0,\ v\in\mathcal U_\gamma.
\end{equation*}
Thus, for all $c>0$ and $(1-\gamma)v>0$, we have
\begin{align*}
f_c(c,v) &= c^{-1/\eta}\big((1-\gamma)v\big)^{\frac{1/\eta-\gamma}{1-\gamma}}, \\
f_{cc}(c,v) &= -\tfrac{1}{\eta}\,c^{-1/\eta-1} \big((1-\gamma)v\big)^{\frac{%
1/\eta-\gamma}{1-\gamma}}, \\
f_u(c,v) &= (\lambda-1)\,c^{1-1/\eta} \big((1-\gamma)v\big)^{\frac{1/\eta-1}{%
1-\gamma}}-\lambda, \\
f_{uu}(c,v) &= \big(\gamma-\tfrac{1}{\eta}\big)\,c^{1-1/\eta} \big(%
(1-\gamma)v\big)^{-(\lambda+1)/\lambda}, \\
f_{cu}(c,v) &= \big(\tfrac{1}{\eta}-\gamma \big)\,c^{-1/\eta} \big(%
(1-\gamma)v\big)^{-1/\lambda}.
\end{align*}
Because $c>0$ and $z>0$, we deduce
\begin{equation*}
f_c(c,v)>0,
\qquad
f_{cc}(c,v)<0.
\end{equation*}
Moreover,
\begin{equation*}
\operatorname{sign}(f_{cu}(c,v))
=\operatorname{sign}\left(\frac1\eta-\gamma\right),
\end{equation*}
and
\begin{equation*}
\operatorname{sign}(f_{uu}(c,v))
=\operatorname{sign}\left(\gamma-\frac1\eta\right).
\end{equation*}
Hence, $f$ is strictly increasing and strictly concave in $c$. It is convex
in $v$ when $\gamma\eta\geq1$ and concave if
$\gamma\eta\leq1$.

For completeness, we provide the determinant of the Hessian,
\begin{equation*}
f_{cc}(c,v)f_{uu}(c,v)-f_{cu}(c,v)^2
=\gamma\left(\frac1\eta-\gamma\right)c^{-2/\eta}z^{-2/\lambda},\qquad c>0,\ v\in\mathcal U_\gamma.
\end{equation*}
Since $f_{cc}(c,v)<0$, the Hessian is negative semidefinite if and only if
$1/\eta-\gamma\geq0$. Therefore, $f$ is jointly concave if and only if
$\gamma\eta\leq1$.

\end{document}